\documentclass[12pt]{amsart}

\newtheorem{theorem}{Theorem}[section]
\newtheorem{prop}[theorem]{Proposition}
\newtheorem{lemma}[theorem]{Lemma}
\newtheorem{definition}[theorem]{Definition}

\newtheorem{cor}[theorem]{Corollary}
\newtheorem{conjecture}[theorem]{Conjecture} 
\newtheorem{remark}[theorem]{Remark}

\newcommand{\kerr}{\mbox{ker} }

\newcommand{\s}{\mathfrak{s}}
\newcommand{\tk}{\mathfrak{t}}
\newcommand{\rr}{\mathfrak{r}}
\renewcommand{\Dot}{\bullet}
\renewcommand{\P}{{\mathcal P}}
\newcommand{\imm}{\mbox{im} }
\newcommand{\spinc}{{\mbox{spin$^c$} }}
\newcommand{\zee}{\mathbb{Z}}
\newcommand{\x}{{\mathbf{x}}}
\newcommand{\A}{\mathbb{A}}
\newcommand{\B}{\mathcal{B}}

\newcommand{\T}{{\mathcal{T}}}

\newcommand{\ul}{\underline}
\newcommand{\ol}{\overline}
\newcommand{\balpha}{{\mbox{\boldmath $\alpha$}}}
\newcommand{\bbeta}{\mbox{\boldmath $\beta$}}
\newcommand{\bgamma}{\mbox{\boldmath $\gamma$}}
\newcommand{\arr}{\mathbb{R}}
\newcommand{\cee}{\mathbb{C}}

\newcommand{\y}{{\mathbf{y}}}
\newcommand{\z}{{\mathbf{z}}}
\newcommand{\w}{{\mathbf{w}}}
\newcommand{\longto}{\longrightarrow}
\newcommand{\V}{{\mathbf{v}}}
\newcommand{\coker}{{\mbox{coker}}}
\newcommand{\gr}{{\mbox{\rm gr}}}
\newcommand{\D}{\mathfrak{d}}

\newcommand{\Hom}{\mbox{Hom}}
\newcommand{\tX}{\widetilde{X}}

\newcommand{\R}{{\mathcal R}}
\newcommand{\tor}{{\mbox{\rm Tor}}}
\newcommand{\Mhat}{{\widehat{\mathcal M}}}
\newcommand{\OS}{{Ozsv\'ath-Szab\'o}}

\newcommand{\ugr}{\underline{\mbox{\rm gr}}}

\renewcommand{\O}{{\mbox{\normalfont\gothfamily O}}}
\newcommand{\supp}{{\mbox{\rm supp}}}
\newcommand{\Z}{{\mathcal Z}}
\newcommand{\rk}{{\mbox{\rm rk}}}
\newcommand{\K}{{\mathcal K}}
\newcommand{\M}{{\mathcal M}}
\newcommand{\N}{{\mathcal N}}

\newcommand{\Mbar}{\overline{M}}
\renewcommand{\L}{{\mathcal L}}
\newcommand{\tO}{\widetilde{\O}}
\newcommand{\tbeta}{\widetilde{\beta}}
\newcommand{\tgamma}{\widetilde{\gamma}}
\newcommand{\ds}{\displaystyle}

\newcommand{\rs}{\mathfrak{r}}
\newcommand{\uf}{\underline{F}}
\newcommand{\ug}{\underline{G}}
\newcommand{\uh}{\underline{H}}

\input xy
\xyoption{all}

\usepackage{amsfonts,yfonts}
\usepackage{epsf}
\usepackage{graphicx}
\usepackage{amssymb}
\usepackage{amsmath}
\usepackage{amscd}
\usepackage{diagrams}
\newarrow{Dashline} {}{dash}{}{dash}{}

\begin{document}

\title{Product Formulae for Ozsv\'ath-Szab\'o 4-manifold Invariants}

\author{Stanislav Jabuka}

\author{Thomas E. Mark}

\maketitle


\begin{abstract}We give formulae for the {\OS} invariants of 4-manifolds $X$ obtained by fiber sum of two manifolds $M_1$, $M_2$ along surfaces $\Sigma_1$, $\Sigma_2$ having trivial normal bundle and genus $g\geq 1$. The formulae follow from a general theorem on the {\OS} invariants of the result of gluing two 4-manifolds along a common boundary, which is phrased in terms of relative invariants of the pieces. These relative invariants take values in a version of Heegaard Floer homology with coefficients in modules over certain Novikov rings; the fiber sum formula follows from the theorem that this ``perturbed'' version of Heegaard Floer theory recovers the usual {\OS} invariants, when the 4-manifold in question has $b^+\geq 2$. The construction allows an extension of the definition of {\OS} invariants to 4-manifolds having $b^+ = 1$ depending on certain choices, in close analogy with Seiberg-Witten theory.  The product formulae lead quickly to calculations of the {\OS} invariants of various 4-manifolds; in all cases the results are in accord with the conjectured equivalence between {\OS} and Seiberg-Witten invariants.
\end{abstract}

\section{Introduction}

At the time of writing, there is no example of a smoothable topological 4-manifold whose smooth structures have been classified. Indeed, no smooth 4-manifold is known to support only finitely many smooth structures, and in virtually every case a 4-manifold that admits more than one smooth structure is known to admit infinitely many such structures. A substantial amount of ingenuity by a large number of authors---see \cite{stern} for a brief survey---has been required to produce these exotic 4-manifolds, though ultimately the list of topological tools used in the constructions is perhaps surprisingly short. The standard approach to distinguishing smooth structures on 4-manifolds has been to make use of gauge-theoretic invariants, which requires an understanding of how these invariants behave under the cut-and-paste operations used in constructing examples. In the case of the Seiberg-Witten invariants, this understanding was provided by Morgan-Mrowka-Szab\'o \cite{MMS}, Morgan-Szab\'o-Taubes \cite{MST}, D. Park \cite{park}, Li-Liu \cite{LL}, and many others, and the Seiberg-Witten invariants have become the tool of choice for studying smooth manifolds. Beginning in 2000, Ozsv\'ath and Szab\'o \cite{OS1,OS2,OS3} introduced invariants of 3- and 4-dimensional manifolds meant to mimic the Seiberg-Witten invariants but also avoid the technical issues that for many years prevented the expected Seiberg-Witten-Floer theory from taking shape. Their theory has been remarkably successful, and has had a number of important consequences in the study of 3-manifolds and knot theory. The 4-dimensional side of the story has been developed to a  somewhat lesser extent, however, and the existing gauge-theoretic technology means it is still the case that Seiberg-Witten invariants are often the easiest to use in the study of smooth 4-manifolds. Our aim here is to develop reasonably general cut-and-paste principles for {\OS} invariants, that will be useful in a variety of situations. A central tool in many constructions of exotic 4-manifolds is the normal connected sum or ``fiber sum,'' in which neighborhoods of diffeomorphic surfaces are excised from closed 4-manifolds and the resulting complements glued together along their boundary. As an application of the formalism we introduce here, and as a motivating test case, we give formulae that essentially determine the behavior of the {\OS} 4-manifold invariants under fiber sum along surfaces of trivial normal bundle.

To realize this goal we are obliged to introduce a substantial amount of machinery, including the development of Heegaard Floer homology with coefficients in certain power series (Novikov) rings. This can be viewed in analogy with Seiberg-Witten Floer homology perturbed by a 2-dimensional cohomology class, and in many ways exhibits parallel behavior. It is our hope that this ``perturbed'' Heegaard Floer theory will be of interest in other applications as well.

For the sake of exposition, we state our results in this introduction in order of increasing technicality. In particular, Ozsv\'ath and Szab\'o defined their invariants initially for 4-manifolds $M$ with $b^+(M)\geq 2$, and since the theory is simplest in that case we begin there.

\subsection{Constructions and statements of results when $b^+ \geq 2$}
The {\OS} invariants \cite{OS3, OS4} are defined using a ``TQFT'' construction, meaning that they are built from invariants of 3-dimensional manifolds (the Heegaard Floer homology groups) and cobordisms between such manifolds. To a closed oriented 4-manifold $M$ with $b^+(M)\geq 2$, with a \spinc structure $\s$, Ozsv\'ath and Szab\'o associate a linear function
\[
\Phi_{M,\s}: \A(M) \to \zee/\pm 1,
\]
where $\A(M)$ is the algebra $\Lambda^*(H_1(M;\zee)/torsion)\otimes \zee[U]$, graded such that elements of $H_1(M)$ have degree 1 and $U$ has degree 2. This invariant has the property that $\Phi_{M,\s}$ is nonzero for at most finitely many \spinc structures $\s$, and furthermore can be nonzero only on homogeneous elements of $\A(M)$ having degree
\begin{equation}\label{ddef}
d(\s) = {\textstyle \frac{1}{4}}(c_1^2(\s) - 3\sigma(M) - 2e(M)),
\end{equation}
where $\sigma$ denotes the signature of the intersection form on $M$ and $e$ is the Euler characteristic. Ozsv\'ath and Szab\'o conjecture \cite{OS3} that $\Phi_{M,\s}$ is identical with the Seiberg-Witten invariant.

We remark that there is a sign ambiguity in the definition of $\Phi_{M,\s}$, so that the results to follow are true up to an overall sign. 

The fiber sum of two smooth 4-manifolds is defined as follows. Let $M_1$ and $M_2$ be closed oriented 4-manifolds, and suppose $\Sigma_i\hookrightarrow M_i$, $i = 1,2$, are smoothly embedded closed oriented surfaces of the same genus $g$. We assume throughout this paper that $g$ is at least 1 and that the $\Sigma_i$ have trivial normal bundles. In this case, $\Sigma_i$ has a neighborhood $N(\Sigma_i)$ diffeomorphic to $\Sigma_i\times D^2$. Choose an orientation-preserving diffeomorphism $f: \Sigma_1\to \Sigma_2$, and lift it to an orientation-reversing diffeomorphism $\phi: \partial N(\Sigma_1)\to \partial N(\Sigma_2)$ via conjugation in the normal fiber. We define the fiber sum $X = M_1\#_\Sigma M_2$ by 
\[
X = (M_1\setminus N(\Sigma_1)) \cup_{\phi} (M_2\setminus N(\Sigma_2)).
\]
In general, the manifold $X$ can depend on the choice of $\phi$. We assume henceforth that the homology classes $[\Sigma_1]$ and $[\Sigma_2]$ are non-torsion elements of $H_2(M_i;\zee)$ (though the results of this paper can in principle be adapted to other situations).

To state the results, it is convenient to express the {\OS} invariant in terms of the group ring $\zee[H^2(M;\zee)]$. That is to say, we write
\[
OS_{M} = \sum_{\s\in Spin^c(M)} \Phi_{M,\s} \,e^{c_1(\s)},
\]
where $e^{c_1(\s)}$ is the formal variable in the group ring corresponding to the first Chern class of the \spinc structure $\s$ (note that $c_1(\s) = c_1(\s')$ for distinct \spinc structures $\s$ and $\s'$ iff $\s - \s'$ is of order 2 in $H^2(M;\zee)$, so the above formulation may lose some information if 2-torsion is present). The coefficients of the above expression are functions on $\A(M)$, so that $OS_{M}$ is an element of $\zee[H^2(M;\zee)]\otimes \A(M)^*$. The value of the invariant on $\alpha\in \A(M)$ is denoted $OS_M(\alpha) \in \zee[H^2(M;\zee)]$.

The behavior of $\Phi_{M,\s}$ under fiber sum depends on the value of $\langle c_1(\s), [\Sigma]\rangle$ (since $c_1(\s)$ is a characteristic class, this value is always even when $[\Sigma]^2 = 0$). Thus, we partition $OS_{M}$ accordingly: for an embedded surface $\Sigma\hookrightarrow M$ with trivial normal bundle,  let
\[
OS_M^k = \sum_{\langle c_1(\s),[\Sigma]\rangle = 2k} \Phi_{M,\s}\,e^{c_1(\s)}.
\]
The adjunction inequality for {\OS} invariants implies that $OS_M^k \equiv 0$ if $|k|>g-1$.

The topology of fiber sums is complicated in general by the presence of {\it rim tori}. A rim torus is a submanifold of the form $\gamma\times S^1\subset \Sigma\times S^1$, where $\gamma$ is an embedded circle on $\Sigma$. Such tori are homologically trivial in the fiber summands $M_i$, but typically essential in $X = M_1\#_\Sigma M_2$. Let $\R$ denote the subspace of $H^2(X;\zee)$ spanned by the Poincar\'e duals of rim tori, and let $\rho: H^2(X;\zee)\to H^2(X;\zee)/\R$ denote the natural projection. If $b_i \in H^2(M_i; \zee)$, $i = 1,2$, are cohomology classes with the property that $b_1|_{\partial N(\Sigma_1)}$ agrees with $b_2|_{\partial N(\Sigma_2)}$ under $\phi$, then Mayer-Vietoris arguments show that there exists a class $b\in H^2(X; \zee)$ whose restrictions to $M_i\setminus N(\Sigma_i)$ agrees with the corresponding restrictions of $b_i$, and furthermore that $b$ is determined uniquely up to elements of $\R$ and multiples of the Poincar\'e dual of $\Sigma$. If $b$, $b_1$ and $b_2$ satisfy these conditions on their respective restrictions, we say that the three classes are {\it compatible} with the fiber sum. We can eliminate part of the ambiguity in $b$ given $(b_1,b_2)$ by requiring that 
\begin{equation}\label{squarerestr}
b^2 = b_1^2 + b_2^2 + 4|m|,
\end{equation}
where $m = \langle b_1,[\Sigma_1]\rangle = \langle b_2,[\Sigma_2]\rangle$. With this convention, the pair $(b_1,b_2)$ gives rise to a well-defined element of $H^2(X;\zee)/\R$ (see section \ref{proofsec} for details). 

\begin{theorem}\label{gengthm1} Let $X = M_1\#_\Sigma M_2$ be obtained by fiber sum along a surface $\Sigma$ of genus $g>1$ from manifolds $M_1$, $M_2$ satisfying $b^+(M_i)\geq 2$, $i = 1,2$. If $|k| > g-1$ then $OS^k_X = OS^k_{M_1} = OS^k_{M_2} = 0$. In general, we have 
\begin{equation}\label{gengformula}
\rho\left( OS^k_X(\alpha)\right) = \sum_{\beta\in \B_k} OS_{M_1}^k(\alpha_1\otimes \beta)\cdot OS_{M_2}^k(f_*(\beta^\circ)\otimes \alpha_2)\cdot u_{\beta,k}
\end{equation}
where $\alpha_i\in \A(M_i\setminus N(\Sigma_i))$ are any elements such that $\alpha_1\otimes \alpha_2$ maps to $\alpha$ under the inclusion-induced homomorphism.
\end{theorem}

The notation of the theorem requires some explanation. First, the product of group ring elements appearing on the right makes use of the construction outlined above, producing elements of $H^2(X;\zee)/\R$ from compatible pairs $(b_1,b_2)$. The set $\B_k$ denotes a basis over $\zee$ for the group $H_*(\mbox{Sym}^d(\Sigma);\zee) \cong \bigoplus_{i = 0}^d \Lambda^iH_1(\Sigma)\otimes \zee[U]/U^{d-i+1}$, thought of as a subgroup of $\A(\Sigma)$, where $d = g-1-|k|$. Likewise, $\{\beta^\circ\}$ denotes the dual basis to $\B_k$ under a certain nondegenerate pairing (see section \ref{proofsec}). The terms $\alpha_1\otimes  \beta$ and $f_*(\beta^\circ)\otimes \alpha_2$ are understood to mean the images of those elements in $\A(M_1)$ and $\A(M_2)$, using the inclusion-induced maps. Finally, $u_{\beta,k}$ is a polynomial in the variable $e^{PD[\Sigma]}$ whose constant coefficient is 1, and which is equal to 1 except possibly in the case $k = 0$.

The left hand side of \eqref{gengformula} lies in the group ring of $H^2(X;\zee)/\R$, and its coefficients are ``rim torus averaged'' {\OS} invariants. That is to say, each coefficient of $\rho(OS_X^k)$ is a term of the form
\[
\Phi^{Rim}_{X,\s} = \sum_{\s'\in Spin^c(X)\atop \s'-\s\in \R} \Phi_{X,\s'}.
\]

A 4-manifold $X$ is said to have (Ozsv\'ath-Szab\'o) {\it simple type} if any \spinc structure $\s$ for which $\Phi_{X,\s}\neq 0$ has $d(\s) = 0$. We have:

\begin{cor}\label{simptypecor} If $M_1$ and $M_2$ have simple type, then the fiber sum $X = M_1\#_\Sigma M_2$ has the property that if $\Phi_{X,\s}^{Rim} = 0$ whenever $d(\s) \neq 0$. Furthermore,
\begin{equation}\label{intermedvanish}
\rho\left(OS_X^k\right) = 0 \qquad \mbox{if $|k| < g-1$,}
\end{equation}
while
\[
\rho\left(OS_X^{g-1}(\alpha)\right) = \left\{\begin{array}{ll} OS^{g-1}_{M_1}(1)\cdot OS^{g-1}_{M_2}(1) & \mbox{if $\alpha = 1$} \\ 0 & \mbox{if $\deg(\alpha) > 1$}.\end{array}\right.
\]
\label{STcor}
\end{cor}

In other words, the fiber sum of manifolds of simple type has simple type after sum over rim tori.
We note that equation \eqref{intermedvanish} holds if $M_1$ and $M_2$ are assumed only to have $\A(\Sigma)$-simple type: that is, if $\Phi_{M,\s}(\alpha) = 0$ whenever $\alpha$ lies in the ideal of $\A(M)$ generated by $U$ and the image of $H_1(\Sigma)$.

We should remark that Taubes \cite{taubesST} has shown that symplectic 4-manifolds with $b^+\geq 2$ have Seiberg-Witten simple type. It seems safe, therefore, to make the following:

\begin{conjecture} If $X$ is a symplectic 4-manifold with $b^+(X)\geq 2$ then $X$ has {\OS} simple type.
\end{conjecture}

Leaving this issue for now, we turn to the case of a fiber sum along a torus, where the product formula is slightly different.

\begin{theorem}\label{g1thm} Let $X = M_1\#_\Sigma M_2$ be obtained by fiber sum along a surface $\Sigma$ of genus $g = 1$, such that $M_1$, $M_2$, and $X$ each have $b^+ \geq 2$. Let $\widetilde{T}$ denote the Poincar\'e dual of the class in $H_2(X;\zee)$ induced by $[\Sigma_i]$, and write $T$ for the image of $\widetilde{T}$ in $H^2(X;\zee)/\R$. Then for any $\alpha\in \A(X)$ we have
\[
\rho(OS_X(\alpha)) = (T-T^{-1})^2\,OS_{M_1}(\alpha_1)\cdot OS_{M_2}(\alpha_2)
\]
where $\alpha_1\otimes \alpha_2\in\A(M_1)\otimes \A(M_2)$ maps to $\alpha$ as before.
\end{theorem}
Here the product between $OS_{M_1}$ and $OS_{M_2}$ uses the construction from previously, while multiplication with $T$ takes place in the group ring of $H^2(X;\zee)/\R$.

We will show (Proposition \ref{torusSTprop}) that any 4-manifold $M$ containing an essential torus $T$ of self-intersection 0 has $\A(T)$-simple type, in analogy with a result of Morgan, Mrowka, and Szab\'o in Seiberg-Witten theory \cite{MMS}.

It is interesting to compare these results with those in Seiberg-Witten theory. Taubes proved an analogue of Theorem \ref{g1thm} in \cite{taubes}, generalizing work of Morgan-Mrowka-Szab\'o \cite{MMS}, and D. Park \cite{park} gave an independent proof of that result. The higher-genus case was considered by Morgan, Szab\'o and Taubes \cite{MST}, but only under the condition that $|k| = g-1$. In this case the sum appearing in Theorem \ref{gengthm1} is trivial since $\B_{g-1} = \{1\}$, and the result here gives a product formula directly analogous to that of \cite{MST}. To our knowledge, no product formulae at the level of generality of Theorem \ref{gengthm1} have yet appeared in the literature on Seiberg-Witten theory.

\subsection{Relative invariants and a general gluing result}
The theorems above are proved as particular cases of a general result on the {\OS} invariants of 4-manifolds obtained by gluing two manifolds along their boundary. In its most general form, the form that is useful in the context of fiber sums (Theorem \ref{intropertproductthm} below), the statement involves perturbed Heegaard Floer invariants. If one is interested in gluing two manifolds-with-boundary that both have $b^+ \geq 1$, however, the perturbed theory is unnecessary and there is a slightly simpler ``intermediate'' result. To state it, recall that the construction of the 4-manifold invariant $\Phi_{M,\s}$ is based on the Heegaard Floer homology groups associated to closed \spinc 3-manifolds $(Y,\s)$. These groups have various incarnations; the relevant one for our immediate purpose is denoted $HF^-_{red}(Y,\s)$. Below, we recall the construction of Heegaard Floer homology with ``twisted'' coefficients, whereby homology groups are obtained whose coefficients are modules $M$ over the group ring $R_Y = \zee[H^1(Y)]$ (here and below, ordinary (co)homology is considered with integer coefficients). If $Y = \partial Z$ is the boundary of an oriented 4-manifold $Z$, then such a module is provided by
\[
M_Z = \zee[\ker(H^2(Z,\partial Z)\to H^2(Z))],
\]
where $H^1(Y)$ acts by the coboundary homomorphism $H^1(Y )\to H^2(Z,\partial Z )$. The intermediate product formula alluded to above can be formulated as follows.

\begin{theorem}\label{genproduct} If $(Z,\s)$ is a \spinc 4-manifold with connected \spinc boundary $(Y, \s_Y)$ and if $b^+(Z)\geq 1$, then there exists a relative {\OS} invariant $\Psi_{Z,\s}$ which is a linear function
\[
\Psi_{Z,\s}: \A(Z)\to HF^-_{red}(Y,\s_{Y}; M_Z),
\]
well-defined up to multiplication by a unit in $\zee[H^1(Y )]$. 

Furthermore, if $(Z_1,\s_1)$ and $(Z_2,\s_2)$ are two such \spinc 4-manifolds with \spinc boundary $\partial Z_1 = (Y,\s) = -\partial Z_2$, write $X = Z_1\cup_Y Z_2$. Then there exists an $R_Y$-sesquilinear pairing
\[
(\cdot\,,\,\cdot) : HF^-_{red}(Y,\s;M_{Z_1})\otimes_{R_Y} HF^-_{red}(-Y, \s;M_{Z_2})\to M_{X,Y},
\]
where $M_{X,Y} = \zee[K(X,Y)]$ and $K(X,Y) =\ker(H^2(X )\to H^2(Z_1 )\oplus H^2(Z_2 ))$. The pairing has
the property that for any \spinc structure $\s$ on $X$ restricting to $\s_i$ on $Z_i$, we have an equality of group ring elements: 
\[
\sum_{h\in K(X,Y)} \Phi_{X,\s+h}(\alpha)\, e^h = ( \Psi_{Z_1,\s_1}(\alpha_1),\,\Psi_{Z_2,\s_2}(\alpha_2)),
\]
up to multiplication by a unit in $\zee[K(X,Y)]$. Here $\alpha\in \A(X)$, $\alpha_1\in \A(Z_1)$ and $\alpha_2\in\A(Z_2)$ are related by inclusion-induced multiplication as before.
\end{theorem}

To understand the term ``$R_Y$-sesquilinear,'' observe that $R_Y = \zee[H^1(Y)]$ is equipped with an involution $r\mapsto \bar{r}$ induced by $h\mapsto -h$ in $H^1(Y)$. To say that the pairing in the theorem is sequilinear means that 
\[
( g\xi, \,\eta ) = g(\xi,\,\eta) = ( \xi,\,\bar{g}\eta)
\]
for $g\in R_Y$, $\xi\in HF^-_{red}(Y,\s;M_{Z_1})$ and $\eta\in HF^-_{red}(-Y,\s;M_{Z_2})$.

We note that the reason for the assumption $b^+(Z)\geq 1$ in the theorem above is that this condition guarantees that the homomorphism in $HF^-$ induced by $Z\setminus B^4$ (which gives rise to the relative invariant $\Psi_{Z,\s}$ above) takes values in the reduced Floer homology $HF^-_{red}(Y,\s; M_Z)\subset HF^-(Y,\s;M_Z)$. That fact in turn is necessary to make sense of the pairing $(\cdot,\cdot)$. In the notation of later sections, $(\cdot, \cdot) = \langle \tau^{-1}(\cdot),\cdot\rangle$ where $\tau: HF^+\to HF^-$ is the natural map; $\tau$ is invertible only on the reduced groups.

The utility of Theorem \ref{genproduct} is limited somewhat by the difficulty of determining the relative invariants $\Psi_{Z,\s}$ in general. Furthermore, in the case of a fiber sum it is natural to hope to relate the relative invariants of the complement of the neighborhood $\Sigma\times D^2$ of the summing surface in $M$ to the absolute invariants of $M$; however the manifold $\Sigma\times D^2$ has $b^+ = 0$ and it is not clear that the relative invariant is well-defined. This issue is addressed by the introduction of a ``perturbation.''

\subsection{Perturbed Heegaard Floer theory and results when $b^+ \geq 1$} Let $Y$ be a closed oriented 3-manifold and $\eta\in H^2(Y;\arr)$ a given cohomology class. The {\it Novikov ring} associated to $\eta$ is the set of formal series
\[
\R_{Y,\eta} = \{ \sum_{g\in H^1(Y;\zee)} a_g \cdot g \,|\, a_g\in \zee\}\subset \zee[[H^1(Y;\zee)]]
\]
subject to the condition that for each $N\in \zee$, the set of $g\in H^1(Y;\zee)$ with $a_g$ nonzero and $\langle g\cup \eta, [Y]\rangle < N$ is finite. This means $\R_{Y,\eta}$ consists of ``semi-infinite'' series with variables in $H^1(Y;\zee)$, with the usual convolution product. 

In section \ref{perturbsec} below, we develop the theory of Heegaard Floer homology for 3-manifolds $Y$ and 4-dimensional cobordisms $W$ equipped with 2-dimensional cohomology classes $\eta$, having coefficients in a module $\M_\eta$ over $\R_{Y,\eta}$. We refer to this theory as Heegaard Floer homology perturbed by $\eta$. Many features of the unperturbed theory carry over to this setting with minimal modification, but one key simplification is that if $\eta$ is chosen ``generically'' in a suitable sense (in particular $\eta\neq 0$), then $HF^\infty(Y,\s;\M_\eta) = 0$ for any $\R_{Y,\eta}$-module $\M_\eta$. In fact, one can arrange this latter fact to hold for any nonzero perturbation $\eta$ by a further extension of coefficients: Heegaard Floer homology is naturally a module over a polynomial ring $\zee[U]$, and we form a ``$U$-completion'' by extension to the power series ring $\zee[[U]]$. The $U$-completed Floer homology is written $HF_\Dot(Y,\s;\M_\eta)$ by notational analogy with a similar construction in monopole Floer homology \cite{KMbook}. The vanishing of $HF^\infty_\Dot(Y,\s;\M_\eta)$ means that $HF^-_\Dot(Y,\s; \M_\eta) = HF^-_{red}(Y,\s;\M_\eta)$ for all such $\M_\eta$, and allows us to define a relative invariant
\[
\Psi_{Z,\s,\eta}\in HF^-_{red}(Y,\s; \M_{Z,\eta})
\]
that has the desired properties so long as $\eta|_Y \neq 0$. Note, however, that $\Psi_{Z,\s,\eta}$ is defined only up to sign and multiplication by an element of $H^1(Y)$. We remark that if $\eta|_Y = 0$ then $\R_{Y,\eta} = R_Y$, and we recover the unperturbed theory.

Now suppose that $X$ is a closed 4-manifold, $Y\subset X$ a separating submanifold, and $\eta\in H^2(X;\arr)$ a cohomology class such that either $\eta|_Y \neq 0$, or in the decomposition $X = Z_1\cup_Y Z_2$ we have $b^+(Z_i)\geq 1$. (Such a submanifold $Y$ is said to be an {\it allowable cut for $\eta$}.) Then we can define the {\it perturbed {\OS} invariant} associated to $X$, $Y$, $\eta$, and a \spinc structure $\s$ to be
\begin{equation}\label{Odef}
\O_{X,Y,\eta,\s}(\alpha) = \langle \tau^{-1}(\Psi_{Z_1,\eta,\s}(\alpha_1)), \Psi_{Z_2,\eta,\s}(\alpha_2)\rangle.
\end{equation}
This invariant takes values in a module $\M_{X,Y,\eta}$, which is a suitable Novikov completion of $M_{X,Y}$ introduced previously. In section \ref{perturbsec} we show (Theorem \ref{pertproductthm}) that if $b^+(X)\geq 2$ then $\O_{X,Y,\eta,\s}$ is in fact a polynomial lying in $M_{X,Y}$, whose coefficients are the {\OS} invariants of $X$ in the various \spinc structures having restrictions to $Z_1$ and $Z_2$ that agree with the restrictions of $\s$. The precise statement is the following:

\begin{theorem}\label{intropertproductthm} Let $X$ be a closed oriented 4-manifold with $b^+(X)\geq 2$, and $Y\subset X$ a submanifold determining a decomposition $X = Z_1\cup_Y Z_2$, where $Z_i$ are connected 4-manifolds with boundary. Fix a class $\eta\in H^2(X;\arr)$, and assume that $Y$ is an allowable cut for $\eta$. If $b^+(Z_1)$ and $b^+(Z_2)$ are not both 0, then for any \spinc structure $\s$ on $X$ and element $\alpha\in \A(X)$,
\begin{equation}\label{intropertproductform}
\sum_{t\in K(X,Y)}\Phi_{X,\s + t}(\alpha) e^{t} = \O_{Y,\eta,\s}(\alpha) = \langle \tau^{-1}\Psi_{Z_1,\eta,\s}(\alpha_1),\,\Psi_{Z_2,\eta,\s}(\alpha_2)\rangle
\end{equation}
up to sign and multiplication by an element of $K(X,Y)$, where $\alpha_1\otimes \alpha_2\mapsto \alpha$ as before. If $b^+(Z_1) = b^+(Z_2) = 0$ then the same is true after possibly replacing $\eta$ by another class $\tilde{\eta}$, where $\tilde{\eta}|_{Z_i} = \eta|_{Z_i}$ for $i = 1,2$.
\end{theorem}

The above definition \eqref{Odef} of $\O_{X,Y,\eta,\s}$ makes sense for any allowable pair $(Y,\eta)$ and \spinc structure $\s$, but its dependence on the choice of $(Y,\eta)$ is not clear. When $b^+(X)\geq 2$ it follows from Theorem \ref{intropertproductthm} that since $\Phi_{X,\s}$ is independent of $Y$ and $\eta$, so is $\O_{X,Y,\eta,\s}$. However when $b^+(X) = 1$ the situation is not so simple; indeed, different choices of $(Y,\eta)$ for a given $(X,\s)$ can lead to different results. This situation is analogous to the chamber structure of Seiberg-Witten invariants for 4-manifolds with $b^+ = 1$; partial results in this direction are given in section \ref{perturbsec}.

Note that the existence of a separating 3-manifold $Y\subset X$ and a class $\eta\in H^2(X,\arr)$ restricting nontrivially to $Y$ implies that $X$ is indefinite, in particular $b^+(X)\geq 1$.

We also point out a minor difference between Theorems \ref{gengthm1} and \ref{g1thm} from the first section, and Theorem \ref{intropertproductthm} above and Theorems \ref{introg1fibsumthm} and \ref{introgengfibsumthm} below. In the former results, the various \spinc structures are labeled by their Chern classes, while in the latter they are identified in an affine way with two-dimensional cohomology classes. Thus the results in the present situation do not lose information corresponding to classes whose difference is of order 2, and to translate from results in this subsection to those in the first one we must square the variables.

An immediate consequence of Theorem \ref{intropertproductthm} is the following result on the {\OS} invariants of a manifold obtained by gluing two 4-manifolds along a boundary 3-torus. To state it, note first that if $Z$ is a 4-manifold with boundary diffeomorphic to $T^3$ and $\eta\in H^2(Z;\arr)$ is a class whose restriction to $T^3$ is nontrivial then the relative invariant $\Psi_{Z,\s,\eta}$ is well-defined, and takes values in the ring $\K(Z,\eta) \subset \zee[[K(Z)]]$, where $K(Z) = \ker(H^2(Z,\partial Z) \to H^2(Z))$ and $\K(Z,\eta)$ is a Novikov completion of the group ring $\zee[K(Z)]$. (If $b^+(Z)\geq 1$ then $\Psi_{Z,\s,\eta}$ lies in $\zee[K(Z)]$.) Indeed, $\K(Z,\eta)$ is precisely the perturbed Floer homology of $T^3$ in the appropriate coefficient system. Note that $\K(Z,\eta)$ can be identified with a multivariable Laurent series ring, which is polynomial in variables that pair trivially with $\eta$ (and some variables may have finite order, if there is torsion in the cokernel of $H^1(Z)\to H^1(\partial Z)$).

If $X = Z_1\cup Z_2$ is obtained by gluing two 4-manifolds $Z_1$ and $Z_2$ with boundary $T^3$, and $\eta\in H^2(X;\arr)$ restricts nontrivially to the splitting 3-torus, then the pairing appearing in \eqref{intropertproductform} is naturally identified with a multiplication map
\[
\K(Z_1,\eta)\otimes \K(Z_2,\eta)\rTo^\sim \M_{X,T^3,\eta}\subset \zee[[H^2(X;\zee)]]
\]
induced by the maps $j^*_i: H^2(Z_i , \partial Z_i)\to H^2(X)$ Poincar\'e dual to the inclusion homomorphisms. Thus Theorem \ref{intropertproductthm} gives:

\begin{cor} Let $X = Z_1\cup_\partial Z_2$ be a 4-manifold obtained as the union of two manifolds $Z_1$ and $Z_2$ whose boundary is diffeomorphic to the 3-torus $T^3$, $\eta\in H^2(X;\arr)$ a class restricting nontrivially to $T^3$, and $\s$ a \spinc structure on $X$. Then
\[
\O_{X,T^3,\s,\eta} = j_1^*(\Psi_{Z_1,\eta,\s})\,j_2^*(\Psi_{Z_2,\eta,\s}).
\]
In particular if $b^+(X)\geq 2$ then
\[
\sum_{k\in \delta H^1(T^3)} \Phi_{X,\s + k} \,e^k = j_1^*(\Psi_{Z_1,\eta,\s})\,j_2^*(\Psi_{Z_2,\eta,\s})
\]
up to sign and translation by an element of $\delta H^1(T^3)$, where $\delta:H^1(T^3)\to H^2(X)$ is the Mayer-Vietoris coboundary.
\end{cor}

We deduce the fiber sum formulae in Theorems \ref{gengthm1} and \ref{g1thm} from the following somewhat more general results, which apply in particular to the situation in which $M_1$, $M_2$, and/or $X$ have $b^+ = 1$. In each case, the perturbed invariants $\O_{M_i,\Sigma\times S^1}$ take values in $\M_{M_i,\Sigma\times S^1,\eta}$, which is isomorphic to the ring $\L(t)$ of Laurent series in the variable $t$ corresponding to the Poincar\'e dual of the surface $\Sigma$. Each of the following is obtained by an application of \eqref{Odef}, combined with knowledge of the relative invariants of manifolds of the form $\Sigma\times D^2$. In particular, Theorem \ref{introg1fibsumthm} follows quickly from the fact that up to multiplication by $\pm t^n$,
\[
\Psi_{T^2\times D^2,\eta,\s} = \frac{1}{t-1}
\]
where $\s$ is the \spinc structure with trivial first Chern class and $\eta\in H^2(T^2\times D^2;\arr)$ has $\int_{T^2} \eta >0$ (Proposition \ref{g1relinvtprop}). Note that this implies that the complement $Z$ of a torus of square 0 in a closed 4-manifold $M$ has relative invariant satisfying
\[
\rho(\Psi_{Z,\eta,\s}) = (t-1) \,\O_{M,T^3,\eta,\s},
\]
where $\eta$ is a class as above, and $\rho$ is induced by the map $\zee[[H^2(Z;\partial Z)]]\to \L(t)$ setting all variables other than $t$ equal to 1. 

\begin{theorem}\label{introg1fibsumthm} Let $X = M_1\#_{T_1 = T_2} M_2$ be the fiber sum of two 4-manifolds $M_1$, $M_2$ along tori $T_1$, $T_2$ of square 0. Assume that there exist classes $\eta_i\in H^2(M_i;\arr)$, $i = 1,2$, such that the restrictions of $\eta_i$ to $T_i\times S^1\subset M_i$ correspond under the gluing diffeomorphism $f: T_1\times S^1\to T_2\times S^1$, and assume that $\int_{T_i}\eta_i > 0$. Let $\eta\in H^2(X;\arr)$ be a class whose restrictions to $Z_i = M_i\setminus(T_i\times D^2)$ agree with those of $\eta_i$, and choose \spinc structures $\s_i\in Spin^c(M_i)$, $\s\in Spin^c(X)$ whose restrictions correspond similarly. Then for any $\alpha\in \A(X)$, the image of $\alpha_1\otimes \alpha_2$ under the map $\A(Z_1)\otimes \A(Z_2)\to \A(X)$, we have
\[
\rho(\O_{X,T\times S^1,\eta,\s}(\alpha)) = (t^{1/2} - t^{-1/2})^2\,\O_{M_1,T_1\times S^1,\eta_1,\s_1}(\alpha_1)\cdot\O_{M_2,T_2\times S^1,\eta_2,\s_2}(\alpha_2)
\]
up to multiplication by $\pm t^n$. 
\end{theorem}

In the higher-genus case we have the following.

\begin{theorem}\label{introgengfibsumthm} Let $X = M_1\#_{\Sigma_1= \Sigma_2} M_2$ be the fiber sum of two 4-manifolds $M_1$, $M_2$ along surfaces $\Sigma_1$, $\Sigma_2$ of genus $g>1$ and square 0. Let $\eta_1$, $\eta_2$, $\eta$ be 2-dimensional cohomology classes satisfying conditions analogous to those in the previous theorem, and choose \spinc structures $\s_1$, $\s_2$, and $\s$ restricting compatibly as before. If the Chern classes of each \spinc structure restrict to $\Sigma\times S^1$ as a class other than $2k\,PD[S^1]$ with $|k|\leq g-1$ then the {\OS} invariants of all manifolds involved vanish. Otherwise, writing $f$ for the gluing map $\Sigma_1\times S^1\to \Sigma_2\times S^1$, we have
\[
\rho(\O_{X,\Sigma\times S^1,\eta,\s}(\alpha)) = \sum_{\beta} \O_{M_1,\Sigma_1\times S^1,\eta_1,\s_1}(\alpha_1\otimes \beta)\cdot \O_{M_2,\Sigma_2\times S^1,\eta_2,\s_2}(\alpha_2\otimes f_*(\beta^\circ))\cdot u_{\beta,k}
\]
up to multiplication by $\pm t^n$.
\end{theorem}

In this theorem, $\{\beta\}$ is a basis for $H_*(\mbox{Sym}^d \Sigma)$, $d = g-1-|k|$, as before, and $u_{\beta,k}$ is a polynomial in $t$ with constant coefficient 1, which is equal to 1 except possibly if $k = 0$.

\subsection{Examples} \subsubsection{Elliptic surfaces} For $n\geq 1$, let $E(n)$ denote the smooth 4-manifold underlying a simply-connected minimal elliptic surface with no multiple fibers and holomorphic Euler characteristic $n$.  In \cite{OSsymp}, Ozsv\'ath and Szab\'o calculated that $OS_{E(2)} = 1$, meaning that $\Phi_{E(2),\s}$ is trivial on all \spinc structures $\s$ with $c_1(\s)\neq 0$, while if $c_1(\s) = 0$ then $\Phi_{E(2),\s} = 1$. We infer {\it a posteriori} that $E(2)$ has simple type. 

In general, we have that $E(n)$ is diffeomorphic to the fiber sum of $n$ copies of the rational elliptic surface $E(1) = \cee P^2 \# 9\overline{\cee P}^2$, summed along copies of the torus fiber $F$ of the elliptic fibration, using the fibration structure to identify neighborhoods of the fibers. From Theorem \ref{introg1fibsumthm} we infer that the perturbed {\OS} invariant of $E(1)$, calculated with respect to the splitting along the boundary of a neighborhood of $F$ and using a \spinc structure whose Chern class restricts trivially to the complement of $F$, is given by the Laurent series $(t-1)^{-1}$, up to multiplication by $\pm t^n$. For other \spinc structures the perturbed invariant vanishes.

It is straightforward to deduce from this and Theorem \ref{g1thm} that for $n\geq 2$,
\[
OS_{E(n)} = (T - T^{-1})^{n-2},
\]
where $T$ is the class Poincar\'e dual to a regular fiber. In fact, Theorem \ref{g1thm} gives this after summing over rim tori using the homomorphism $\rho$ on the left hand side. Arguments based on the adjunction inequality \cite{OS3,OSsymp}, familiar from Seiberg-Witten theory \cite{GS}, show that only multiples of $T$ can contribute to $OS_{E(n)}$ and therefore application of $\rho$ is unnecessary. Likewise, the only ambiguity remaining in the formula above is an overall sign; the conjugation-invariance of $\Phi_{X,\s}$ when $b^+(X)\geq 2$ due to Ozsv\'ath and Szab\'o \cite{OS3} shows that $OS_{E(n)}$ must be a symmetric polynomial.
\subsubsection{Higher-genus sums} The elliptic surface $E(n)$ can be realized as the double branched cover of $S^2\times S^2$, branched along a surface obtained by smoothing the union of 4 parallel copies of $S^2\times\{pt\}$ and $2n$ copies of $\{pt\}\times S^2$. The projection $\pi_1:S^2\times S^2\to S^2$ to the first factor lifts to an elliptic fibration on $E(n)$, while projection $\pi_2$ on the second factor realizes $E(n)$ as a fibration with typical fiber a surface $\Sigma$ of genus $n-1$, which can be perturbed to be a Lefschetz fibration if desired. Note that $\Sigma$ intersects the fiber $F$ of the elliptic fibration in two (positive) points. Let $X_n = E(n)\#_\Sigma E(n)$ denote the fiber sum of two copies of $E(n)$ along $\Sigma$, and suppose $n\geq 3$. We wish to use Theorem \ref{gengthm1} to calculate the {\OS} invariants of $X_n$. 

A useful observation is that $E(n)$ has simple type by the example above. Corollary \ref{STcor} then shows that we can have a nontrivial contribution to $\rho(OS_{X_n})$ only when $|k| = g-1$, i.e., from \spinc structures $\s$ with $|\langle c_1(\s), [\Sigma]\rangle| = 2g-2 = 2n - 4$. From the preceding example and the fact that $[\Sigma].[F] = 2$, the right-hand side of \eqref{gengformula} in the case $|k| = g-1$ is equal to $\pm 1$, being the product of the invariants arising from $T^{\pm(n-2)}$. Since $T^{\pm(n-2)}$ is equal (up to sign) to the first Chern class $c_1(E(n))$, a convenient way to express these conclusions is that $OS_{X_n} = \pm K \pm K^{-1}$, where $K$ is the canonical class on $X_n$. This formula is true after summing over rim tori.

Note that $X_n$ is diffeomorphic to a minimal complex surface of general type, and therefore this calculation agrees with the corresponding one in Seiberg-Witten theory \cite{witten}.

\subsection{Organization} The first goal of the paper is to set up enough machinery for the proof of Theorem \ref{genproduct}. To this end, the next section recalls the definition of Heegaard Floer homology with twisted coefficients from \cite{OS1} and the corresponding constructions associated to 4-dimensional cobordisms in \cite{OS3}. Section \ref{relgradingsec} discusses a refinement of the relative grading on Heegaard Floer homology, available with twisted coefficients. Sections \ref{pairingsec}, \ref{H1actionsec}, and \ref{conjsec} extend other algebraic features of Heegaard Floer homology to the twisted-coefficient setting, including the pairing mentioned in Theorem \ref{genproduct} and the action on Floer homology by $H_1(Y;\zee)/tors$ which is useful in later calculations. With this machinery in place, section \ref{gluingsec} proves Theorem \ref{genproduct}. Section \ref{perturbsec} defines perturbed Heegaard Floer theory, and deals with the extension of many of the results in preceding sections to that case; in particular Theorem \ref{intropertproductthm}. After making the necessary Floer homology calculations in section \ref{calcsec}, section \ref{productformulasec} gives the proofs of Theorems \ref{introg1fibsumthm} and \ref{introgengfibsumthm}, and thence Theorem \ref{gengthm1} and \ref{g1thm}. We conclude with some remarks on manifolds of simple type in section \ref{simptypesec}.

\section{Preliminaries on Twisted Coefficients}

\subsection{Definitions}

We briefly recall the construction of the Heegaard Floer homology 
groups with ``twisted'' coefficients. For more details, the reader is
referred to \cite{OS1,OS2}. To a closed oriented 3-manifold $Y$ we can
associate a pointed Heegaard diagram $(\Sigma, \balpha,\bbeta,z)$
where $\Sigma$ is a surface of genus $g\geq 1$ and $\balpha =
\alpha_1,\ldots,\alpha_g$ and $\bbeta = \beta_1,\ldots,\beta_g$ are
sets of attaching circles for the two handlebodies in the Heegaard
decomposition. We consider intersection points between the
$g$-dimensional tori $T_\alpha = \alpha_1\times\cdots\times \alpha_g$
and $T_\beta = \beta_1\times\cdots\times\beta_g$ in the symmetric power
$Sym^g(\Sigma)$, which we assume intersect transversely. Recall that
the basepoint $z$, chosen away from the $\alpha_i$ and $\beta_i$,
gives rise to a map $s_z: T_\alpha\cap T_\beta \to Spin^c(Y)$. Given a
\spinc structure $\s$ on $Y$, and under suitable admissibility hypotheses on the Heegaard diagram, the generators for the Heegaard Floer
chain complex $CF^\infty(Y,\s)$ are pairs $[\x, i]$ where $i \in \zee$
and $\x\in T_\alpha\cap T_\beta$ satisfies $s_z(\x) = \s$.

The differential in $CF^\infty$ counts certain maps $u: D^2\to
Sym^g(\Sigma)$ of the unit disk in $\cee$ that connect pairs of
intersection points $\x$ and $\y$. That is to say, we consider maps
$u$ satisfying the boundary conditions:
\begin{eqnarray*}
u(e^{i\theta})\in T_\alpha \mbox{ for $\cos\theta \geq0$}
&&u(i) = \y\\
u(e^{i\theta})\in T_\beta \mbox{ for $\cos\theta \leq 0$}&&u(-i) = \x.
\end{eqnarray*}

For $g>2$ we let $\pi_2(\x,\y)$ denote the set of homotopy classes of
such maps; for $g=2$ we let $\pi_2(\x,\y)$ be the quotient of the set
of such homotopy classes by a further equivalence, the details of
which need not concern us (see \cite{OS1}).

There is a topological obstruction to the existence of any such disk
connecting $\x$ and $\y$, denoted $\epsilon(\x,\y)\in H_1(Y;\zee)$.
To any homotopy class $\phi\in\pi_2(\x,\y)$ we can associate the
quantity $n_z(\phi)$, being the algebraic intersection number between
$\phi$ and the subvariety $\{z\}\times Sym^{g-1}(\Sigma)$. The
following is a basic fact in Heegaard Floer theory:

\begin{prop}[\cite{OS1}]\label{disktop} Suppose $g>1$ and let $\x,\y\in T_\alpha\cap
T_\beta$. If $\epsilon(\x,\y)\neq 0$ then $\pi_2(\x,\y)$ is empty,
while if $\epsilon(\x,\y) = 0$ then there is an affine isomorphism
\[
\pi_2(\x,\y) = \zee\oplus H^1(Y;\zee),
\]
such that the projection $\pi_2(\x,\y)\to \zee$ is given by the map
$n_z$.
\end{prop}

We remark that if $\x = \y$, then the isomorphism in the above proposition is natural (not merely affine).

There is a natural ``splicing'' of homotopy classes
\[
\pi_2(\x,\y) \times \pi_2(\y,\mathbf{z})\to \pi_2(\x,\mathbf{z}),
\]
as well as an action
\[
\pi_2'(Sym^g(\Sigma_g)) \times \pi_2(\x,\y)\to \pi_2(\x,\y),
\]
where $\pi_2'$ denotes the second homotopy group divided by the action
of the fundamental group. (For $g>1$, $\pi_2'(Sym^g(\Sigma_g)) 
\cong \zee$, generated by a class $S$ with $n_z(S) = 1$. When $g>2$,
$\pi_2'(Sym^g(\Sigma_g)) = \pi_2(Sym^g(\Sigma_g))$.) The isomorphism
in the above proposition is affine in the sense that it respects the
splicing action by $\pi_2(\x,\x)$, under the natural identification $\pi_2(\x,\x) = \zee\oplus H^1(Y)$.

The ordinary ``untwisted'' version of Heegaard Floer homology takes
$CF^\infty$ to be generated (over $\zee$) by pairs $[\x, i]$ as above,
equipped with a boundary map such that the coefficient of $[\y,j]$ in
the boundary of $[\x,i]$ is the number of pseudo-holomorphic maps in
all homotopy classes $\phi\in\pi_2(\x,\y)$ having moduli spaces of formal dimension 1 and $n_z(\phi) = i-j$.
The twisted version is similar, but
where one keeps track of all possible homotopy data associated to
$\phi$. In light of the above proposition, this means that we should
form a chain complex freely generated by intersection points $\x$ as a
module over the group ring of $\zee\oplus H^1(Y)$, or
equivalently by pairs $[\x,i]$ over the group ring of $H^1(Y)$.
Following \cite{OS2}, we define:

\begin{definition} An {\em additive assignment} for the diagram
$(\Sigma,\balpha,\bbeta,z)$ is a collection of functions
\[
A_{\x,\y}: \pi_2(\x,\y)\to H^1(Y;\zee)
\]
that satisfies
\begin{enumerate}
\item $A_{\x,\z}(\phi*\psi) = A_{\x,\y}(\phi) + A_{\y,\z}(\psi)$
whenever $\phi\in\pi_2(\x,\y)$ and $\psi\in\pi_2(\y,\z)$.
\item $A_{\x,\y}(S*\phi) = A_{\x,\y}(\phi)$ for $S\in\pi_2'(Sym^g(\Sigma_g))$.
\end{enumerate}
\end{definition}

We will drop the subscripts from $A_{\x,\y}$ whenever possible. It is
shown in \cite{OS2} how a certain finite set of choices (a
``complete set of paths'') gives rise to an additive assignment in the
above sense. We can also assume that $A_{\x,\x}: \pi_2(\x,\x)\cong \zee\oplus H^1(Y)\to H^1(Y)$ is the natural projection on the second factor.

\begin{definition}\label{twistcxdef} Let $(\Sigma,\balpha,\bbeta,z)$ be a pointed
Heegaard diagram for $Y$ and $\s\in Spin^c(Y)$. Fix an additive assignment
$A$ for the diagram. The twisted Heegaard Floer chain complex
$CF^\infty(Y,\s;\zee[H^1(Y)])$ is the module freely generated
over $\zee[H^1(Y)]$ by pairs $[\x,i]$, with differential
$\partial^\infty$ given by
\[
\partial^\infty[\x,i] = \sum_{y\in T_\alpha\cap T_\beta}
\sum_{\begin{array}{c}\mbox{\scriptsize $\phi\in\pi_2(\x,\y)$}\\ \mbox{\scriptsize $\mu(\phi) = 1$}\end{array}}
\#\widehat{\mathcal{M}}(\phi)\cdot e^{A(\phi)} [\y, i-n_z(\phi)],
\]
where the symbol $e^{A(\phi)}$ indicates the variable in $\zee[H^1(Y)]$ corresponding to $A(\phi)$.
\end{definition}
Here $\mathcal{M}(\phi)$ denotes the space of holomorphic disks in the 
homotopy class $\phi$, where ``holomorphic'' is defined relative to 
an appropriately generic path of almost-complex structure on 
$Sym^g(\Sigma_g)$. For such a path, $\mathcal{M}(\phi)$ is a smooth 
manifold of dimension given by a Maslov index $\mu(\phi)$. There is 
an action of $\arr$ on  $\mathcal{M}(\phi)$ by reparametrization of the disk, and $\widehat{\mathcal{M}}(\phi)$ denotes 
the quotient of $\mathcal{M}(\phi)$ by this action. When $\mu(\phi) = 
1$, $\widehat{\mathcal{M}}(\phi)$ is a compact, zero-dimensional manifold. An appropriate choice 
of ``coherent orientation system'' serves to orient the points of  
$\widehat{\mathcal{M}}(\phi)$ in this case, and 
$\#\widehat{\mathcal{M}}(\phi)$ denotes the signed count of these points.
It is shown in \cite{OS1,OS2} that under appropriate admissibility
hypotheses on the diagram $(\Sigma,\balpha,\bbeta,z)$ the chain 
homotopy type of $CF^\infty(Y,\s;\zee[H^1(Y)])$ is an invariant of 
$(Y,\s)$.

As in the introduction, in much of what follows we will write $R_Y$ for 
the ring $\zee[H^1(Y)]$, or simply $R$ when the underlying 
3-manifold is apparent from context. Note that by choosing a basis 
for $H^1(Y)$ we can identify $R$ with the ring of Laurent 
polynomials in $b_1(Y)$ variables.

By following the usual constructions of Heegaard Floer homology, we 
obtain other variants of the above with coefficients in $R_Y$: namely 
by considering only generators $[\x,i]$ with $i<0$ we obtain a 
subcomplex $CF^-(Y,\s;R)$ whose quotient complex is $CF^+(Y,\s; R)$, with 
associated homology groups $HF^-$ and $HF^+$ respectively. There is 
an action $U: [\x,i]\mapsto[\x,i-1]$ on $CF^\infty$ as usual; the 
kernel of the induced action on $CF^+$ is written $\widehat{CF}$ with 
homology $\widehat{HF}(Y,\s;R)$. There is a relative grading on the Floer complex with respect to which $U$ decreases degree by 2; we will discuss gradings further in section \ref{relgradingsec}.

Given any module $M$ for $R_Y$ we can form Heegaard Floer 
homology with coefficients in $M$ by taking the homology of the 
complex $CF\otimes_R M$. In particular if $M = \zee$, equipped with 
the action of $R_Y$ by which every element of $H^1(Y)$ acts as 
the identity, we recover the ordinary untwisted theory.

For use in later sections, we introduce the following notion of {\it conjugation} of $R_Y$-modules. First, observe that the automorphism $x\mapsto -x$ of $H^1(Y)$ induces an automorphism $R_Y\to R_Y$ that we refer to as conjugation, and write as $r\mapsto \bar{r}$ for $r\in R_Y$. Now if $M$ is a module for $R_Y$, we let $\overline{M}$ denote the additive group $M$ equipped with the ``conjugate'' module structure in which module multiplication is given by 
\[
r\otimes m\mapsto \bar{r}\cdot m
\]
for $r\in R_Y$ and $m\in \overline{M}$. 

\subsection{Twisted cobordism invariants}\label{cobordismsec}

We now sketch the construction and main properties of 
twisted-coefficient Heegaard Floer invariants associated to 
cobordisms, which can be found in greater detail in \cite{OS3}. 
Recall that if $W: Y_1\to Y_2$ is an oriented 4-dimensional cobordism 
and $M$ is a module for $R_1 :=R_{Y_1} = \zee[H^1(Y_1)]$, then 
there is an induced module $M(W)$ for $R_2 = R_{Y_2}$ defined as follows. 
Let
\[
K(W) = \ker(H^2(W,\partial W)\to H^2(W))
\]
be the kernel of the map in the long exact sequence for the pair 
$(W,\partial W)$: then $\zee[K(W)]$ is a module for $R_1$ and $R_2$ 
via the coboundary maps $H^1(Y_i)\to K(W)\subset H^2(W,\partial W)$. Define
\[
M(W) = \overline{M}\otimes_{R_1} \zee[K(W)].
\]
Then $M(W)$ is a module for $R_2$ in the obvious way. The reason for the appearance of the conjugate module $\overline{M}$ above has to do with the fact that the orientation of $W$ induces the opposite orientation on $Y_1$ from the given one, and will be explained more fully in the next section.

Ozsv\'ath and Szab\'o show in \cite{OS3} how to associate to a
cobordism $W$ as above with \spinc structure $\s$ a homomorphism
\[
F_{W,\s}^\circ: 
HF^\circ(Y_1,\s_1; M)\to HF^\circ(Y_2,\s_2;M(W))\]
(where $\s_i$ denotes the 
restriction of $\s$ to $Y_i$, and $\circ$ indicates a map between each of the 
varieties of Heegaard Floer homology, respecting the long exact 
sequences relating them). This is defined as a composition
\[
F^\circ_W = E^\circ\circ H^\circ\circ G^\circ,
\]
where $G^\circ$ is associated to the 1-handles in $W$, $H^\circ$ to 
the 2-handles, and $E^\circ$ to the 3-handles. Note that the 
coefficient module remains unchanged by cobordisms consisting of 1- or 
3-handle additions. Indeed, such cobordisms induce homomorphisms in 
an essentially formal way, so we simply refer the reader to \cite{OS3} 
for the definition of $E^\circ$ and $G^\circ$. 

Suppose that $W$ is a cobordism consisting of 2-handle additions, so 
that we can think of $W$ as associated to surgery on a framed link 
$L\subset Y_1$. In this situation, Ozsv\'ath and Szab\'o  
construct a ``Heegaard triple'' $(\Sigma,\balpha,\bbeta,\bgamma, z)$ associated to $W$. This diagram describes three 
3-manifolds $Y_{\alpha\beta}$, $Y_{\beta\gamma}$ and 
$Y_{\alpha\gamma}$ obtained by using the indicated circles on 
$\Sigma$ as attaching circles, such that 
\[
Y_{\alpha\beta} = Y_1, \qquad Y_{\beta\gamma} = \#^{k}S^1\times S^2, 
\qquad Y_{\alpha\gamma} = Y_2,
\]
where $k$ is the genus of $\Sigma$ minus the number of components of 
$L$. In fact the diagram $(\Sigma, \balpha,\bbeta,\bgamma,z)$ 
describes a 4-manifold $X_{\alpha\beta\gamma}$ in a natural way, 
whose boundaries are the three manifolds above. Furthermore, in the 
current situation, $X_{\alpha\beta\gamma}$ is obtained from $W$ by 
removing the regular neighborhood of a 1-complex (see \cite{OS3}). 

We can arrange that the top-dimensional generator of 
$HF^{\leq 0}(Y_{\beta\gamma}, \s_0; \zee)\cong 
\Lambda^* 
H^1(Y_{\beta\gamma};\zee)\otimes\zee[U]$ is represented by an intersection point 
$\Theta\in T_\beta\cap T_\gamma$ (here $\s_0$ denotes the \spinc 
structure on $\#^k S^1\times S^2$ having $c_1(\s_0) = 0$). The map $F^\circ$ is defined by 
counting holomorphic triangles, with the aid of another additive 
assignment. To describe this, suppose $\x\in T_\alpha\cap T_\beta$, 
$\y\in T_\beta\cap T_\gamma$, and $\w\in T_\alpha\cap T_\gamma$ are 
intersection points arising from a Heegaard triple $(\Sigma, 
\balpha,\bbeta,\bgamma,z)$. Let $\Delta$ denote a standard 2-simplex, 
and write $\pi_2(\x,\y,\w)$ for the set of 
homotopy classes of maps $u: \Delta\to Sym^g(\Sigma)$ that send 
the boundary arcs of $\Delta$ into $T_\alpha$, $T_\beta$, and 
$T_\gamma$ respectively, under a clockwise ordering of the boundary 
arcs $e_\alpha$, $e_\beta$, and $e_\gamma$ of $\Delta$, and such that
\[
u(e_\alpha\cap e_\beta) = \x, \qquad u(e_\beta\cap e_\gamma) = \y, 
\qquad u(e_\alpha\cap e_\gamma) = \w.
\]
Again there is a topological obstruction $\epsilon(\x,\y,\w)\in 
H_1(X_{\alpha\beta\gamma}; \zee)$ that 
vanishes if and only if $\pi_2(\x,\y,\w)$ is nonempty. The analogue of 
Proposition \ref{disktop} in this context is the following. 

\begin{prop}[\cite{OS3}]\label{triangletop} Let 
$(\Sigma,\balpha,\bbeta,\bgamma,z)$ be a pointed Heegaard triple as 
above, and $X_{\alpha\beta\gamma}$ the associated 4-manifold. Then 
whenever $\epsilon(\x,\y,\w)=0$  we have an (affine) isomorphism
\[
\pi_2(\x,\y,\w) \cong \zee\oplus H_2(X_{\alpha\beta\gamma};\zee)
\]
where the projection to $\zee$ is given by $\psi\mapsto n_z(\psi)$. 
\end{prop}

There is an obvious ``splicing'' action on homotopy classes of 
triangles by disks on each corner; the above identification respects 
this action. 

Recall from \cite{OS2} that the basepoint $z$ gives rise to a map 
\[
s_z:\coprod_{\x,\y,\w}\pi_2(\x,\y,\w)\to Spin^c(X_{\alpha\beta\gamma}),
\]
such that triangles $\psi\in\pi_2(\x,\y,\w)$ and 
$\psi'\in\pi_2(\x',\y',\w')$ have $s_z(\psi) = s_z(\psi')$ if and 
only if there exist disks $\phi_\x\in\pi_2(\x,\x')$, 
$\phi_\y\in\pi_2(\y,\y')$ and $\phi_{\w}\in\pi_2(\w,\w')$ with 
$\psi' = \psi + \phi_\x + \phi_\y + \phi_\w$. In this case $\psi$ 
and $\psi'$ are said to be {\it \spinc equivalent}. Note that in 
case $(\Sigma,\balpha,\bbeta,\bgamma,z)$ describes a 2-handle 
cobordism $W$ as previously, we can think of $s_z$ as a function
\[
s_z : \coprod_{\x,\w} \pi_2(\x,\Theta,\w)\to Spin^c(W).
\]

\begin{definition}\label{assndef} An {\em additive assignment} for a Heegaard 
triple $(\Sigma,\balpha,\bbeta,\bgamma,z)$ describing a 2-handle cobordism 
$W: Y_1\to Y_2$ as above 
is a function 
\[
A_W: \coprod_{\s\in Spin^c(W)} s_z^{-1}(\s)\to K(W)
\]
obtained in the following manner. For a fixed $\psi_0\in 
s_z^{-1}(\s)$, let $\psi = \psi_0 + \phi_{\alpha\beta} + 
\phi_{\beta\gamma} + \phi_{\alpha\gamma}$ be an arbitrary element of 
$s_z^{-1}(\s)$. Then set
\[
A_W(\psi) = \delta(-A_1(\phi_{\alpha\beta}) + A_2(\phi_{\alpha\gamma}))
\]
where $A_i$ are additive assignments for $Y_i$ and $\delta: 
H^1(\partial W)\to H^2(W,\partial W)$ is the 
coboundary from the long exact sequence of $(W,\partial W)$. 
\end{definition}

We are now in a position to define the map on Floer homology induced 
by $W$ (given additive assignments on $Y_1$, $Y_2$, and $W$). We 
again refer to \cite{OS2,OS3} for the details required to make full 
sense of the following.

\begin{definition}\label{twistmapdef} For a triple $(\Sigma,\balpha,\bbeta,\bgamma,z)$ 
describing a 2-handle cobordism $W$ with \spinc structure $\s$, we 
define
\[
F_{W,\s}^\circ: HF^\circ(Y_1,\s_1;M)\to HF^\circ(Y_2,\s_2;M(W)),
\]
where $\s_i = \s|_{Y_i}$, to be the map induced on homology by the chain map
\[
[\x,i]\mapsto \sum_{\w\in T_\alpha\cap T_{\gamma}} 
\sum_{\begin{array}{c}\mbox{\scriptsize $\psi\in\pi_2(\x,\Theta,\w)$}\\ 
\mbox{\scriptsize $\mu(\psi) = 
0$}\end{array}} \#{\mathcal M}(\psi) \cdot [\w,i-n_z(\psi)]\otimes 
e^{A_W(\psi)}.
\]
Here $\mu(\psi)$ denotes the expected dimension of the moduli space 
${\mathcal M}(\psi)$ of pseudo-holomorphic triangles in the homotopy 
class $\psi$, and $\#{\mathcal M}(\psi)$ indicates the signed count of 
points in a compact oriented 0-dimensional manifold.
\end{definition}

We should note that while the Floer homology $HF^\circ(Y,\s; M)$ does 
not depend on the additive assignment $A_Y$, the map $F_{W,\s}$ does 
depend on the choice of $A_W$ as in definition \ref{assndef} through 
the reference triangle $\psi_0$. Changing this choice 
has the effect of pre-  (post-) composing $F_{W}$ with the action of 
an element of $H^1(Y_1)$ (resp $H^1(Y_2)$), which in turn 
act in $M(W)$ via the coboundary. Likewise the 
generator $\Theta$ is determined only up to sign, so that $F_{W}$ has 
a sign indeterminacy as well. Following \cite{OS3}, we let 
$[F_{W,\s}^\circ]$ denote the orbit of $F_{W,\s}^\circ$ under the 
action of $H^1(Y_1)\oplus H^1(Y_2)$. 

With the conventions employed here $F_{W,\s}^\circ$ is ``antilinear'' with respect to the action of $R_{Y_1}$, meaning that $F_{W,\s}^\circ(r\,\xi) = \bar{r}\,F_{W,\s}^\circ(\xi)$ for $r\in R_{Y_1}$.

\subsection{Composition law} \label{complaw}

An advantage to using twisted coefficent modules for Heegaard 
Floer homology is the availability of a refined composition law in 
this situation. To describe this, we must first understand the 
behavior of the coefficient modules themselves under composition of 
cobordisms. The following lemma will be useful in formulating results 
to come; as usual, ordinary (co)homology is taken with integer coefficients.

\begin{lemma}\label{coefflemma} Let $W = W_1\cup_{Y_1} W_2$ be the composition of two 
cobordisms $W_1:Y_0\to Y_1$ and $W_2:Y_1\to Y_2$. Define
\[
K(W,Y_1) = \ker[\rho_1\oplus\rho_2: H^2(W,\partial W) \to 
H^2(W_1)\oplus H^2(W_2)],
\]
where $\rho_i$ denotes the restriction map $H^2(W,\partial W)\to 
H^2(W_i)$. Then 
\[
\zee[K(W_1)]\otimes_{\zee[H^1(Y_1)]} \zee[K(W_2)] \cong \zee[K(W,Y_1)]
\]
as modules over $\zee[H^1(Y_0)]$ and $\zee[H^1(Y_2)]$.
\end{lemma}

\begin{proof} We have
\[
\zee[K(W_1)]\otimes_{\zee[H^1(Y_1)]} \zee[K(W_2)] \cong 
\zee\left[\frac{K(W_1)\oplus K(W_2)}{H^1(Y_1)}\right],
\]
so the claim amounts to exhibiting an isomorphism 
\[
\frac{K(W_1)\oplus K(W_2)}{H^1(Y_1)} \cong K(W, Y_1).
\]
To see this, consider the diagram
\begin{diagram}
H^1(Y_1) &\rTo & H^2(W_1,\partial W_1)\oplus H^2(W_2,\partial W_2) & 
\rTo^f & H^2(W, \partial W) \\
 & & & & \dTo^{\rho_1\oplus\rho_2} \\
 & & & & H^2(W_1)\oplus H^2(W_2),
\end{diagram}
where the horizontal row is (the Poincar\'e dual of) the 
Mayer-Vietoris sequence. Write
\[
i_*: H^2(W_1,\partial W_1)\to H^2(W,\partial W) 
\qquad\mbox{and}\qquad j_*: H^2(W_2,\partial W_2)\to H^2(W,\partial W)
\]
for the components of $f$; then it is not hard to see that
\[
\rho_1\circ i_*:  H^2(W_1,\partial W_1)\to H^2(W_1) 
\qquad\mbox{and}\qquad \rho_2\circ j_*: H^2(W_2,\partial W_2)\to 
H^2(W_2)
\]
agree with the maps induced by inclusion, while
\[
\rho_2\circ i_* = 0 \qquad\mbox{and}\qquad \rho_1\circ j_* = 0.
\]
From this it is easy to deduce that 
$f^{-1}(K(W, Y_1)) = K(W_1)\oplus K(W_2)$,
from which the lemma follows.
\end{proof}

\begin{remark}\label{spincrestriction} If $W$ is a cobordism between homology spheres, or 
more generally if $H^2(W,\partial W)\to H^2(W)$ is an isomorphism, 
then there is an identification
\[
K(W,Y_1) = \ker[H^2(W)\to H^2(W_1)\oplus H^2(W_2)],
\]
the kernel of the restriction map in the ordinary Mayer-Vietoris 
sequence in cohomology. In this case if $\s_1$ and $\s_2$ are \spinc structures on 
$W_1$ and $W_2$, then $K(W,Y_1)$ parametrizes \spinc structures $\s$ on $W$ 
such that $\s|_{W_i} = \s_i$ (when that set is nonempty). In the case of a closed 4-manifold $X$, the module $M_{X,Y}$ of the introduction is simply $\zee[K(W,Y)]$ where $W$ is obtained from $X$ by removing a 4-ball on each side of $Y$.
\end{remark}

When regarding $W$ as a single cobordism the group relevant to 
twisted coefficient modules is $K(W)$, while if $W = W_1\cup W_2$ is 
viewed as a composite the coefficient modules change by tensor product 
with the group ring of $K(W,Y_1)$ (in light of the lemma above). By 
commutativity of the diagram
\begin{diagram}
H^2(W,\partial W) &\rTo & H^2(W) \\
& \rdTo^{\rho_1\oplus\rho_2} & \dTo \\
&& H^2(W_1)\oplus H^2(W_2),
\end{diagram}
there is a natural inclusion $\iota : K(W)\to K(W,Y_1)$. This gives rise to a 
projection map 
\[
\Pi : \zee[K(W,Y_1)]\to \zee[K(W)],
\]
namely (c.f. \cite{OS3})
\[
\Pi(e^w) = \left\{ \begin{array}{ll} e^w & \mbox{if $w = \iota(v)$ for 
some $v$} \\ 0 & \mbox{otherwise.}\end{array}\right.
\]

Thus, if $M$ is a module for $\zee[H^1(Y_0)]$ we obtain a map
\[
\Pi_M: M(W_1)(W_2)\to M(W)
\]
by tensor product of the identity with $\Pi$ under the identifications
\[
M(W_1)(W_2) = \overline{M}\otimes_{\zee[H^1(Y_0)]} \zee[K(W,Y_1)] 
\qquad\mbox{and}\qquad M(W) = \overline{M}\otimes_{\zee[H^1(Y_0)]}\zee[K(W)].
\]

The refined composition law for twisted coefficients can be stated as 
follows.

\begin{theorem}[Theorem 3.9 of \cite{OS3}]\label{complaw} Let $W = W_1\cup_{Y_1} W_2$ be 
a composite cobordism as above with \spinc structure $\s$. Write 
$\s_i = \s|_{W_i}$. Then there are choices of representatives for the 
various maps involved such that
\[
[F_{W,\s}^\circ] = [\Pi_M\circ F^\circ_{W_2,\s_2}\circ 
F^\circ_{W_1,\s_1}].
\]
More generally, if $h\in H^1(Y_1)$ then for these choices we have
\[
[F_{W,\s - \delta h}^\circ] = [\Pi_M\circ F_{W_2,\s_2}^\circ\circ 
e^h\cdot F_{W_1,\s_1}^\circ],
\]
where $\delta h$ is the image of $h$ under the Mayer-Vietoris 
coboundary $H^1(Y_1)\to H^2(W)$.
\end{theorem}

We should also remark that for a cobordism $W:Y_1\to Y_2$ with \spinc structure 
$\s$ the map 
\[
F_{W,\s}^\circ: HF^\circ(Y_1, \s_1; \zee)\to HF^\circ(Y_2,\s_2; \zee)
\]
in untwisted Floer homology can be obtained from the twisted-coefficient map
\[
HF^\circ(Y_1,\s_1;\zee)\to HF^\circ(Y_2,\s_2; \zee(W))
\]
(here $\zee(W)$ is the module $M(W)$ with $M = \zee$, namely 
$\zee(W) = \zee\otimes_{\zee[H^1(Y_1)]}\zee[K(W)] = 
\zee[\ker(H^2(W,Y_2)\to H^2(W))]$)
by composition with the map $\epsilon_*$ induced in homology by the homomorphism
\[
\epsilon: \zee(W)\to \zee
\]
of coefficient modules that sends each element of $\ker(H^2(W,Y_2)\to 
H^2(W))$ to $1$.

\section{Refined relative gradings}\label{relgradingsec} 

The $\zee$-coefficient version 
of Heegaard Floer homology is naturally a relatively cyclically 
graded theory, in general. This means that if $\mathcal S = 
\{[\x,i]\,|\,s_z(\x) = \s\}$ denotes the natural generating set for 
$CF^\infty(Y,\s;\zee)$ then there is a map
\[
\gr: {\mathcal S}\times{\mathcal S}\to \zee/\D(\s)\zee,
\]
where 
\begin{equation}\label{divisdef}
\D(\s) = \mbox{gcd}\{\langle c_1(\s), h\rangle\,| \, h\in 
H_2(Y;\zee)\}
\end{equation}
is the {\em divisibility} of $c_1(\s)$ (or by abuse of 
language, of $\s$ itself). The 
differential in $CF^\infty$ has degree $-1$ with respect to this 
grading, while the endomorphism $U$ has degree $-2$.

In the case of fully twisted coefficients (coefficients in 
$\zee[H^1(Y)]$), Ozsv\'ath and Szab\'o \cite{OS1} observe that there 
is a lift of this cyclic grading to a relative $\zee$-grading. Here we provide an extension of this construction to Floer homology with coefficients in an arbitrary (graded) module $M$, in which elements of $H^1(Y)\subset R_Y$ are explicitly assigned nontrivial degrees depending on their Chern numbers. That the action of such elements on fully-twisted Floer homology shifts degree by their Chern numbers is implicit in the definition given in \cite{OS2}.

\begin {definition}\label{sgradingdef} Fix a closed, oriented, \spinc 3-manifold $(Y,\s)$. Define the {\em $\s$-grading} of $\zee[H^1(Y)]$ by
\begin{equation}\label{sgrading}
\gr_\s(x) = -\langle c_1(\s)\cup x, [Y]\rangle \quad \mbox{for $x\in H^1(Y)$.}
\end{equation}
\end{definition}
The $\s$-grading makes $\zee[H^1(Y)]$ into a graded ring, isomorphic to a multivariable Laurent polynomial ring in which the variables have degrees determined by their negative Chern numbers \eqref{sgrading}. When thinking of $\zee[H^1(Y)]$ as a graded ring, we write it as $R_{Y,\s}$ or just $R_Y$. It is important to recognize that this grading depends on both the \spinc structure $\s$ and the orientation of $Y$, though we usually do not include $\s$ in the notation. In particular, if $-Y$ denotes the 3-manifold $Y$ with its orientation reversed, then although $R_{Y,\s} = R_{-Y,\s}$ as sets, the gradings have opposite sign. On the other hand, the conjugation homomorphism $c: r\mapsto \bar{r}$ induces an isomorphism of graded rings $c: R_{Y,\s}\to R_{-Y,\s}$. 

\begin{definition} Let $(\Sigma,\balpha,\bbeta,z)$ be a marked Heegaard triple describing the 3-manifold $Y$. Fix a \spinc structure $\s$ for $Y$ and an additive assignment $\{A_{\x,\y}\}$ for the diagram. The {\em relative $\zee$ grading} between generators $[\x,i]$ and $[\y,j]$ for $CF^\circ(Y,\s; R_Y)$ is defined by
\begin{equation}\label{grdef}
\ugr([\x,i],[\y,j]) = \mu(\phi) + 2(i-j) - 2n_z(\phi) - \langle c_1(\s)\cup A_{\x,\y}(\phi),[Y] \rangle,
\end{equation}
where $\phi$ is any element of $\pi_2(\x,\y)$. More generally, if $r_1, r_2\in R_Y$ are homogeneous elements, then we set
\[
\ugr(r_1\cdot [\x,i],\,r_2\cdot [\y,j]) = \ugr([\x,i],[\y,j]) + \gr_\s(r_1) - \gr_\s(r_2).
\]
\end{definition}

It is not hard to check that the expression \eqref{grdef} is independent of the choice of $\phi\in\pi_2(\x,\y)$, and that the differential in $CF^\infty(Y,\s; R_Y)$ has relative degree $-1$ with respect to the above grading.

Now suppose $M$ is a module for $R_Y$, which is equipped with a grading $\gr_M$ satisfying
\begin{equation}\label{compatgr}
\gr_M(r\cdot m) = \gr_\s(r) + \gr_M(m).
\end{equation}
(Here we suppose $R_Y$ is equipped with the $\s$-grading induced by some $\s\in Spin^c(Y)$.) Then the twisted Floer complex  $CF(Y,\s; M) = CF(Y,\s;R_Y)\otimes_{R_Y} M$ naturally carries a relative $\zee$ grading given by
\[
\ugr(m_1\,[\x,i], m_2\,[\y,j]) = \ugr([\x,i],[\y,j]) + \gr_M(m_1) - \gr_M(m_2),
\]
inducing a relative $\zee$ grading on the Floer homology with coefficients in $M$.

More generally, if \eqref{compatgr} holds modulo some integer $d\in \zee$, we obtain a relative $\zee/d\zee$ grading on $CF(Y,\s;M)$. For example, taking $M = \zee$ to be the trivial $R_Y$-module supported in grading 0, we have for $n\in \zee$ and $r\in H^1(Y)$, 
\[
\gr_M(r\cdot n) = \gr_M(n) = 0 \quad \mbox{while} \quad \gr_\s(r) + \gr_M(n) = -\langle c_1(\s)\cup r, [Y]\rangle.
\]
Thus \eqref{compatgr} holds modulo the divisibility $\D(\s)$ of $\s$; in particular, the ``untwisted'' Floer complex $CF(Y,\s; \zee)$ carries a relative cyclic grading by $\zee/\D(\s)\zee$.

In light of these observations, henceforth we will assume that all modules for $R_Y$ are graded, and often omit this assumption from statements. In particular, we will suppose that \eqref{compatgr} holds at least modulo some $d\in \zee$.

That the homology of $CF^\circ(Y,\s;M)$ is an invariant of $(Y,\s)$ follows by verifying that the arguments in \cite{OS1,OS2} respect the grading described here, together with the following. 

\begin{lemma}\label{addassnlemma} The chain complex $CF^\infty(Y,\s;M)$, equipped with the relative grading defined above, is independent of the choice of additive assignment up to graded chain isomorphism.
\end{lemma} 

\begin{proof} Suppose $A_1$ and $A_2$ are two additive assignments satisfying the criteria at the beginning of the previous section, and let $\{\theta_\x\in \pi_2(\x_0,\x)\}$ be a complete set of paths for the \spinc structure $\s$. That is, $\x_0$ is a fixed intersection point with $s_z(\x_0) = \s$ and $\theta_\x$ is some choice of homotopy class for each $\x\in s_z^{-1}(\s)$. Define a homomorphism $F: CF^\infty(Y,\s, A_1)\to CF^\infty(Y,\s, A_2)$ between the chain complexes constructed with the two choices of assignment, by the formula
\[
F([\x,i]) = e^{A_2(\theta_\x) - A_1(\theta_\x)} [\x,i].
\]
Then
\[
F(\partial^\infty[\x,i]) = \sum_{\y,\phi} \#\widehat{\mathcal M}(\phi) [\y,i-n_z(\phi)] \,e^{A_1(\phi)}\,e^{A_2(\theta_\y) - A_1(\theta_\y)},
\]
while
\[
\partial^\infty F([\x,i]) = \sum_{\y,\phi}\#\widehat{\mathcal M}(\phi)[\y,i-n_z(\phi)]\,e^{A_2(\theta_\x) - A_1(\theta_\x)}\,e^{A_2(\phi)}.
\]
Now, for $i = 1,2$ and given $\phi\in \pi_2(\x,\y)$, we have $\theta_\x * \phi = \theta_\y * {\mathcal P}_\phi$ for a periodic domain ${\mathcal P}_\phi\in \pi_2(\y,\y)$ (up to addition of a multiple of the sphere class $[S]$, which does not affect the value of the additive assignment). Therefore by additivity
\begin{equation}\label{Aeqn}
A_i(\theta_\x) + A_i(\phi) = A_i(\theta_\y) + H({\mathcal P}_\phi)
\end{equation}
where $H({\mathcal P}_\phi)\in H^1(Y;\zee)$ is the cohomology class corresponding to ${\mathcal P}_\phi$. It follows that the group ring elements appearing in the previous expressions are equal, so that $F$ is a chain map. Since $F$ is clearly an isomorphism of $R_Y$-modules, we get that $CF^\infty(Y,\s, A_1)$ and $CF^\infty(Y,\s, A_2)$ are isomorphic as ungraded $R_Y$-chain complexes. 

To see invariance of the relative gradings, we calculate that 
\begin{eqnarray*}
\ugr_{A_2}(F([\x,i]), F([\y,j])) &=& \ugr_{A_2}([\x,i],[\y,j]) + \gr_\s(A_2(\theta_\x) - A_1(\theta_\x)) - \gr_\s(A_2(\theta_\y) - A_1(\theta_\y))\\
&=& \mu(\phi) + 2(i-j) - 2n_z(\phi) - \langle c_1(\s)\cup A_2(\phi),[Y]\rangle\\&&\hspace*{.25cm} - \langle c_1(\s)\cup (A_2(\theta_\x) - A_1(\theta_\x)), [Y]\rangle\\&&\hspace*{.25cm} + \langle c_1(\s)\cup(A_2(\theta_\y) - A_1(\theta_\y)), [Y]\rangle
\end{eqnarray*}
Applying the identity \eqref{Aeqn} twice, this easily reduces to $\ugr_{A_1}([\x,i],[\y,j])$.
\end{proof}

We now show that one can always work with relatively $\zee$-graded Floer homology (rather than groups with a finite cyclic grading) if the coefficient module and \spinc structure are induced by a cobordism from $S^3$ to $Y$. To do so we spell out the notion of a conjugate module in the current, graded, context. As usual, if $Y$ is an oriented 3-manifold then $-Y$ denotes the same manifold with the opposite orientation.

\begin{definition} Suppose $M$ is a graded $R_Y$-module. The {\em conjugate module} $\overline{M}$ is the graded $R_{-Y}$ module whose underlying graded group is the same as $M$, but whose multiplication is given by 
\[
r\otimes m\mapsto \bar{r}\cdot m, \quad r\in R_{-Y}.
\]
\end{definition}

It is clear that if \eqref{compatgr} is satisfied for $M$ as a graded $R_Y$-module (modulo $d$), then the same is true for $\Mbar$ as a graded $R_{-Y}$ module.

\begin{prop}\label{Zgradingprop} Suppose $W: S^3\to Y$ is an oriented cobordism with \spinc structure $\s_W$, and $M(W)$ is the induced module for $R_Y$. Then $M(W)$ carries a natural grading induced by $\s_W$ that is compatible with the $\s$-grading on $R_Y$ in the sense of \eqref{compatgr}, where $\s$ is the restriction of $\s_W$ to $Y$. In particular, $HF^\circ(Y,\s;M(W))$ carries a relative $\zee$ grading.

More generally, if $W: Y_1\to Y_2$ is an oriented cobordism with \spinc structure $\s_W$ whose restrictions to $Y_1$ and $Y_2$ are $\s_1$ and $\s_2$ respectively, and $M$ is a module for $R_{Y_1}$ satisfying \eqref{compatgr} modulo $d$, then the induced module $M(W)$ carries a grading induced by $\s_W$ also satisfying \eqref{compatgr} modulo $d$. 
\end{prop}

\begin{proof} Observe first that since $\partial W = -Y_1\coprod Y_2$, we should most naturally consider $\zee[K(W)]$ as a module for $R_{-Y_1}$ and $R_{Y_2}$. Recall that $M(W) = \Mbar\otimes_{R_{-Y_1}} \zee[K(W)]$, where $K(W) = \ker(H^2(W,\partial W) \to H^2(W))$. Define $\gr_W : K(W) \to \zee$ by
\[
\gr_W(k) = -\langle c_1(\s_W)\cup k, [W,\partial W]\rangle,
\]
and use this to impose a grading $\gr_W$ on $\zee[K(W)]$. We claim that this grading respects the action by $R_{-Y_1}$ and $R_{Y_2}$, where the latter are equipped with gradings coming from the restrictions of $\s_W$. To see this, it suffices to note that the actions of $R_{-Y_1}$ and $R_{Y_2}$ on $\zee[K(W)]$ are induced by the coboundary maps $\delta_i: H^1(Y_i)\to K(W)$, and that
\[
c_1(s_W) \cup \delta_i r = \delta_i( j_i^*c_1(\s_W) \cup r) = \delta_i(c_1(\s_i)\cup r)
\]
where $j_i: Y_i\to W$, $i = 1,2$ are the inclusions of the boundary components. Hence for $r\in H^1(Y_1)$,
\[
\langle c_1(\s_W)\cup\delta_1r, [W,\partial W]\rangle  = \langle c_1(\s_1)\cup r, [-Y_1]\rangle
\]
and correspondingly for elements of $H^1(Y_2)$. 
\end{proof}

\section{Pairings and Duality}\label{pairingsec}

In \cite{OS3}, Ozsv\'ath and Szab\'o defined a pairing
\[
\langle\cdot,\cdot\rangle: HF^+(Y,\s;\zee)\otimes HF^-(-Y,\s;\zee)\to 
\zee
\]
on Floer homology, with respect to which cobordism-induced maps 
satisfy a certain duality. Here we extend this pairing to Floer 
homology with twisted coefficients and prove a corresponding duality; throughout we use the ring $R_Y$ and modules $M$ that are graded via some choice of \spinc structure on $Y$ as in the previous section.

Recall that if $(\Sigma,\balpha,\bbeta,z)$ is a pointed Heegaard 
diagram for $Y$, then $(-\Sigma,\balpha,\bbeta, z)$ describes the oppositely-oriented manifold $-Y$, 
and the map $s_z$ is invariant under this change of orientation.

\begin{definition} Define a pairing
\[
\langle\cdot,\cdot\rangle : CF^\infty(Y,\s;R_Y)\otimes_{R_Y} 
\overline{CF^\infty(-Y,\s;R_{-Y})}\longto R_Y
\]
as follows: for generators $[\x,i]\in CF^\infty(Y,\s; R_Y)$ 
and $[\y,j]\in CF^\infty(-Y,\s; R_{-Y})$ set
\begin{equation}\label{pairingdef}
\langle [\x,i],[\y,j]\rangle = \left\{\begin{array}{ll} 1 & 
\mbox{if $\x= \y$ and $j = -i-1$} \\ 0 & 
\mbox{otherwise.}\end{array}\right.
\end{equation}
The desired pairing is obtained by extending by $R_Y$-linearity.
\end{definition}

We must check that this definition has the desired properties:

\begin{lemma}\label{pairinglemma} For any $\xi\in CF^\infty(Y,\s;R_Y)$, $\eta\in 
CF^\infty(-Y,\s;R_{-Y})$, we have
\begin{eqnarray*}
\langle \partial^\infty \xi, \eta\rangle &=& \langle \xi, 
\partial^\infty\eta\rangle\\
\langle U\xi, \eta\rangle &=& \langle \xi, U\eta\rangle.
\end{eqnarray*}
\end{lemma}

\begin{proof} This is much like the proof of the corresponding fact 
in untwisted Floer homology \cite{OS3}, but we must be more careful 
with the coefficients. Observe that composition with the reflection $r: D^2\to 
D^2$ across the real axis gives a map $\pi_2(\x,\y)\to \pi_2(\y,\x)$ 
that exchanges $J$-holomorphic disks in $Sym^g(\Sigma)$ 
with $-J$-holomorphic disks in $Sym^g(-\Sigma)$; in other words
\[
{\mathcal M}_{-\Sigma}(\phi\circ r)  = {\mathcal M}_{\Sigma}(\phi)
\]
for $\phi\in\pi_2(\x,\y)$. 

Furthermore, if $A_Y$ is an additive assignment for 
$(\Sigma,\balpha,\bbeta,z)$ then we can think of $A_Y$ as also giving 
an additive assignment $A_{-Y}$ for $(-\Sigma,\balpha,\bbeta,z)$. For 
$\phi\in\pi_2(\x,\y)$ we have that $\phi * (\phi\circ r)$ is 
homotopic to a constant map, from which it follows that 
\[
A_{-Y}(\phi\circ r) = -A_{Y}(\phi).
\]
Since $n_z^\Sigma(\phi) = n_z^{-\Sigma}(\phi\circ r)$, we have
\begin{eqnarray*}
\partial^{\infty}[\y,j] &=& \sum_{\phi\in\pi_2(\y,\w)} \#\Mhat_{-\Sigma}(\phi) [\w,j-n_z^{-\Sigma}(\phi)]\otimes e^{A_{-Y}(\phi)}\\
&=& \sum_{\tilde{\phi}\in\pi_2(\w,\y)} \#\Mhat_{-\Sigma}(\tilde{\phi}\circ r)[\w, 
j-n_z^{-\Sigma}(\tilde{\phi}\circ r)]\otimes e^{A_{-Y}(\tilde{\phi}\circ r)}\\
&=& \sum_{\tilde{\phi}\in\pi_2(\w,\y)} \#\Mhat_{\Sigma}(\tilde{\phi}) [\w,j-n_z^{\Sigma}(\tilde{\phi})]\otimes 
e^{-A_Y(\tilde{\phi})}.
\end{eqnarray*}
From this it follows (using the conjugate module structure on the 
second factor) that
\begin{eqnarray*}
\langle [\x,i],\partial^\infty[\y,j]\rangle &=& 
\sum_{\begin{array}{c}\mbox{\scriptsize $\phi\in\pi_2(\x,\y)$} \\ 
\mbox{\scriptsize $\mu_\Sigma(\phi) = 1$} \\ \mbox{\scriptsize $n_z^\Sigma(\phi) = 
i + j + 1$}\end{array}} \#\Mhat(\phi) e^{A_Y(\phi)}\\
&=& \langle \partial^\infty[\x,i],[\y,j]\rangle.
\end{eqnarray*}
The first claim of the lemma follows from this, while the second is 
obvious.
\end{proof}

Thus we obtain a pairing on homology
\[
HF^+(Y,\s;R_Y)\otimes_{R_Y} \overline{HF^-(-Y,\s;R_{-Y})}\longto R_Y
\]
that descends to the reduced homologies.

More generally, suppose $M$ and $N$ are (graded) modules for  $R_Y$ and $R_{-Y}$, respectively: we can 
extend the construction above to a pairing between $HF^+(Y,\s;M)$ and 
$HF^-(-Y,\s;N)$. To this end, define
\[
\langle\cdot,\cdot\rangle : CF^\infty(Y,\s;M)\otimes_{R_Y} 
\overline{CF^\infty(-Y,\s;N)}\to M\otimes_{R_Y}\overline{N}
\]
on generators by
\[
\langle [\x,i]\otimes m, [\y,j]\otimes n\rangle = \langle 
[\x,i],[\y,j]\rangle\cdot m\otimes n,
\]
where the pairing on the right is the universal one just defined. It follows 
from the calculation above that the 
pairing descends to homology:
\[
HF^+(Y,\s;M)\otimes_{R_Y}\overline{HF^-(-Y,\s;N)}\longto M\otimes_{R_Y}\overline{N}.
\]

We can now give the analogue for twisted coefficients of Theorem 3.5 
of \cite{OS3}.

\begin{theorem}[Duality for twisted coefficients]\label{dualitythm} Let $W: Y_1\to Y_2$ 
be a cobordism and $M_1$ and $M_2$ coefficient modules for $R_{Y_1}$ and 
$R_{-Y_2}$ respectively. Write $W'$ for the manifold $W$ regarded as a 
cobordism $-Y_2\to -Y_1$, and let $\s$ be a \spinc structure on $W$ 
with restrictions $\s_i = \s|_{Y_i}$. Then for any $\xi\in 
HF^+(Y_1,\s_1;M_1)$ and $\eta\in HF^-(-Y_2,\s_2;M_2)$, we have
\[
\langle F^+_{W,\s}(\xi),\eta\rangle = \langle \xi, 
F^-_{W',\s}(\eta)\rangle.
\]
\end{theorem}

Observe that the two pairings in the theorem above take values in
\[
M_1(W)\otimes_{R_{Y_2}}\overline{M_2} = \Mbar_1\otimes_{R_{-Y_1}}\zee[K(W)]\otimes_{R_{Y_2}}\Mbar_2
\]
(for the right hand side) and
\[
M_1\otimes_{R_{Y_1}}\overline{M_2(W)} = M_1\otimes_{R_{Y_1}}\overline{\zee[K(W)]\otimes_{R_{Y_2}} \Mbar_2} =  M_1\otimes_{R_{Y_1}}\overline{\zee[K(W)]}\otimes_{R_{-Y_2}} M_2,
\]
(for the left). Thus the two target groups are identical (with conjugate module structures), and the equality of the theorem makes sense.

\begin{proof} We adapt the proof from \cite{OS3}. Decompose $W$ into a composition of $1$-handle 
additions, followed by $2$-handles and then $3$-handles. The 
verification of duality for 1- and 3-handle cobordisms is unchanged 
from the untwisted case given in \cite{OS3}, so we omit it here. 

Assume, then, that $W$ is a cobordism comprised entirely of 2-handle 
additions. Let $R$ denote the reflection of the standard 2-simplex 
$\Delta$ that fixes one corner and exchanges the other two. 
Specifically, if the edges are labeled $e_\alpha$, $e_\beta$ and 
$e_\gamma$, we take $R$ to exchange $e_\beta$ and $e_\gamma$ while 
reversing $e_\alpha$. If $A_W$ is an additive assignment for a Heegaard triple 
$(\Sigma,\balpha,\bbeta,\bgamma,z)$ associated to $W$ as in 
Definition \ref{assndef} (using a base triangle $\psi_0$), then we 
obtain an additive assignment $A_{W'}$ for $W'$ (described by the triple $(-\Sigma, \balpha,\bgamma,\bbeta,z)$) from the triangle 
$\psi_0\circ R$. 

More generally, for any (homotopy class of) triangle 
$\psi\in\pi_2(\x,\y,\w)$, precomposition with $R$ gives a triangle 
$\psi\circ R\in\pi_2(\w,\y,\x)$. Moreover, if $\psi = \psi_0 + 
\phi_{\alpha\beta} + \phi_{\beta\gamma} + \phi_{\alpha\gamma}$ then it 
is easy to see that 
\[
\psi \circ R = \psi_0\circ R + (\phi_{\alpha\gamma}\circ r) + 
\phi_{\gamma\beta} + (\phi_{\alpha\beta}\circ r),
\]
where $r$ is the reflection across the real axis used previously. 
Therefore
\[
A_{W'}(\psi\circ R) = \delta(-A_{-Y_2}(\phi_{\alpha\gamma}\circ r) + 
A_{-Y_1}(\phi_{\alpha\beta}\circ r)) = A_W(\psi)
\]
(c.f. the proof of Lemma \ref{pairinglemma}). Furthermore, just as in 
the case of disks we have an identification
\[
{\mathcal M}_{-\Sigma}(\psi\circ R) = {\mathcal M}_{\Sigma}(\psi).
\]
Thus for $m_i\in M_i$: 
\begin{eqnarray*}
\langle F_{W,\s}([\x,i]m_1), [\w,k]m_2\rangle &=& 
\left\langle\sum_{\V\in T_\alpha\cap T_\gamma \atop \psi\in\pi_2(\x,\Theta,\V)} \#{\mathcal 
M}_{\Sigma}(\psi)\cdot [\V, i - n_z(\psi)]m_1\otimes e^{A_W(\psi)}, 
[\w,k]m_2 \right\rangle \\
&=& \sum_{\psi\in\pi_2(\x,\Theta,\w)} \#{\mathcal M}_{\Sigma}(\psi)\cdot 
(m_1\otimes e^{A_W(\psi)})\otimes m_2,
\end{eqnarray*}
an element of $M_1(W)\otimes \overline{M}_2$. On the other hand,
\begin{eqnarray*}
\langle [\x,i]m_1, F_{W',\s}([\w,k]m_2)\rangle &=& \left\langle [\x,i]m_1, 
\sum_{\V\in T_\alpha\cap T_\beta \atop\tilde{\psi}\in \pi_2(\w,\Theta,\V)} \#{\mathcal 
M}_{-\Sigma}(\tilde{\psi})\cdot [\V, i-n_z(\tilde{\psi})]m_2\otimes 
e^{A_{W'}(\tilde{\psi})}\right\rangle\\
&=& \sum_{\tilde{\psi}\in \pi_2(\w,\Theta,\x)} \#{\mathcal 
M}_{-\Sigma}(\tilde{\psi})\cdot m_1\otimes (m_2\otimes 
e^{A_{W'}(\tilde{\psi})})\\
&=& \sum_{\psi\in\pi_2(\x,\Theta,\w)}\#{\mathcal 
M}_{-\Sigma}(\psi\circ R)\cdot m_1\otimes (m_2\otimes e^{A_{W'}(\psi\circ 
R)})\\
&=& \sum_{\psi\in\pi_2(\x,\Theta,\w)} \#{\mathcal M}_{\Sigma}(\psi) 
\cdot
m_1\otimes(m_2\otimes e^{A_W(\psi)})
\end{eqnarray*}
in $M_1\otimes \overline{M_2(W)}$.\end{proof}



\section{Action of First Homology}\label{H1actionsec}

In this section we extend to twisted coefficients an additional aspect 
of the algebraic structure of Heegaard Floer homology, namely the 
action of $\Lambda^*(H_1(Y)/tors)$ on $HF^\circ(Y,\s)$. We also 
discuss the interaction of this structure with cobordism-induced 
homomorphisms. Much of this section is a straightforward 
generalization of material from \cite{OS1,OS2,OS3}, so we omit many of 
the details. 

\begin{prop} Fix an oriented \spinc 3-manifold $(Y,\s)$ and a 
module $M$ for $R_Y = \zee[H^1(Y)]$. Then for any $h \in 
H_1(Y)/tors$ there is a chain endomorphism ${\mathcal A}_h$ of 
$CF^\infty(Y,\s;M)$ of degree $-1$, equivariant with respect 
to $U$ and the $R_Y$ action, with the property that 
${\mathcal A}_h\circ {\mathcal A}_h$ is chain homotopic to 0. 

Thus, the collection of maps ${\mathcal A}_h$ provides $HF^\circ(Y,\s;M)$ with the 
structure of a module over $R_Y[U]\otimes \Lambda^*(H_1(Y)/tors)$. 
\end{prop}

\begin{proof} For a generator $[\x,i]\otimes m\in CF^\infty(Y,\s; M)$ we set
\[
{\mathcal A}_h([\x,i]\otimes m) = \sum_{\y\in T_\alpha\cap T_\beta}
\sum_{\begin{array}{c}\mbox{\scriptsize 
$\phi\in\pi_2(\x,\y)$}\\\mbox{\scriptsize $\mu(\phi) = 
1$}\end{array}} \#\widehat{\mathcal M}(\phi)\langle 
A(\phi),h\rangle\cdot [\y,i-n_z(\phi)]\otimes e^{A(\phi)}\cdot m.
\]
Then the proof that 
${\mathcal A}_h$ is a chain map whose square is trivial 
in homology is virtually identical to the proof in the untwisted 
case (c.f. Proposition 4.17 of \cite{OS1}), and it is straightforward to check that the action of ${\mathcal A}_h$ on homology is independent of the choice of additive assignment $A$ (c.f. the proof of Lemma \ref{addassnlemma}). 
\end{proof}

We will omit the map ${\mathcal A}_h$ from the notation and simply write 
$h.\xi$ for the action of $h$ on the element $\xi\in HF^\circ(Y,\s;M)$.

\begin{remark} Though the action of $H_1(Y)/tors$ is defined for 
Floer homology with any coefficients, it may be largely trivial. 
Indeed, suppose $M$ is an $R_Y$-module, and let $Z_M\subset H^1(Y)$ 
denote the stabilizer of $M$: that is, the set of all $\alpha\in H^1(Y)$ such that $\alpha m = m$ for all $m\in M$. Then it can be shown that if $h\in H_1(Y)$ has the 
property that
\[
\langle\alpha,h\rangle = 0 \quad\mbox{for all $\alpha\in Z_M$}
\]
then ${\mathcal A}_h$ is chain homotopic to 0. In particular, this implies 
that the $H_1(Y)/tors$ action on the fully twisted homology 
$HF^\circ(Y,\s; R_Y)$ is trivial.
\end{remark}

\begin{lemma}\label{skewaction} Let $(Y,\s)$ be as above, and let $M$ and $N$ be 
modules for $R_Y$ and $R_{-Y}$ respectively. Then for any $h\in H_1(Y)/tors$, any 
$\xi\in HF^+(Y,\s;M)$ and any $\eta\in HF^-(-Y,\s;N)$ we have
\[
\langle h.\xi,\, \eta\rangle = - \langle \xi,\, h.\eta\rangle.
\]
\end{lemma}

\begin{proof} This follows from a calculation very similar to the one 
in Lemma \ref{pairinglemma}. Indeed, the only difference is the 
appearance of the factors $\langle A(\phi),h\rangle$, which change 
sign under orientation reversal.
\end{proof}

We now extend the twisted cobordism invariants from the previous 
section to include the action of first homology. Specifically, for a 
cobordism $W:Y_0\to Y_1$ we 
wish to define $F_{W,\s}$ as a map
\begin{equation}\label{extendedhom}
F^\circ_{W,\s}: HF^\circ(Y_0,\s_0; M)\otimes \Lambda^*H_1(W)/tors 
\longto HF^\circ(Y_1,\s_1; M(W)).
\end{equation}

With the preceding 
in place the definition runs precisely as in the untwisted case in \cite{OS3}; we 
summarize the construction.

Suppose first that $W: Y_0\to Y_1$ is a cobordism consisting only of 
2-handle additions. Then it is easy to see that the map 
\[
i_* = {i_0}_* - {i_1}_* : H_1(Y_0)/tors \oplus H_1(Y_1)/tors \to 
H_1(W)/tors
\]
is surjective. Fix $h\in H_1(W)/tors$ and suppose $h = 
i_*(h_0,h_1)$. For $\xi\in HF^\circ(Y_0,\s_0;M)$, we set
\begin{equation}\label{commutatordef}
F^\circ_{W,\s}(\xi\otimes h) = F^\circ_{W,\s}(h_0.\xi) - 
h_1.F^\circ_{W,\s}(\xi).
\end{equation}
Clearly $F^\circ_{W,\s}((\xi\otimes h)\otimes h) = 0$, so the action 
extends to $\Lambda^*H_1(W)/tors$. 

In fact, we can define this action using a Heegaard triple 
$(\Sigma,\balpha,\bbeta,\bgamma,z)$ describing the cobordism, just 
as in Lemma 2.6 of \cite{OS3}. It follows as in that proof that the 
action of pairs $(h_0,h_1)$ in the image of $H_2(W,\partial W,\zee)$ 
is trivial, so the action descends as claimed to $H_1(W)/tors$. 

In general for a cobordism containing 1-, 2-, and 3-handles we write 
the induced homomorphism as a composition $F_W^\circ = E^\circ\circ 
H^\circ\circ G^\circ$ as in section \ref{cobordismsec}. This 
composition corresponds to a factorization $W = W_1\cup W_2\cup W_3$ 
where $W_i$ includes only handles of index $i$. As observed in 
\cite{OS3}, the inclusion induces an isomorphism $H_1(W_2)\to 
H_1(W)$; thus for $\omega\in \Lambda^*H_1(W)/tors$ we set 
\[
F_W^\circ(\xi\otimes \omega) = E^\circ(H^\circ(G^\circ(\xi)\otimes \omega))
\]
just as in \cite{OS3}. 

Many properties of the extended cobordism maps \eqref{extendedhom} 
follow from corresponding properties of the original ones. We mention 
two results here.

\begin{theorem}\label{gendualitythm} Let $W: Y_0\to Y_1$ be a cobordism with \spinc 
structure $\s$ and suppose $\omega\in\Lambda^* H_1(W)/tors$. Write $\s_i$ for 
$\s|_{Y_i}$ Then for modules $M$ and $N$ over $R_{Y_0}$ and $R_{-Y_1}$ 
respectively, and for any $x\in HF^+(Y_0,\s_0;M)$ and $y\in 
HF^-(-Y_1,\s_1;N)$, we have
\[
\langle F^+_{W,\s}(x\otimes \omega), \,y\rangle = \langle x, \, 
F^-_{W',\s}(y\otimes \omega)\rangle.
\]
\end{theorem}

\begin{proof} Assume first that $W$ consists of 2-handles only, and 
suppose $h\in H_1(W)/tors$ has the expression $h = i_*(h_0, h_1)$ for 
$h_i\in H_1(Y_i)/tors$. Then using the duality theorem for 
twisted coefficients (Theorem \ref{dualitythm}) and Lemma 
\ref{skewaction} we have
\begin{eqnarray*}
\langle F^+_{W,\s}(x\otimes h), \,y\rangle &=& \langle 
F^+_{W,\s}(h_0.x) - h_1.F^+_{W,\s}(x), y\rangle\\
&=& - \langle x, h_0.F^-_{W',\s}(y)\rangle + \langle x, 
\,F^-_{W',\s}(h_1.y)\rangle\\
&=& \langle x,\, F^-_{W',\s}(y\otimes h)\rangle.
\end{eqnarray*}
It is a simple matter to extend to general cobordisms and general $\omega$.
\end{proof}

\begin{theorem}\label{gencomplaw} The composition law (Theorem \ref{complaw}) holds for the extended maps 
\eqref{extendedhom}. More precisely, suppose $W = W_1\cup_{Y_1} W_2$ is a 
composite cobordism and write 
\[
j_*: \Lambda^* (H_1(W_1)/tors) \otimes 
\Lambda^* (H_1(W_2)/tors)\to \Lambda^*(H_1(W)/tors)
\]
for the surjection 
induced on exterior algebras by the Mayer-Vietoris map $H_1(W_1)\oplus 
H_1(W_2)\to H_1(W)$. Fix $\omega_i\in \Lambda^*H_1(W_i)/tors$ and 
write $\omega$ for the image of 
$\omega_1\otimes\omega_2$ under $j_*$. 
Then for any \spinc structure $\s$ on $W$, we can find choices of representatives for 
the maps $F^\circ$ such that for any $\alpha\in H^1(Y_1)$
\[
F_{W,\s-\delta \alpha}^\circ(\xi\otimes \omega) = \Pi_W\left[
F_{W_2,\s|_{W_2}}^\circ(e^\alpha\cdot F_{W_1,\s|_{W_1}}^\circ(\xi\otimes \omega_1)\otimes 
\omega_2)\right].
\]
\end{theorem}

\begin{proof} This follows from Theorem \ref{complaw} together with 
the formal properties of the $H_1$-action, particularly 
\eqref{commutatordef} in the case of 2-handles. (See \cite{OS3}, 
particularly Proposition 4.20. Note that here the strengthened 
composition law means that summing over \spinc structures is 
unnecessary.)
\end{proof}

\section{Conjugation and orientation reversal}\label{conjsec}

As in the original Heegaard Floer theory, there are simple relationships between the twisted Heegaard Floer homologies of $(Y,\s)$, $(-Y,\s)$, and $(Y, \bar{\s})$, where $\bar{\s}$ is the conjugate \spinc structure. To describe the effect of \spinc conjugation, recall that though we normally do not include it in the notation, the ring $R_Y$ depends on $\s$ through the grading \eqref{sgrading}, and here we write $R_{Y,\s}$ to indicate this. Thus $R_{Y,\s}$ and $R_{Y,\bar{\s}}$ are identical rings with opposite gradings; in fact $R_{Y,\bar{\s}} = R_{-Y,\s}$ as graded rings. In particular, if $M$ is a graded module for $R_{Y,\s}$, the conjugate module $\Mbar$ can be considered either as a module for $R_{-Y,\s}$ or for $R_{Y,\bar{\s}}$.

\begin{theorem} If $(Y,\s)$ is a closed \spinc 3-manifold and $M$ is a module for $R_{Y,\s}$, then there is a grading-preserving isomorphism of $R_{Y,\s}$-modules
\[
J: HF^\circ(Y,\s; M) \rTo^\sim \overline{HF^\circ(Y,\bar{\s};\Mbar)}.
\]
\end{theorem}

\begin{proof} We mimic the argument in the untwisted case \cite{OS2}. Recall that if $(\Sigma,\balpha,\bbeta,z)$ is a Heegaard diagram for $Y$ and $\x\in T_\alpha\cap T_\beta$ has $s_z(\x) = \s$, then $(-\Sigma,\bbeta,\balpha,z)$ also describes $Y$, and in this diagram $s_z(\x) =\bar{\s}$. If $\{A_{\x,\y}\}$ is an additive assignment for $(\Sigma,\balpha,\bbeta,z)$, then we obtain an assignment $A'$ for $(-\Sigma,\bbeta,\balpha,z)$ by $A'_{\x,\y}(\phi) = -A_{\x,\y}(\phi\circ f)$, where $f$ is the reflection across the imaginary axis in $\cee$.

Define a homomorphism $J: CF^\circ(Y,\s; R_{Y,\s})\to CF^\circ(Y,\bar{\s};R_{Y,\bar{\s}})$ by mapping $[\x,i]$ in the diagram $(\Sigma,\balpha,\bbeta,z)$ to $[\x,i]$ in the diagram $(-\Sigma, \bbeta,\balpha,z)$ and extending by $\zee[H^1(Y)]$-antilinearity. Then it is a straightforward exercise to check that $J$ is a chain map preserving relative gradings, recalling that $\M_{-\Sigma}(\phi\circ f) = \M_{\Sigma}(\phi)$, $n_z^{-\Sigma}(\phi\circ f) = n_z^\Sigma(\phi)$, and $\mu_{-\Sigma}(\phi\circ f) = \mu_\Sigma(\phi)$. In general, 
\[
J: CF^\circ(Y,\s)\otimes_{R_{Y,\s}}M \to CF^\circ(Y,\bar{\s})\otimes_{R_{Y,\bar{\s}}} \Mbar
\]
is given by $[\x, i]\otimes m \mapsto [\x,i]\otimes m$. Since this is an antilinear chain isomorphism, the statement of the theorem follows.
\end{proof}

It is not hard to generalize the the naturality of cobordism-induced maps under conjugation to the twisted case.

Before describing the effect of orientation reversal, we pause to spell out our duality conventions. Let $M$ be a graded $R_{Y,\s}$-module, and set
\[
CF^*_\circ(Y,\s;M) = \Hom_{R_{Y,\s}}(CF_*^\circ(Y,\s), M),
\]
made into an $R_{Y,\s}$-module in the obvious way. For the grading, suppose $\alpha,\beta\in CF^*_\circ(Y,\s;M)$ are homogeneous (as homomorphisms between relatively graded $R_{Y,\s}$-modules). Set
\[
\gr_{CF^*}(\alpha,\beta) = \gr_M(\alpha(f)) - \gr_M(\beta(g)) - \gr_{CF_*}(f,g)
\]
for any homogeneous $f,g\in CF_*(Y,\s)$ with $\alpha(f)$ and $\beta(f)$ nonzero in $M$. Thus, for example, $\gr_{CF^*}(r\alpha,\alpha) = \gr_\s(r)$ for $r\in R_{Y,\s}$.

Observe that there is a natural generating set for $CF^*_\infty(Y,\s;R_{Y,\s})$. Namely, for a generator $[\x,i]\in CF_*^\infty(Y,\s)$, define $[\x,i]^*: CF_*^\infty(Y,\s)\to R_{Y,\s}$ by setting $[\x,i]^*([\y,j]\otimes r) = r$ if $[\y,j] = [\x,i]$, and $0$ otherwise. Since $CF_*^\infty(Y,\s)$ is a free complex over $R_{Y,\s}$, elements of $CF^*_\infty(Y,\s;M)$ can be expressed as combinations of elements of the form $[\x,i]^*\otimes m$, whose value on $[\y,j]$ is $[\x,i]^*([\y,j])\cdot m$. 

In terms of these generators, the coboundary in $CF^*$ can be expressed explicitly by
\[
\delta([\x,i]^*\otimes m) = \sum_{\phi\in \pi_2(\y,\x)\atop \mu(\phi) =1} \#\Mhat(\phi) [\y,i + n_z(\phi)]^* \otimes e^{A(\phi)}m.
\]
With the grading conventions outlined above, we have
\[
\gr_{CF^*(Y)}([\x,i]^*,[\y,j]^*) = -\gr_{CF_*(Y)}([\x,i],[\y,j]).
\]
Observe that with these conventions, the codifferential has degree $-1$, i.e., $CF^*_\infty(Y,\s;M)$ is a chain complex rather than a cochain complex. Likewise, the transpose action of $U$ given by $U: [\x,i]^*\mapsto [\x,i+1]^*$ decreases grading by 2.

\begin{theorem} For $(Y,\s)$ a closed \spinc 3-manifold and $M$ a module for $R_{Y,\s}$, there is a grading-preserving isomorphism of $R_{Y,\s}$-modules
\[
HF^\pm_*(Y,\s; M) \cong \overline{HF^*_\mp(-Y,\s; \Mbar)}.
\]
\end{theorem}

\begin{proof} Just as in the proof in \cite{OS2}, define a homomorphism $CF^\circ_*(Y,\s; M)\to CF^*_\circ(-Y,\s; \Mbar)$ by $[\x,i]\otimes m\mapsto [\x,-1-i]^*\otimes m$, where on the right we consider $m$ as an element of $\Mbar$. One checks easily that this gives rise to a $\zee[H^1(Y)]$-antilinear chain isomorphism that preserves relative grading.
\end{proof}

\section{Invariants for 4-manifolds}\label{gluingsec}. 

We briefly recall the definition of Ozsv\'ath-Szab\'o 4-manifold 
invariants from \cite{OS3}, and then proceed to discuss their calculation in the 
context of 4-manifolds obtained by gluing two manifolds with boundary.

Suppose $X$ is a closed 4-manifold having $b^+(X)\geq 2$. Then we can 
find an {\em admissible cut} for $X$: that is, a hypersurface 
$N\subset X$ separating $X$ into components $X = V_1\cup_N V_2$ with 
the following properties:
\begin{enumerate}
\item For $i = 1,2$, we have $b^+(V_i)\geq 1$.
\item The image of the Mayer-Vietoris map $\delta: H^1(N)\to H^2(X)$ 
is trivial.
\end{enumerate}
As observed previously (Remark \ref{spincrestriction}), the second 
condition ensures that \spinc structures on $X$ are determined by 
their restrictions to $V_1$ and $V_2$. 

The first condition is relevant because of the following.

\begin{lemma}[\cite{OS3}]\label{b+lemma} If $W$ is a cobordism having $b^+(W)\geq 1$ 
then for any \spinc structure $\s$ and in any coefficient module, the 
map $F_{W,\s}^\infty$ vanishes.
\end{lemma}

Recall that for all sufficiently large integers $r$, the subgroups 
$\ker(U_-^r)\subset HF^-(Y,\s)$ and $\imm(U_+^r)\subset HF^+(Y,\s)$ are 
independent of $r$ (where $U_\pm$ denotes the action of $U$ on 
$HF^\pm$). The {\em reduced Floer homology groups} are 
defined by $HF^-_{red}(Y,\s) = \ker(U_-^r)$ and $HF^+_{red}(Y,\s) = 
\coker(U_+^r)$. From the long exact sequence
\begin{diagram}
\cdots\longto HF^\infty(Y,\s)\longto  
HF^+(Y,\s) \stackrel{\tau}{\longto} HF^-(Y,\s)\longto\cdots
\end{diagram}
and the fact that $U$ is an isomorphism on $HF^\infty$ 
we see that Lemma \ref{b+lemma} implies that the image of 
$F^-_{W,\s}$ for $W$ a cobordism with $b^+(W)\geq 1$ lies in 
$HF^-_{red}$, while $F^+_{W,\s}$ factors through $HF^+_{red}$. Note 
also that the homomorphism $\tau$ in the sequence induces an 
isomorphism
\[
\tau: HF^+_{red}(Y,\s)\to HF^-_{red}(Y,\s).
\]
(All of the above holds in any coefficient system).

\begin{definition}[\cite{OS3}] Let $\Theta^-$ denote a top-degree 
generator of $HF^-(S^3)$. Let $N$ be an admissible cut for a 4-manifold $X$ as 
above, and fix a \spinc structure $\s$ on $X$. The {\em 
Ozsv\'ath-Szab\'o invariant} of $(X,\s)$ is the integer-valued function
\[
\Phi_{X,\s}: \A(X) := \zee[U]\otimes 
\Lambda^*(H_1(X)/tors)\longto \zee/\pm 1
\]
defined by
\[
\Phi_{X,\s}(U^n\otimes \omega) = \langle (F^+_{V_2}\circ \tau^{-1}\circ 
F^-_{V_1})(U^n\cdot\Theta^-\otimes \omega), \Theta^-\rangle.
\]
\end{definition}

Note that $\Phi_{X,\s}$ is defined only modulo a sign, due to the 
sign ambiguity of the maps associated to cobordisms.
\begin{remark}As a slight abuse of notation, if $Z$ is a 4-manifold with one 
boundary component $Y$, and $\s$ is a \spinc structure on $Z$, we will 
denote by $F^\circ_{Z,\s}$ the homomorphism $HF^\circ(S^3)\to 
HF^\circ(Y)$ induced by the cobordism obtained by removing a 4-ball 
from the interior of $Z$. 
\end{remark}
\begin{remark} It follows from the formula for the degree shift 
induced by a cobordism that $\Phi_{X,\s}$ is nonzero only on elements 
of $\A(X)$ having degree $d(\s)$, where 
\[
d(\s) = \frac{1}{4}(c_1^2(\s) - 2e(X) - 3\sigma(X)).
\]
Here $e(X)$ is the Euler characteristic of $X$ and $\sigma(X)$ is the 
signature, and $\A(X)$ is graded so that $U$ carries degree 2 and elements 
of $H_1(X)/tors$ carry degree 1.
\end{remark}


Ozsv\'ath and Szab\'o show that $\Phi_{X,\s}$ does not depend on the 
choice of admissible cut $N$, and therefore gives an invariant of 
smooth \spinc 4-manifolds with $b^+\geq 2$. An important property of $\Phi_{X,\s}$ is that it is nonzero for at most finitely many \spinc structures $\s$ on $X$. 

In many situations there are 
convenient decompositions $X = Z_1\cup_Y Z_2$, in which $Y$ fails to 
be admissible in the sense above---specifically, condition (2) in the definition of 
admissibility is violated. Ozsv\'ath and Szab\'o prove that one can 
use such a cut to obtain information about sums of invariants of $X$ 
(Lemma 8.8 of \cite{OS3}), but in order to obtain more detailed 
information we must pass to twisted coefficients. 

We express our results in terms of group rings. In 
the situation of cutting $X$ along a 3-manifold $Y$ satisfying (1) 
but not (2) in the definition of admissible cut, the relevant group 
is $K(X,Y) = \ker(H^2(X)\to H^2(Z_1)\oplus H^2(Z_2))$ (c.f. Remark 
\ref{spincrestriction}). For a given $\s\in Spin^c(X)$ and 
$\alpha\in \A(X)$, we would like a way to calculate the element
\begin{equation}\label{gpringexpression}
\sum_{t\in K(X,Y)} \Phi_{X,\s+t}(\alpha) \cdot e^t \in \zee[K(X,Y)]
\end{equation}
in terms of invariants on the manifolds-with-boundary $Z_1$ and $Z_2$. 
Indeed, the invariants of all \spinc structures on $X$ 
can be read from the coefficients of the above expressions for 
various $\s$.

 Since we 
we need to refer to maps in both twisted and untwisted Floer 
homology, in this section we will follow the notation of Ozsv\'ath and Szab\'o and 
write $\ul{F}_W^\circ$ for the map 
in twisted coefficients induced by $W$ and $F_W^\circ$ for the untwisted map. 

\begin{definition} Suppose $Z$ is an oriented 4-manifold with connected 
boundary $\partial Z = Y$ and $\s\in Spin^c(Z)$. Define the {\em relative 
Ozsv\'ath-Szab\'o invariant} $\Psi_{Z,\s}$ of $Z$ to be the 
function
\[
\Psi_{Z,\s}: \A(Z)\longto  HF^-(Y, \s|_Y; \zee[K(Z)])
\]
given by
\[
\Psi_{Z,\s}(U^n\otimes\omega) = 
[\ul{F}^-_{Z,\s}(U^n\cdot\Theta^-\otimes\omega)].
\]
Here the brackets indicate equivalence class under the action of $K(Z)$, 
where $K(Z) = \ker(H^2(Z,Y)\to H^2(Z))$.
\end{definition}

Recall that the twisted-coefficient map $\ul{F}^-_{Z,\s}$ is defined 
only up to the action of $\delta(H^1(\partial Z) ) = K(Z)$. Note also that if $b^+(Z)\geq 1$ then $\Psi_{Z,\s}$ takes values in $HF^-_{red}(Y)$.

The following result gives the central statement of Theorem \ref{genproduct} from the introduction, and shows how to calculate \eqref{gpringexpression} 
in terms of relative invariants.

\begin{theorem}\label{productthm} Let $X$ be a closed 4-manifold with $b^+(X)\geq 2$ and $Y\subset X$ a 
3-dimensional submanifold separating $X$ into components $Z_1$ and 
$Z_2$. Let $\s$ be a \spinc 
structure on $X$ and write $\s_i = \s|_{Z_i}$. Assume that $\Psi_{Z_i,\s_i}$ takes values in $HF^-_{red}$ for $i = 1,2$, and also that $b^+(Z_i) \geq 1$ for at least one of $Z_1$, $Z_2$. Then for any 
$\alpha_i\in \A(Z_i)$ we have
\begin{equation}\label{productformula}
\sum_{t\in K(X,Y)} \Phi_{X,\s+t}(\alpha)\cdot e^t  = \langle 
\tau^{-1}(\Psi_{Z_1,\s_1}(\alpha_1)), \Psi_{Z_2,\s_2}(\alpha_2)\rangle
\end{equation}
as elements of $\zee[K(X,Y)]$, up to sign and multiplication by an 
element of $K(X,Y)$. Here $\alpha$ is the image of 
$\alpha_1\otimes\alpha_2$ under the natural map 
$\A(Z_1)\otimes\A(Z_2)\to \A(X)$. 
\end{theorem}

In the statement of the theorem, we are implicitly choosing 
representatives for $\Psi_{Z_i,\s_i}(\alpha_i)$ and pairing them using the 
twisted-coefficient pairing defined earlier. Lemma 
\ref{coefflemma} shows that the pairing does indeed take values in 
$\zee[K(X,Y)]$, and it follows also that different choices of 
representatives give rise to elements of $\zee[K(X,Y)]$ differing by 
multiplication by
an element of $K(X,Y)$.

The rest of this section is devoted to the proof of Theorem 
\ref{productthm}. For simplicity, we focus on the case $\alpha = 1$ in 
the following; the general case follows by an entirely analogous 
argument with Theorems \ref{gendualitythm} and \ref{gencomplaw} 
replacing Theorems \ref{dualitythm} and \ref{complaw}.

We begin with a few easy preparatory lemmas.

\begin{lemma}\label{lemma1} Fix a \spinc 3-manifold $Y$ and $R_Y$-modules $M$ and $N$. 
Let $\phi: M\to N$ be a module homomorphism, and write $\phi_*: 
HF^\circ(Y;M)\to HF^\circ(Y;N)$ for the induced map in Floer homology. 
Then the following diagram commutes:
\begin{diagram}
{}&\rTo & HF^-(Y;M) & \rTo & HF^\infty(Y;M) & \rTo & HF^+(Y;M) & \rTo &\\
&& \dTo^{\phi_*} & & \dTo^{\phi_*} & & \dTo^{\phi_*} & \\
{}&\rTo & HF^-(Y;N) & \rTo & HF^\infty(Y;N) & \rTo & HF^+(Y;N) & \rTo &
\end{diagram}
In particular, $\phi_*$ descends to a map on reduced homology, and 
commutes with $\tau$ (and $\tau^{-1}$). 
\end{lemma}

\begin{proof} This is clear. \end{proof}

\begin{lemma}\label{lemma2} For $i = 1,2$ let $M_i$ and $N_i$ be modules for $R_Y$ and $R_{-Y}$ respectively, and 
consider homomorphisms $\phi: M_1\to M_2$ and $\psi: N_1\to N_2$. For 
any $\xi\in HF^+(Y; M_1)$ and $\eta\in HF^-(-Y;N_1)$, we have
\[
\langle \phi_*(\xi),\psi_*(\eta)\rangle  =  \phi\otimes \psi 
(\langle\xi,\eta\rangle) \in M_2\otimes_{R_Y} \overline{N}_2
\]
\end{lemma}

\begin{proof} This follows easily from the definitions.\end{proof}

\begin{lemma}\label{lemma3} Suppose $W = W_1\cup_{Y_1} W_2: Y_0\to Y_2$ is a composite 
cobordism, and $\s_1$ and $\s_2$ are \spinc structures on $W_1$ and 
$W_2$ with $\s_1|_{Y_1} = \s_2|_{Y_1}$. If $F^-_{W_1,\s_1}$ has image in $HF^-_{red}(Y_1)$ then for 
any coefficient module $M$ for $Y_0$,
\begin{enumerate}
\item $\imm(F^-_{W_2,\s_2}\circ F^-_{W_1,\s_1}) \subset 
HF^-_{red}(Y_2; M(W_1)(W_2))$, and
\item $\tau^{-1}\circ F^-_{W_2,\s_2}\circ F^-_{W_1,\s_1} = 
F^+_{W_2,\s_2}\circ\tau^{-1}\circ F^-_{W_1,\s_1}$.
\end{enumerate}
\end{lemma}

\begin{proof} (1) is clear from the fact that 
$F^-_{W_2,\s_2}$ maps $HF^-_{red}(Y_1; M(W_1))$ into 
$HF^-_{red}(Y_2;M(W_1)(W_2))$.

Statement (2) follows from the commutative diagram
\begin{diagram}
&&HF^+(Y_1;M(W_1))&\rTo^{F^+_{W_2,\s_2}} & HF^+(Y_2; M(W_1)(W_2))\\
&&\dTo^\tau &&\dTo^\tau\\
HF^-(Y_0; M) &\rTo^{F^-_{W_1,\s_1}} & HF^-(Y_1;M(W_1))& 
\rTo^{F^-_{W_2,\s_2}} & 
HF^-(Y_2; M(W_1)(W_2))
\end{diagram}
together with part (1).
\end{proof}

With these preliminaries in place, we turn our attention to the proof 
of Theorem \ref{productthm} (with $\alpha_1 = \alpha_2 = \alpha = 1$). Thus, let $X = Z_1\cup_Y Z_2$ be as in 
the statement of the theorem, and let us assume $b^+(Z_2)\geq 1$. Then we can find 
an admissible cut $N$ for $X$ contained in $Z_2$ (c.f. the 
construction in example 8.4 of \cite{OS3}). Suppose $X$ is decomposed 
into pieces $V_1$ and $V_2$ along $N$, so that 
\[
X = V_1\cup_N V_2 = Z_1\cup_Y W\cup_N V_2
\]
where $W = V_1\cap Z_2$ is a cobordism $Y\to N$. 

Let us fix a \spinc structure $\s$ on $X$. For simplicity we will 
omit the \spinc structure from the notation for homomorphisms induced 
by cobordisms, but all relevant cobordisms and their boundaries will be equipped with 
\spinc structures obtained by restricting $\s$.

By definition, we have
\begin{eqnarray}
\Phi_{X,\s}(1) &=& \langle F^+_{V_2}\circ \tau^{-1}\circ 
F^-_{V_1}(\Theta^-), \,\Theta^-\rangle\nonumber\\
&=& \langle \tau^{-1}\circ F^-_{V_1}(\Theta^-), \,
F^-_{V_2}(\Theta^-)\rangle\nonumber \\&=& \langle \tau^{-1}\circ\epsilon_*\circ 
\ul{F}^-_{V_1}(\Theta^-),\, \epsilon_*\circ 
\ul{F}^-_{V_2}(\Theta^-)\rangle\nonumber\\
&=& \langle 
\tau^{-1}\circ\ul{F}^-_{V_1}(\Theta^-),\ul{F}^-_{V_2}(\Theta^-)\rangle 
\label{formula1}
\end{eqnarray}
We have passed to twisted coefficients using the remark after 
Theorem \ref{complaw}. The last line uses Lemma \ref{lemma1} and 
the twisted pairing which 
takes values in $\zee[K(X,N)]$. Since $N$ is admissible the group 
$K(X,N)$ is trivial and hence the pairing is $\zee$-valued; the 
homomorphism $\epsilon_*\otimes \epsilon_*$ arising from Lemma 
\ref{lemma2} is the identity here.

According to Theorem \ref{complaw} we can find representatives for the 
maps involved that satisfy
\[
\ul{F}^-_{V_1} = \Pi_{V_1}\circ \ul{F}^-_{W}\circ \ul{F}^-_{Z_1},
\]
where $\Pi_{V_1}$ is the map induced in homology by 
a projection map $\zee[K(V_1,Y)]\to \zee[K(V_1)]$, which we also 
denote by $\Pi_{V_1}$. 
Different choices of representatives for $[\ul{F}^-_{V_1}]$ and the 
other maps differ by the action of $R_N$ on $\zee[K(X,N)] = \zee$, 
which is trivial. Hence we can replace \eqref{formula1} with
\begin{eqnarray}
\Phi_{X,\s}(1) &=& \langle \tau^{-1} \circ\Pi_{V_1}\circ 
\ul{F}^-_{W}\circ \ul{F}^-_{Z_1}(\Theta^-), \, 
\ul{F}^-_{V_2}(\Theta^-)\rangle\nonumber\\
&=& \Pi_{V_1}\otimes 1 \cdot \langle 
\tau^{-1}\circ\ul{F}^-_W\circ\ul{F}^-_{Z_1}(\Theta^-),\, 
\ul{F}^-_{V_2}(\Theta^-)\rangle\label{formula2}
\end{eqnarray}

\begin{lemma} Under the isomorphism
\[
\zee[K(V_1,Y)]\otimes_{R_N}\zee[K(V_2)] \cong 
\zee\left[\frac{K(V_1,Y)\oplus K(V_2)}{H^1(N)}\right],
\]
the map $\Pi_{V_1}\otimes 1$ corresponds to the homomorphism 
$\Pi_\zee$ sending an element of a group ring to the coefficient of 
the identity element.
\end{lemma}

\begin{proof} We have a diagram of identifications
\begin{diagram}
\zee[K(V_1,Y)]\otimes_{R_N}\zee[K(V_2)] & \rTo^{\Pi_{V_1}\otimes 
1} & \zee[K(V_1)]\otimes_{R_N}\zee[K(V_2)] \\
\dTo^= &&\dTo^=\\
\zee\left[\frac{K(V_1,Y)\oplus K(V_2)}{H^1(N)}\right] &\rTo^p & 
\zee\left[\frac{K(V_1)\oplus K(V_2)}{H^1(N)}\right]
\end{diagram}
Again, since $N$ is admissible 
\[
\frac{K(V_1)\oplus K(V_2)}{H^1(N)} = \ker(H^2(X)\to H^2(V_1)\oplus 
H^2(V_2)) = 0.
\]
The projection $p$ is induced by some map 
\[
\frac{K(V_1)\oplus K(V_2)}{H^1(N)}\to \frac{K(V_1,Y)\oplus 
K(V_2)}{H^1(N)},
\]
for which there is only one choice since the domain group is trivial. 
The construction of $p$ from this map proves the claim.
\end{proof}

Returning with this to equation \eqref{formula2}, we have
\begin{eqnarray}
\Phi_{X,\s}(1) &=& \Pi_\zee \langle \tau^{-1}\circ\ul{F}^-_W\circ\ul{F}^-_{Z_1}(\Theta^-),\, 
\ul{F}^-_{V_2}(\Theta^-)\rangle\nonumber\\
&=& \Pi_\zee\langle 
\ul{F}^+_{W}\circ\tau^{-1}\circ\ul{F}^-_{Z_1}(\Theta^-),\, 
F^-_{V_2}(\Theta^-)\rangle\nonumber\\
&=& \Pi_\zee\langle \tau^{-1}\circ \ul{F}^-_{Z_1}(\Theta^-),\, 
\ul{F}^-_{W}\circ\ul{F}^-_{V_2}(\Theta^-)\rangle\label{formula3}
\end{eqnarray}
using Lemma \ref{lemma3} and Theorem \ref{dualitythm}. Note that the pairings above can be thought of as taking values in
\[
\zee\left[\frac{K(Z_1)\oplus K(W)\oplus K(V_2)}{H^1(Y)\oplus 
H^1(N)}\right]
\]
with appropriate grading.

We would like to apply the composition law in \eqref{formula3} to 
replace $\ul{F}^-_{W}\circ\ul{F}^-_{V_2}$ by $\ul{F}^-_{Z_2}$, but we 
are missing a factor of $\Pi_{Z_2}$ required by Theorem \ref{complaw}. 
By commutativity of the square ($*$) in the following diagram, we are 
free to introduce this factor:
\begin{diagram} 
\zee[K(Z_1)]\otimes_{R_Y} \ol{\zee[K(Z_2,N)]} & \rTo^{1\otimes\Pi_{Z_2}} & 
\zee[K(Z_1)]\otimes_{R_Y}\overline{\zee[K(Z_2)]} \\
\dTo^= & (*) & \dTo^=\\
\zee\left[\frac{K(Z_1)\oplus K(Z_2,N)}{H^1(Y)}\right] & \rTo & 
\zee\left[\frac{K(Z_1)\oplus K(Z_2)}{H^1(Y)}\right]\\
\dTo^= & &\dTo^{\Pi_\zee}\\
\zee\left[ \frac{K(Z_1)\oplus K(W)\oplus K(V_2)}{H^1(Y)\oplus 
H^1(N)}\right] & \rTo^{\Pi_\zee} & \zee
\end{diagram}
Indeed, it follows that $\Pi_\zee = \Pi_\zee\circ (1\otimes 
\Pi_{Z_2})$ (after identifying the groups in the column on the left). 
Thus \eqref{formula3} becomes:
\begin{eqnarray*}
\Phi_{X,\s}(1) &=& \Pi_\zee\circ (1\otimes\Pi_{Z_2})\cdot\langle 
\tau^{-1}\circ \ul{F}^-_{Z_1}(\Theta^-),\, 
\ul{F}^-_{W}\circ\ul{F}^-_{V_2}(\Theta^-)\rangle\\
&=& \Pi_\zee\langle \tau^{-1}\circ \ul{F}^-_{Z_1}(\Theta^-), \, 
\Pi_{Z_2}\circ\ul{F}^-_W\circ\ul{F}^-_{V_2}(\Theta^-)\rangle\\
&=& \Pi_\zee\langle \tau^{-1}\circ\ul{F}^-_{Z_1}(\Theta^-),\, 
\ul{F}^-_{Z_2}(\Theta^-)\rangle,
\end{eqnarray*}
after possibly translating by an element of $R_Y$. This verifies the 
``constant coefficient'' of \eqref{productformula}. For the general 
statement, suppose $t = \delta h \in K(X,Y)$. Then since $\s + t = 
\s$ when restricted to $V_2$ we can follow the same steps as above 
(and using the second part of Theorem \ref{complaw}) to see
\begin{eqnarray}
\Phi_{X,\s + t}(1) &=& \langle\tau^{-1}\circ F^-_{V_1,\s- t}(\Theta^-), 
\, F^-_{V_2,\s}(\Theta)\rangle\nonumber\\
&=& \langle\tau^{-1}\circ \ul{F}^-_{V_1,\s - 
t}(\Theta^-),\,\ul{F}^-_{V_2,\s}(\Theta^-)\rangle\nonumber\\
&=& \langle\tau^{-1}\circ\Pi_{V_1}\circ\ul{F}^-_{W,\s}\circ e^{-h}\cdot 
\ul{F}^-_{Z_1,\s}(\Theta), \,\ul{F}^-_{V_2,\s}(\Theta^-)\rangle\label{transeqn}\\
&=& \Pi_\zee[ e^{-h}\cdot\langle\tau^{-1}\circ\ul{F}^-_{Z_1,\s}(\Theta^-), 
\ul{F}^-_{Z_2,\s}(\Theta^-)\rangle ]\nonumber
\end{eqnarray}
where we can use the same representatives for $[\ul{F}^-_{Z_i,\s}]$ as 
before. Since the action of $R_Y$ on $\zee[K(X,Y)]$ is via the 
coboundary, this last expression is exactly the coefficient of 
$e^{\delta h} = e^t$ in 
$\langle\tau^{-1}\circ\ul{F}^-_{Z_1,\s}(\Theta^-), 
\,\ul{F}^-_{Z_2,\s}(\Theta^-)\rangle$. This completes the proof of 
Theorem \ref{productthm}.

\section{Perturbed Heegaard Floer invariants}\label{perturbsec}

The utility of Theorem \ref{productthm} is limited in many practical circumstances by the restriction on $b^+(Z_i)$. In particular, if one wishes to split a 4-manifold along the boundary of a tubular neighborhood of a surface of square 0, it is not obvious whether the assumptions of that theorem are satisfied. In this section we show how to remedy this circumstance by making use of Heegaard Floer homology ``perturbed'' by a 2-dimensional cohomology class $\eta\in H^2(Y;\arr)$. (A version of this theory was mentioned briefly in \cite{OS1}; here we give a rather fuller treatment.) 

\subsection{Definitions and basic properties}
\begin{definition} Fix a closed oriented 3-manifold $Y$ and a class $\eta\in H^2(Y;\arr)$. The {\em Novikov ring} associated to $(Y,\eta)$ is the ring $\R_{Y,\eta}\subset \zee^{H^1(Y;\zee)}$ of $\zee$-valued functions on $H^1(Y;\zee)$ defined by the condition that $f\in \R_{Y,\eta}$ if and only if for each $N\in \zee$, the set $\supp(f) \cap \{a\in H^1(Y) | \langle a\cup \eta, [Y]\rangle < N\}$ is finite.

More concretely, we can think of $\R_{Y,\eta}$ as the collection of formal series
\[
\R_{Y,\eta} = \{\sum_{g\in H^1(Y;\zee)} a_g\cdot g \,|\, a_g\in \zee\}
\]
subject to the condition that for each $N\in \zee$ the set of $g\in H^1(Y)$ with $a_g$ nonzero and $\langle g\cup \eta,[Y]\rangle < N$ is finite.
\end{definition}

The multiplication on $\R_{Y,\eta}$ is the usual convolution product; note that in the case $\eta = 0$ we have $\R_{Y,\eta} = \zee[H^1(Y)]$. Clearly, $\R_{Y,\eta} = \R_{Y, c\eta}$ for any positive constant $c$. Furthermore, $\R_{Y,\eta}$ depends on the orientation of $Y$ in the sense that $\R_{-Y,\eta} = \R_{Y,-\eta}$.

We can now recite the definition of twisted-coefficient Heegaard Floer homology using $\R_{Y,\eta}$ in place of $R_Y$. 

\begin{definition} Let $(Y,\s)$ be a closed oriented \spinc 3-manifold, and let $\eta\in H^2(Y;\arr)$. Endow $\R_{Y,\eta}$ with the $\s$-grading defined by \eqref{sgrading}. Let $(\Sigma,\balpha, \bbeta, z)$ be a marked Heegaard diagram for $Y$, and choose an additive assignment $A$ for the diagram. The {\em $\eta$-perturbed Heegaard Floer complex} is the free $\R_{Y,\eta}$-module $CF^\infty(Y,\s; \R_{Y,\eta})$ generated by pairs $[\x, i]$ where $\x\in T_\alpha\cap T_\beta$ is an intersection point with $s_z(\x) = \s$, equipped with the relative $\zee$ grading defined in \eqref{grdef}.

The boundary operator is given as in Definition \ref{twistcxdef}, where $e^{A(\phi)}$ is interpreted as an element of $\R_{Y,\eta}$. 
\end{definition}

If $(\Sigma, \balpha, \bbeta, z)$ is strongly $\s$-admissible, in the sense of \cite{OS1}, then the definition above obviously yields the Heegaard Floer complex for the unperturbed theory with coefficients in the $R_Y$-module $\R_{Y,\eta}$, i.e., the complex $CF^\circ(Y,\s; R_Y)\otimes_{R_Y}\R_{Y,\eta}$.

In fact, the perturbed complex can be defined with relaxed admissibility hypotheses: if $\eta$ is generic in the sense that the induced map $H^1(Y;\zee)\to \arr$ is injective, weak admissibility suffices to define $HF^\infty(Y,\s; \R_{Y,\eta})$ and $HF^-(Y,\s;\R_{Y,\eta})$, while no admissibility conditions are necessary to define $HF^+(Y,\s;\R_{Y,\eta})$ or $\widehat{HF}(Y,\s;\R_{Y,\eta})$. However, we have no need for this generality, and the observation in the previous paragraph suffices to show that the perturbed Heegaard Floer homology is a topological invariant of $(Y,\s,\eta)$.

Note that if $M$ is a (graded) module for $R_Y$, we can obtain a module for $\R_{Y,\eta}$ by tensor product: $\M_\eta \equiv M\otimes_{R_Y}\R_{Y,\eta} $. Thus we can consider perturbed Heegaard Floer homology with coefficients in the ``completed'' module $\M_\eta$, namely the homology of the complex $CF(Y,\s; \R_{Y,\eta})\otimes_{\R_{Y,\eta}} \M_\eta$ (of course, since any module for $\R_{Y,\eta}$ is also a module for $R_Y$, we see trivially that any $\R_{Y,\eta}$ module is obtained in this way).
 
Calculation of perturbed Floer homology is facilitated by the following.

\begin{lemma}\label{flatlemma}
For any $\eta\in H^2(Y;\arr)$, the ring $\R_{Y,\eta}$ is flat as an $R_Y$-module.
\end{lemma}

\begin{proof} Let $K$ denote the kernel of the homomorphism $H^1(Y;\zee)\to \arr$ given by $x \mapsto \langle x\cup \eta, [Y]\rangle$; note that $K$ is a direct summand of $H^1(Y;\zee)$. Let $\rk(K) = k$. The ring $R_Y$ can be identified with a Laurent polynomial ring in variables $\{x_1,\ldots, x_b\}$, $b = b_1(Y)$,  and we can choose the generators $x_i$ such that $\langle x_i\cup \eta,[Y]\rangle = 0$ for $i = 1,\ldots, k$, while $\langle x_i\cup \eta, [Y]\rangle >0$ for $i > k$. The Novikov ring $\R_{Y,\eta}$ can be constructed as follows. First let $\Z_\eta$ denote the (``partial'') power series ring obtained by completing the ring $\Z = \zee[x_1,\ldots, x_b]$ with respect to the ideal generated by $x_{k+1}, \ldots, x_b$. Then if $V$ denotes the multiplicative subset generated by the variables $x_1,\ldots, x_b$, we have that $\R_{Y,\eta} = V^{-1}\Z_\eta$. It is a standard fact that $\Z_\eta$ is flat over $\Z$ (see, e.g., \cite{eisenbud}), and it follows easily that $\R_{Y,\eta} = V^{-1}\Z_\eta$ is flat over $R_Y = V^{-1}\Z$. \end{proof}

\begin{definition} Let $Y$ be a closed oriented 3-manifold, $\s\in Spin^c(Y)$, and $\eta\in H^2(Y;\arr)$. We say $\eta$ is {\em generic for $\s$} if $\ker(c_1(\s))\not\subset \ker(\eta)$. That is to say, $\eta$ is generic for $\s$ if there exists a class $x\in H^1(Y)$ such that
\[
\eta\cup x \neq 0 \quad\mbox{but}\quad c_1(\s)\cup x = 0.
\]
\end{definition}

Observe that if $c_1(\s)$ is torsion and $\eta$ is nonzero then $\eta$ is automatically generic for $\s$, while if $b_1(Y) = 1$ and $c_1(\s)$ is non-torsion, then a generic class $\eta$ for $\s$ does not exist. Once $b_1(Y)>1$, however, any class $\eta\in H^2(Y;\arr)$ that is ``generic'' in the sense that $\langle\eta\cup \cdot , [Y]\rangle: H^1(Y;\zee)\to \arr$ is injective, is automatically generic for any \spinc structure $\s$.

In Seiberg-Witten theory, once $b_1(Y)>0$ it is possible to ``perturb away'' reducible solutions to the Seiberg-Witten equations on $Y$. The following can be seen as an analog of that statement in Heegaard Floer theory.

\begin{cor}\label{perturbcor} If $\eta\in H^2(Y;\arr)$ is generic for a \spinc structure $\s$ on $Y$, then
\[
HF^\infty(Y,\s;\R_{Y,\eta}) = 0,
\]
and therefore $HF^\infty(Y,\s;{\mathcal M}) = 0$ for any $\R_{Y,\eta}$-module $\mathcal M$. In particular, under this assumption, for any $R_Y$-module $M$ with completion $\M_\eta = M\otimes_{R_Y} \R_{Y,\eta}$ we have isomorphisms
\[
HF^\pm(Y,\s; \M_\eta) = HF^\pm_{red}(Y,\s; \M_\eta) \cong HF^\pm_{red}(Y,\s; M) \otimes_{R_Y} \R_{Y,\eta}.
\]
\end{cor}

\begin{proof} By the previous lemma, $HF^\infty(Y,\s; \R_{Y,\eta}) \cong HF^\infty(Y,\s; R_Y)\otimes_{R_Y}\R_{Y,\eta}$. Oszv\'ath and Szab\'o showed \cite{OS2} that for any 3-manifold $Y$, the fully-twisted Floer homology satisfies
\[
HF^\infty(Y,\s;R_Y) \cong \zee[U,U^{-1}]
\]
where an element $x\in H^1(Y;\zee)$ having $\langle c_1(\s)\cup x, [Y]\rangle = 2k$ acts as multiplication by $U^k$. Take $x$ to be as in the definition of generic above, so that $k = 0$. Without loss of generality we can assume that $\langle \eta \cup x, [Y]\rangle >0$, so that the element $1-x$ has an inverse $\sum_{n\geq 0} x^n$ in $\R_{Y,\eta}$. But then $1-x$ is a unit that acts as 0 on $HF^\infty(Y,\s;\R_{Y,\eta})$, meaning the latter module must vanish. 
The remaining statements follow easily from the flatness of $\R_{Y,\eta}$.
\end{proof}

As noted above, it is not always possible to guarantee the existence of a generic perturbation (namely when $b_1(Y) = 1$). Of more concern for our purposes, a similar situation arises when considering cobordisms $W: Y_1\to Y_2$, where perturbations and \spinc structures on $Y_i$ are taken to be induced from $W$. Here if $b_2(W) = 1$, for example, then any class $\eta$ induced from $W$ must be a multiple of the Chern class of a \spinc structure on $Y_2$ induced from $W$, and again we cannot arrange genericity regardless of the value of $b_1(Y)$. 

To deal with this situation we make a further completion of Heegaard Floer homology, this time with respect to $U$. 

\begin{definition}\label{Ucompldef} Let $\zee[[U]]$ denote the ring of integer power series in $U$. The {\em $U$-completed  Heegaard Floer groups} for $(Y,\s,\eta)$ in a module $\M$ are defined by
\[
HF^\circ_\Dot(Y,\s;\M) = HF^\circ(Y,\s; \M)\otimes_{\zee[U]} \zee[[U]].
\]
\end{definition}

Thus the perturbed, completed Floer homology $HF^\circ_\Dot(Y,\s; \M)$ is a module for $\R_{Y,\s}[[U]]$. Observe that since the action of $U$ is nilpotent on elements of $HF^+$, this completion has no effect on the latter group:
\[
HF^+_{\Dot}(Y,\s;\M) = HF^+(Y,\s;\M).
\]
There is a natural map $HF^\circ(Y,\s; \M)\to HF^\circ_\Dot(Y,\s;\M)$ that is typically (when $\eta$ is generic for $\s$, for example) an injection. We will often implicitly make use of this homomorphism when extending previous results to the $U$-completed setting.

The definition is most useful when the uncompleted group $HF^\circ(Y,\s; \M)$ carries a relative $\zee$ grading (not a cyclic grading). We will generally be interested in coefficient modules $\M$ that arise from cobordisms $S^3\to Y$, and in light of Proposition \ref{Zgradingprop} we will therefore be in the $\zee$-graded case.

\begin{cor}Let $(Y,\s)$ be a closed \spinc 3-manifold and $\eta\in H^2(Y;\arr)$ a fixed class. If $c_1(\s)$ is torsion, assume that $\eta\neq 0$. Then 
\[
HF^\infty_\Dot (Y,\s; \M_\eta) = 0
\]
for any $\R_{Y,\eta}$ module $\M_\eta$.
\end{cor}

The other conclusions of Corollary \ref{perturbcor} of course follow as well for the $U$-completed Floer homology perturbed by a compatible class $\eta$. Note that if $c_1(\s)$ is non-torsion, then it suffices to take $\eta = 0$.

\begin{proof} It suffices to show the vanishing with coefficients in $\R_{Y,\eta}$; if $c_1(\s)$ is torsion then a nonzero $\eta$ is necessarily generic for $\s$ so that Corollary \ref{perturbcor} applies. Otherwise, we can find  $t\in H^1(Y)$ such that $\langle c_1(\s)\cup t, [Y]\rangle = -2k$ with $k>0$; then as before $t$ acts as multiplication by $U^{-k}$ on $HF^\infty(Y,\s;R_Y)$. Hence the element $1-tU^k$ acts as 0, but the latter is a unit in the completed ring $\R_{Y,\eta}[[U]]$.
\end{proof}

We now wish to extend perturbed Heegaard Floer theory to cobordism-induced homomorphisms. To do so, we again follow the program from the unperturbed case; we need only make sure that the coefficient modules respect the algebraic nature of the Novikov rings. 

\begin{definition} Let $W: Y_1\to Y_2$ be an oriented cobordism between 3-manifolds $Y_i$, and fix $\eta\in H^2(W; \arr)$ with restrictions $\eta_i = \eta|_{Y_i}$. Let $\K(W,\eta)$ be the Novikov completion of $\zee[K(W)]$ with respect to $\eta$, where as usual $K(W) = \imm(H^1(\partial W;\zee)\to H^2(W,\partial W;\zee))$. Concretely, $\K(W,\eta)$ is the ring of formal series
\[
\K(W,\eta) = \{ \sum_{g\in K(W)} a_g \cdot g \,|\, a_g\in\zee\}
\]
subject to the condition that for each $N\in \zee$ the set of $g\in K(W)$ with $a_g$ nonzero and $\langle g\cup \eta, [W,\partial W]\rangle <N$ is finite. Then $\K(W,\eta)$ is a module for both $\R_{Y_1,\eta_1}$ and $\R_{Y_2, \eta_2}$.

If $\M$ is a (graded) module for $\R_{Y_1,\eta_1}$, the module $\M(W,\eta)$ for $\R_{Y_2,\eta_2}$ induced by $(W,\eta)$ is defined by
\[ 
\M(W,\eta) = \overline{\M} \otimes_{\R_{-Y_1,\eta_1}} \K(W,\eta).
\]
\end{definition}

Here $\K(W,\eta)$ can be given an integer grading depending on a choice of \spinc structure just as in Proposition \ref{Zgradingprop}. For the conjugate module appearing in the last statement, observe that the map $x\mapsto -x$ in $H^1(Y)$ induces a conjugation map $\R_{Y,\eta}\to \R_{-Y,\eta}$. Thus $\overline{\M}$, defined to be the same graded group as $\M$ with conjugate module structure, makes sense as a graded $\R_{-Y,\eta}$-module.

It is now straightforward to define a homomorphism
\[
F_{W,\eta,\s}^\circ: HF^\circ_\Dot(Y_1,\s_1,\M)\to HF^\circ_\Dot(Y_2,\s_2,\M(W,\eta))
\]
associated to a \spinc cobordism $(W,\s)$ with chosen perturbation $\eta$ (or similar maps between the groups without the ``$\Dot$''), by making the usual formal construction for 1- and 3-handles, and using Definition \ref{twistmapdef} for the 2-handles, where $e^{A_W(\psi)}$ is considered to lie in $\K(W,\eta)$. The proof that the result of this construction is a chain map whose induced map in homology is an invariant of $W$ (up to a sign and the action of $H^1(Y_1;\zee)\oplus H^1(Y_2;\zee)$) is identical to the proof in the unperturbed case in \cite{OS3}. Alternatively, one can deduce this fact from the corresponding fact in the unperturbed theory using Lemma \ref{flatlemma}.

Similarly, there is a composition law for perturbed cobordism maps that follows from the usual one given in Theorem \ref{complaw}. Indeed, suppose we are given a cobordism $W: Y_1\to Y_2$ and a module $M$ for $R_Y$, along with a class $\eta\in H^2(W;\arr)$. Write $\M = M\otimes_{R_{Y_1}} \R_{Y_1,\eta}$ for the Novikov completion of $M$ (we do not distinguish in the notation between $\eta$ and its restrictions to $Y_1$, $Y_2$); then $\M$ can also be considered as an $R_{Y_1}$-module. As such, we obtain an induced $R_{Y_2}$-module $\M(W) = \overline{\M}\otimes_{R_{-Y_1}}\zee[K(W)]$. It is not hard to see that the $\R_{Y_2,\eta}$-module induced by $(W,\eta)$ is then
\begin{equation}\label{inducedmodule}
\M(W, \eta) = \M(W)\otimes_{R_{Y_2}}\R_{Y_2,\eta} = \R_{-Y_1,\eta}\otimes_{R_{-Y_1}} M(W) \otimes_{R_{Y_2}} \R_{Y_2,\eta}
\end{equation}
and we have a commutative diagram (with or without $\Dot$'s)
\begin{equation}\label{pertdiag}
\begin{diagram}
HF_\Dot(Y_1,\M) & \rTo^{F_W} & HF_\Dot(Y_2, \M(W)) &\rTo^{\cdot\otimes 1} & HF_\Dot(Y_2, \M(W))\otimes_{R_{Y_2}} \R_{Y_2,\eta}\\
&\rdTo(4,2)_{F_{W,\eta}}&&&\dTo\\
&&&& HF_\Dot(Y_2,\M(W,\eta))
\end{diagram}
\end{equation}
where the vertical arrow is an isomorphism according to Lemma \ref{flatlemma}. Combining this observation with the original composition law gives the desired result for perturbed Floer homology. 

All of the algebraic constructions introduced previously for twisted coefficients go through with only minor modifications in the perturbed setup. The action of $H_1(Y;\zee)/tors$ on $HF^\circ_\Dot(Y,\s;\M)$ for $\M$ an $\R_{Y,\eta}$-module is defined just as before, as is the extension of cobordism-induced maps to incorporate this action. Likewise the previous definition applies to give a pairing
\[
HF^+_\Dot(Y,\s;\M)\otimes_{\R_{Y,\eta}} \overline{HF^-_\Dot(-Y,\s; \N)}\to \M\otimes_{\R_{Y,\eta}}\overline{\N}
\]
for any $\R_{Y,\eta}$-module $\M$ and $\R_{-Y,\eta}$-module $\N$ (and similarly without $\Dot$'s).

\subsection{Conjugation and orientation reversal}
The perturbed versions of the results of section \ref{conjsec} are straightforward generalizations, with the caveat that conjugation $R_{Y,\s}\to R_{Y,\bar{\s}}$ extends to the setting of Novikov rings only at the cost of reversing the sign of $\eta$. Indeed, if $r\in \R_{Y,\s,\eta}$, we can consider the conjugate $\bar{r}$ to lie either in $\R_{-Y,\s,\eta}$ or in $\R_{Y,\bar{\s},-\eta}$. Hence if $\M$ is a module for $\R_{Y,\s,\eta}$, we can think of $\overline{\M}$ as a module either for $\R_{-Y,\s,\eta}$ or for $\R_{Y,\bar{\s},-\eta}$.

\begin{theorem}\label{pertconjinvthm} If $(Y,\s)$ is a closed \spinc 3-manifold with class $\eta\in H^2(Y;\arr)$, and $\M$ is a module for $\R_{Y,\eta}$, then we have an isomorphism of $\R_{Y,\eta}$-modules
\[
HF^\circ_\Dot(Y,\s; \M) \cong \overline{HF^\circ_\Dot(Y,\bar{\s}, \overline{\M})}
\]
preserving relative gradings, where $\overline{\M}$ is considered as a module for $\R_{Y,\bar{\s},-\eta}$. In particular,
\[
HF^\circ_\Dot(Y,\s;\R_{Y,\s,\eta}) \cong \overline{HF^\circ_\Dot(Y,\bar{\s},\overline{\R_{Y,\s,\eta}})} \cong \overline{HF^\circ_\Dot(Y,\bar{\s}, \R_{Y,\bar{\s},-\eta})}.
\]
\end{theorem}\hfill$\Box$

Thus in the perturbed case, there is a natural equivalence between Floer homology for \spinc structure $\s$ perturbed by a form $\eta$, and the homology for $\bar{\s}$ perturbed by $-\eta$.

\begin{theorem} For $(Y,\s)$ a closed oriented \spinc 3-manifold with class $\eta\in H^2(Y;\arr)$, and $\M$ a module for $\R_{Y,\eta}$, there is an isomorphism
\[
HF^\pm_\Dot(Y,\s;\M)\cong \overline{HF^\Dot_\mp(-Y,\s; \overline{\M})},
\]
of relatively graded $\R_{Y,\s,\eta}$-modules, where $\overline{\M}$ is taken to be a module for $\R_{-Y,\s,\eta}$.
\end{theorem}\hfill$\Box$

As usual, there are obvious analogues of each of these results before taking $U$-completions. 

\subsection{Perturbed 4-manifold invariants}
We are now in a position to define invariants for closed 4-manifolds using perturbed Floer homology. If $(Z,\s)$ is a \spinc 4-manifold with boundary $Y$ and $\eta\in H^2(Z;\arr)$, we define the {\it perturbed relative invariant} for $(Z,\s, \eta)$ to be the map
\[
\Psi_{Z,\s,\eta}: \A(Z) \to HF^-_\Dot(Y,\s; \K(Z,\eta))
\]
given by $\Psi_{Z,\s,\eta}(U^n\otimes \omega) = [F^-_{Z,\s,\eta}(U^n\cdot \Theta^-\otimes\omega)]$, where the brackets indicate the equivalence class under the action of $K(Z)$ as before. Here $\K(Z,\eta)$ is the $\R_{Y,\eta}$-module induced by $Z$, thought of as a cobordism $S^3\to Y$; in other words $\K(Z,\eta)$ is the Novikov completion of $\zee[\ker(H^2(Z,Y)\to H^2(Z))]$ with respect to $\eta$.

\begin{definition} Let $X$ be a closed $4$-manifold and $\eta\in H^2(X;\arr)$. An oriented 3-dimensional embedded submanifold $Y\subset X$ is an {\em allowable cut for $\eta$} if $Y$ separates $X$ into two components, $X = Z_1\cup_Y Z_2$ with $\partial Z_1 = Y = -\partial Z_2$, and at least one of the following conditions are satisfied:
\begin{enumerate}
\item $\eta|_Y \neq 0$.
\item $b^+(Z_i)\geq 1$ for $i = 1,2$.
\end{enumerate}
\end{definition}

Observe that if property (1) of the definition holds, then it follows from Corollary \ref{perturbcor} that the induced map $F^\infty_{Z,\s,\eta}$ is trivial in perturbed, $U$-completed Floer homology for any \spinc structure on $X$. (Indeed, if the restriction of $c_1(\s)$ to $Y$ is a non-torsion class then we need not even assume (1), but of course in this case one can always find a class $\eta$ satisfying (1), namely the image in real cohomology of $c_1(\s)$. To avoid complicating the statements of results to follow, we ignore this point.) 

On the other hand, if $W: Y_1\to Y_2$ is a cobordism with $b^+(W)>0$ and $\eta\in H^2(W;\arr)$, then since the unperturbed map in $HF^\infty$ induced by $W$ is trivial, the same is true for the map perturbed by $\eta$, whether $\eta$ vanishes on $\partial W$ or not. Hence the perturbed relative invariant $\Psi_{Z,\s,\eta}$, for a component of $X$ arising from a cut allowable for $\eta$, takes values in the reduced Floer homology in both cases, and the following makes sense.

\begin{definition} Let $X$ be a closed oriented 4-manifold and $\s$ a \spinc structure on $X$. For a pair $(Y,\eta)$ consisting of an element $\eta\in H^2(X;\arr)$ and a cut $X = Z_1\cup_Y Z_2$ of $X$ that is allowable for $\eta$, the {\em perturbed {\OS} invariant} of $X$ associated to $(Y,\eta, \s)$ is the linear map ${\mbox{\normalfont\gothfamily O}}_{X,Y,\eta,\s}: \A(X)\to \K(X,Y, \eta)$ defined by
\[
\O_{X,Y,\eta,\s}(\alpha) = \langle \tau^{-1}\Psi_{Z_1,\s, \eta}(\alpha_1), \Psi_{Z_2,\s, \eta}(\alpha_2)\rangle,
\]
up to sign and multiplication by an element of $K(X,Y)$. Here $\alpha$ is the image of $\alpha_1\otimes \alpha_2$ under the natural map $\A(Z_1)\otimes \A(Z_2)\to \A(X)$. 
\end{definition}

In this definition, we set
\[
\K(X,Y,\eta) = \K(Z_1,\eta)\otimes_{\R_{Y,\eta}} \overline{\K(Z_2,\eta)}.
\]
This can be identified with the Novikov completion of the $R_Y$-module $\zee[K(X,Y)]$ with respect to $\eta$ as in Lemma \ref{coefflemma}. That is to say, $\K(X,Y,\eta) = \zee[K(X,Y)]\otimes_{R_Y} \R_{Y,\eta}$. Note that $\O_{X,Y,\eta,\s}$ depends on the orientation of $Y$ in the sense that $\O_{X,Y,\eta,\s} = \overline{\O_{X,-Y,\eta,\s}}$. Indeed, it is easy to see that  the two are related by the action of the obvious anti-homomorphism $\K(Z_1,\eta)\otimes\overline{\K(Z_2,\eta)}\to \K(Z_2,\eta)\otimes\overline{\K(Z_1,\eta)}$, which in turn corresponds to the conjugation homomorphism $\K(X,Y,\eta)\to \K(X,-Y,\eta)$.

We will see that when $b^+(X)\geq 2$, the definition above recovers the ordinary {\OS} invariants in the sense of Theorem \ref{productthm}: that is, $\O_{X,Y,\eta,\s}$ lies in $\zee[K(X,Y)]$, and the coefficients of this group ring element are the {\OS} invariants of $X$ in the various \spinc structures that have given restrictions to $Z_1$ and $Z_2$. The utility of this definition is that we no longer need to assume that $b^+(Z_i)\geq 1$ or even that $F^-_{Z_i,\s}$ takes values in the reduced Floer homology, so long as $\eta|_Y \neq 0$. 

It should be noted, however, that the existence of a 3-manifold $Y$ separating $X$ and a class $\eta\in H^2(X,\arr)$ restricting nontrivially to $Y$ implies that $X$ is indefinite; in particular $b^+(X)\geq 1$.

\begin{lemma}\label{relationlemma} Suppose $X$ is a 4-manifold with $b^+(X)\geq 2$, $Y$ a submanifold splitting $X$ into components $Z_1$ and $Z_2$ with $b^+(Z_i)\geq 1$ (or more generally satisfying the hypotheses of Theorem \ref{productthm}), and $\eta\in H^2(X;\arr)$ a perturbing class on $X$. Then for any \spinc structure $\s$ on $X$, $\O_{X,Y,\eta,\s}$ takes values in $\zee[K(X,Y)]\subset \K(X,Y,\eta)$, and 
\[
\O_{X,Y,\eta, \s} = \sum_{t\in K(X,Y)} \Phi_{X,\s+t}\cdot e^t
\]
up to multiplication by $\pm 1$ and an element of $K(X,Y)$.
\end{lemma}

\begin{proof} We have a commutative diagram
\begin{equation}\label{pertcoborddiag}
\begin{diagram}
HF^-(S^3) & \rTo & HF^-_{red}(Y, \zee[K(Z_1)])\\
\dTo^= & & \dTo_{i_*}\\
HF^-(S^3) &\rTo & HF^-_\Dot(Y, \K(Z_1,\eta))
\end{diagram}
\end{equation}
where the upper arrow is the unperturbed, twisted-coefficient homomorphism induced by $(Z_1,\s)$, and the lower arrow uses the perturbation $\eta|_{Z_1}$. Here $i_*$ is the natural map induced by the homomorphism $\zee[K(Z_1)]\to \K(Z_1,\eta)$ of a ring to its Novikov completion; commutativity of the diagram is obvious from the definition of cobordism-induced maps. We have a similar diagram for $Z_2$ with $Y$ replaced by $-Y$. 

Likewise, there is a diagram
\begin{equation}\label{pertpairingdiag}
\begin{diagram} 
HF^-_{red}(Y, \zee[K(Z_1)]) \otimes_{R_Y} \overline{HF^-_{red}(-Y, \zee[K(Z_2)])} & \rTo^{\langle \tau^{-1}(\cdot), \cdot \rangle} & \zee[K(X,Y)] \\
\dTo^{i_*\otimes i_*} & & \dTo_j\\
HF^-_\Dot(Y, \K(Z_1,\eta)) \otimes_{\R_{Y,\eta}} \overline{HF^-_\Dot(-Y, \K(Z_2, \eta))} & \rTo^{\langle \tau^{-1}(\cdot), \cdot \rangle} & \K(X,Y, \eta)
\end{diagram}
\end{equation}

If $\eta|_Y = 0$, then the maps $i_*$ and $j$ are the identity maps: indeed, it follows from \eqref{inducedmodule} that $\K(Z_i,\eta) = \zee[K(Z_i)]$ and $\K(X,Y,\eta) = \zee[K(X,Y)]$. Thus in this case, the lemma is just a restatement of Theorem \ref{productthm}.

Assume that $\eta|_Y\neq 0$. We claim that the map $j:\zee[K(X,Y)]\to \zee[K(X,Y)]\otimes_{R_{Y}} \R_{Y,\eta}$ is injective, and this combined with Theorem \ref{productthm} clearly implies the lemma. To see the injectivity, observe that since $\eta$ is induced from the 4-manifold $X$, we have a diagram
\begin{diagram}
H^1(Y;\zee) && \rTo^\delta && K(X,Y)\\
& \rdTo_{\langle\eta\cup \cdot, [Y]\rangle} && \ldTo_{\langle\eta\cup \cdot, [X]\rangle}\\
&& \arr 
\end{diagram}
with $\delta$ surjective (recall that $K(X,Y)$ is identified with a subgroup of $H^2(X)$; c.f. Remark \ref{spincrestriction}).
From this it follows that the Novikov completion of $\zee[K(X,Y)]$ as an $R_Y$-module is the same as its Novikov completion $(\zee[K(X,Y)])^\wedge_{\eta}$ as a ring with respect to the function $\langle \eta\cup\cdot,[X]\rangle$. It is straightforward to see that the map $\zee[K(X,Y)]\to (\zee[K(X,Y))^{\wedge}_{\eta}$ from a group ring to its Novikov completion is injective.
\end{proof}

\begin{remark} The injectivity of $j: \zee[K(X,Y)]\to \zee[K(X,Y)]\otimes_{R_{Y}} \R_{Y,\eta}$ from $\zee[K(X,Y)]$ to its (module) completion is not automatic, as can be seen in the case of a cut $Y$ that is admissible in the sense of Ozsv\'ath and Szab\'o: in this case $K(X,Y)$ is trivial, $\zee[K(X,Y)] \cong \zee$, and $\zee\otimes_{R_Y}\R_{Y,\eta} = 0$ if $\eta\neq 0$. 
\end{remark}

To handle cases where not both sides of the cut $Y$ have nonvanishing $b^+$, we consider the dependence of $\O_{X,Y,\eta,\s}$ on $Y$. For our present purposes, the following suffices.

\begin{lemma}\label{cutinvlemma} Suppose $X$ is given with two disjoint cuts $Y_1$ and $Y_2$ that are allowable for a class $\eta\in H^2(X;\arr)$, and whose orientations are compatible. Then for a \spinc structure $\s$ on $X$, $\O_{X,Y_1,\eta, \s}$ and $\O_{X,Y_2,\eta,\s}$ contain the same information in the following sense.

There is a group $K(X; Y_1, Y_2)$ whose group ring $\zee[K(X; Y_1, Y_2)]$ has a Novikov completion $\K(X; Y_1,Y_2,\eta)$ depending on $\eta$. There are injections $K(X,Y_i)\to K(X; Y_1,Y_2)$ inducing surjective restriction homomorphisms $\pi_i: \K(X;Y_1,Y_2,\eta)\to \K(X,Y_i, \eta)$. Finally, there is a homomorphism $\tO:\A(X)\to \K(X;Y_1,Y_2,\eta)$ (depending on $\s$ and defined up to sign and translation by an element of $K(X;Y_1,Y_2)$) such that for appropriate choices of representatives,
\[
\pi_i\circ\tO = \O_{Y_i,\eta,\s}
\]
for $i = 1,2$. In particular, if we think of $\O_{X,Y_i,\eta,\s}$ as Laurent series (depending on elements of $\A(X)$), the coefficients of those series are equal to coefficients of $\tO$.
\end{lemma}

Compatibility of the orientations of $Y_1$ and $Y_2$ means that in the decomposition $X = Z_1\cup_{Y_1} Z_2 \cup_{Y_2} Z_3$, we have $\partial Z_1 = Y_1$, $\partial Z_2 = -Y_1\coprod Y_2$, and $\partial Z_3 = -Y_2$.

Intuitively, $\K(X;Y_1,Y_2,\eta)$ is a Novikov ring coming from a group containing both $K(X,Y_1)$ and $K(X,Y_2)$ as subgroups. One can think of $\tO$ as a polynomial (Laurent series) in a large number of variables, and the perturbed invariants $\O_{Y_i,\eta,\s}$ are ``sections'' of this polynomial. Put another way, $\tO$ is an integer function on a group $K(X;Y_1,Y_2)$, of which $K(X,Y_1)$ and $K(X,Y_2)$ are subgroups. One obtains each of the functions $\O_{X,Y_i,\eta,\s}$ by restriction of $\tO$ to the corresponding subgroup (or its cosets); it is in this sense that the $\O_{X,Y_i,\eta,\s}$ contain the same information.

\begin{proof} Suppose $Y_1$, $Y_2$ divide $X$ into pieces $X = Z_1\cup_{Y_1} Z_2\cup_{Y_2} Z_3$. We prove the lemma by evaluating the invariants coming from the two splittings on $\alpha\in\A(X)$; in fact we assume $\alpha = 1$. The general case is only notationally more difficult. 

According to the composition law, we can find representatives for the maps involved such that
\begin{eqnarray}
\O_{X,Y_1,\eta,\s}&=& \langle\tau^{-1}\Psi_{Z_1,\eta}, \Psi_{Z_2\cup Z_3, \eta}\rangle\nonumber\\
&=& \langle\tau^{-1} F^-_{Z_1,\eta}(\Theta^-), F^-_{Z_2\cup Z_3,\eta}(\Theta^-)\rangle\nonumber\\
&=& 1\otimes \Pi \langle\tau^{-1}F^-_{Z_1,\eta}(\Theta^-),\, F^-_{Z_2,\eta}\circ F^-_{Z_3,\eta}(\Theta^-)\rangle.\label{side1}
\end{eqnarray}
Similarly,
\begin{equation}
\O_{X,Y_2,\eta,s} = \Pi\otimes 1\langle \tau^{-1} F^-_{Z_2,\eta}\circ F^-_{Z_1,\eta}(\Theta^-),\, F^-_{Z_3,\eta}(\Theta^-)\rangle.
\label{side2}
\end{equation}
We therefore define
\[
\tO = \langle\tau^{-1}F^-_{Z_1,\eta}(\Theta^-), F^-_{Z_2,\eta}\circ F^-_{Z_3,\eta}(\Theta^-)\rangle= \langle\tau^{-1} F^-_{Z_2,\eta}\circ F^-_{Z_1,\eta}(\Theta^-), F^-_{Z_3,\eta}(\Theta^-)\rangle
\]
using duality and the analog of Lemma \ref{lemma3} in the perturbed case. Here we also note that the pairings above take values in
\[
\K(Z_1,\eta)\otimes_{\R_{Y_1,\eta}}\overline{\K(Z_2\cup Z_3, Y_2,\eta)} \quad\mbox{and}\quad  \K(Z_1\cup Z_2, Y_1,\eta)\otimes_{\R_{Y_2,\eta}} \overline{\K(Z_3,\eta)},
\]
which are mutually isomorphic to the Novikov completion of
\[
\zee\left[\frac{K(Z_1)\oplus K(Z_2)\oplus K(Z_3)}{H^1(Y_1)\oplus H^1(Y_2)}\right]
\]
with respect to (the linear function on $K(Z_1)\oplus K(Z_2)\oplus K(Z_3)$ induced by) $\eta$. Note that just as in Lemma \ref{coefflemma}, there is an isomorphism
\[
\frac{K(Z_1)\oplus K(Z_2)\oplus K(Z_3)}{H^1(Y_1)\oplus H^1(Y_2)} \cong \ker[\rho_1\oplus\rho_2\oplus\rho_3: H^2(X)\to H^2(Z_1)\oplus H^2(Z_2)\oplus H^2(Z_3)], 
\]
where $\rho_i$ is the restriction $H^2(X)\to H^2(Z_i)$. We denote the above group by $K(X; Y_1,Y_2)$ and the Novikov completion of $\zee[K(X;Y_1,Y_2)]$ with respect to $\eta$ by $\K(X;Y_1,Y_2,\eta)$. With these algebraic identifications understood, the lemma follows from \eqref{side1} and \eqref{side2}.
\end{proof}

Finally we obtain the following, which is a restatement of Theorem \ref{intropertproductthm} from the introduction. It should be seen as a generalization of Theorem \ref{productthm} that allows us to calculate {\OS} invariants using essentially any cut $Y$, if we are willing to use an appropriate perturbation.

\begin{theorem}\label{pertproductthm} Let $X$ be a closed oriented 4-manifold with $b^+(X)\geq 2$, and $Y\subset X$ a submanifold determining a decomposition $X = Z_1\cup_Y Z_2$, where $Z_i$ are connected 4-manifolds with boundary. Fix a class $\eta\in H^2(X;\arr)$, and assume that $Y$ is an allowable cut for $\eta$. If $b^+(Z_1)$ and $b^+(Z_2)$ are not both 0, then for any \spinc structure $\s$ on $X$ and element $\alpha\in \A(X)$,
\begin{equation}\label{pertproductform}
\sum_{t\in K(X,Y)}\Phi_{X,\s + t}(\alpha) e^{t} = \O_{X,Y,\eta,\s}(\alpha) = \langle \tau^{-1}\Psi_{Z_1,\eta,\s}(\alpha_1),\,\Psi_{Z_2,\eta,\s}(\alpha_2)\rangle
\end{equation}
up to sign and multiplication by an element of $K(X,Y)$, where $\alpha_1\otimes \alpha_2\mapsto \alpha$ as before. If $b^+(Z_1) = b^+(Z_2) = 0$ then the same is true after possibly replacing $\eta$ by another class $\tilde{\eta}$, where $\tilde{\eta}|_{Z_i} = \eta|_{Z_i}$ for $i = 1,2$.
\end{theorem}

\begin{proof} If both $b^+(Z_1)\geq 1$ and $b^+(Z_2)\geq 1$ then this follows from Lemma \ref{relationlemma}. Assume, therefore, that $b^+(Z_1) = 0$. We wish to find a cut $Y'$ for $X$ such that (1) $Y'$ is disjoint from $Y$, and (2) in the decomposition $X = Z_1'\cup_{Y'} Z_2'$, we have $b^+(Z_i') \geq 1$ for $i = 1,2$.

To find $Y'$, first consider the restriction $\eta|_Y$. Since $Y$ is allowable for $\eta$ and $b^+(Z_1) =0$, we must have $\eta|_Y \neq 0\in H^2(Y;\arr)$. Hence we can find a surface $\Sigma\subset Y$ such that $\int_\Sigma\eta \neq 0$, and since $\eta$ is defined on $X$, we infer $[\Sigma]$ is nonvanishing in $H_2(X;\zee)$. Clearly $\Sigma . \Sigma = 0$. Let $S$ be an embedded surface in $X$ intersecting $\Sigma$ transversely in a single point; then $\{[\Sigma], [S]\}$ determine a direct summand of the intersection matrix of $X$ having one positive and one negative eigenvalue. Let $N$ be a tubular neighborhood of $S$ in $X$; then $Y$ separates $N$ into two components $N_1\cup N_2$, with $N_i\subset Z_i$. Let $\tilde{Z}_1$ be obtained by adding a collar $Y\times [0,\epsilon]\subset Z_2$ to $Z_1$, and set $Z_1' = \tilde{Z}_1 \cup N_2$. Thus $Y' = \partial Z_1'$ is obtained by pushing $Y$ slightly into $Z_2$ and attaching the boundary of $N_2$.

Since $\Sigma\cup S\subset Z_1'$ we see $b^+(Z_1') = 1$; on the other hand, the complement $Z_2' = X\setminus Z_1'$ has $b^+(Z_2') = b^+(Z_2)$. There are several cases to distinguish.

{\it Case 1: $b^+(Z_2') \geq 1$.} Here we are done, by Lemma \ref{relationlemma} and Lemma \ref{cutinvlemma}.

{\it Case 2: $b^+(Z_2') = 0$, but $\eta|_{Y'} \neq 0$.} Then $Y'$ is still an allowable cut for $\eta$ and disjoint from $Y$, so Lemma \ref{cutinvlemma} applies. We can now run the construction above with $Z_2'$ playing the role of $Z_1$; the result is a new cut $Y''$, disjoint from $Y'$, with $b^+(Z_i'')\geq 1$ for $i = 1,2$. Lemma \ref{cutinvlemma} implies that the invariants calculated from $Y$, $Y'$, and $Y''$ agree, while Lemma \ref{relationlemma} shows that the invariants calculated from $Y''$ are the {\OS} invariants.

{\it Case 3: $b^+(Z_2') = 0$ and $\eta|_{Y'} = 0$.} Let $\tilde{\eta} = \eta + PD_X[\Sigma]$, where $PD_X[\Sigma]$ denotes the image in real cohomology of the Poincar\'e dual of $[\Sigma]$ in $H^2(X;\arr)$. Then it is easy to see that $\tilde{\eta}|_{Y'} \neq 0$, so that $Y'$ is an allowable cut for $\tilde{\eta}$. Note that since the classes $\tilde{\eta}$ and $\eta$ differ by an element in the image of $\delta: H^1(Y;\arr)\to H^2(X;\arr)$, they agree on $Z_1$ and $Z_2$. Running the preceding proof with $\tilde{\eta}$ in place of $\eta$ we end in case 2 above, hence the conclusion of the theorem holds with the modified perturbation.
\end{proof}

Note that in case 3 of the proof, it works just as well to take $\tilde{\eta} = \eta + \epsilon PD_X[\Sigma]$, where $\epsilon$ is an arbitrary nonzero real number. Thus Theorem \ref{pertproductthm} could be rephrased to say that when $b^+(X)\geq 2$, the perturbed {\OS} invariants are equal to the ordinary {\OS} invariants when calculated with respect to a cut $Y$ that is allowable for a ``generic'' class $\eta\in H^2(X;\arr)$.

The preceding results provide sufficient understanding of the dependence of $\O_{X,Y\eta,\s}$ on $Y$ for our purposes. We do not study the dependence of the perturbed invariants on $\eta$ here.

\section{Heegaard Floer homology of a surface times a circle}\label{calcsec}

From the general considerations of the preceding sections, we turn now to the problem mentioned in the introduction of determining the behavior of {\OS} invariants under fiber sum. Since a fiber sum along surfaces with trivial normal bundle is obtained by gluing two manifolds together along the product of the summing surface $\Sigma$ with a circle, and the relative invariants of the pieces take values in the Floer homology of the latter manifold, we will need a fairly detailed understanding of that Floer homology.

This section is devoted to the calculation of the perturbed Heegaard Floer homology groups of $\Sigma\times S^1$, for a particular choice of perturbation $\eta$. Indeed, our choice of $\eta$ is induced by the cobordism $\Sigma_g \times D^2 - D^4$. 
The main input for this computation comes from 
\cite{us} where most of the technical tools required have been developed.  
We start this section by elucidating the new phenomena associated with working with twisted coefficients in 
surgery exact sequences. 

\subsection{Exact sequences with twisted coefficients}
Let $K$ be a nullhomologous knot in a 3-manifold $Y$. Following typical notation in the subject, we write $Y_\ell = Y_\ell (K)$ for the 3-manifold obtained 
by $\ell$-framed surgery on $K$. As described in \cite{OS2}, 
there are exact sequences relating $HF^+$ (or $\widehat{HF}$) 
of the two triples of three 3-manifolds ($Y_0$, $Y$, $Y_{-n}$) and ($Y_0$, $Y_{n}$,$Y$)
with $n>0$ but otherwise arbitrary:
\begin{align} \label{LES}
... & \stackrel{G}{\longrightarrow}  HF^+(Y_0, [\s _k])  \stackrel{H}{\longrightarrow}  
HF^+(Y,\s)  
\stackrel{F}{\longrightarrow} HF^+(Y_{-n}, \s_k) \stackrel{G}{\longrightarrow} ...  \cr
... & \stackrel{G}{\longrightarrow}  HF^+(Y_0, [\s _k])  \stackrel{H}{\longrightarrow}  
HF^+(Y_n,\s_k)  
\stackrel{F}{\longrightarrow} HF^+(Y, \s) \stackrel{G}{\longrightarrow} ... 
\end{align}
By abuse of notation we have labeled the maps appearing in the two sequences by the same letters although 
they are of course different functions. 
The map $F:HF^+(Y,\s)\to HF^+(Y_{-n},\s_k)$ will be of special interest below and we proceed by first providing
more details concerning its definition as well as explaining the notation from \eqref{LES}. 

Let $W_{-n}$ be the cobordism from $Y$ to 
$Y_{-n}$ obtained by attaching a $-n$-framed 2-handle to $Y\times [0,1]$ along $K\times \{1\}$. Let  
$\sigma\subset Y$ be a Seifert surface of $K$ and let $S\subset W_{-n}$ be the surface obtained by capping off 
$\sigma \times \{1\}$ with the core of the attaching 2-handle. Given a 
spin$^c$-structure $\s \in Spin^c(Y)$ let $\s_k \in Spin^c(Y_{-n})$ be the spin$^c$-structure on $Y_{-n}$ which is 
spin$^c$-cobordant to $\s$ via $(W_{-n}, \rs _{k,0})$ where $\rs_{k,\ell} \in Spin^c(W_{-n})$ is uniquely determined 
by $\rs_{k,\ell }|_{Y}=\s$, $\langle c_1(\rs_{k,\ell}),[S]\rangle = 2k- (2\ell -1)n$ 
and $k\in \{0,1,...,n-1\}$.\footnote{Every spin$^c$-structure 
$\rs \in Spin^c(W_{-n})$ with $\rs|_{Y} = \s$ and $\rs |_{Y_{-n}} = \s_k$ is of the form $\rs = \rs_{k,\ell}$ 
for some $\ell \in \mathbb{Z}$.} By $[\s_k]$ we denote the preimage $Q_{\pm }^{-1}(\s_k)$ of a 
surjective map $Q_{\pm} :Spin^c(Y_0) \to Spin^c(Y_{\pm n})$ defined in \cite{OS2} whose details need not concern us 
save the fact that when $n\gg 0$ this preimage includes at most a single spin$^c$-structure whose Floer homology $HF^+(Y_0,\tk)$ is nontrivial. By writing 
$HF^+(Y_0,[\s_k])$ we mean the direct sum of $HF^+(Y_0,\tk)$ over all spin$^c$-structures $\s \in [\s_k]$.    

The map $F$ from \eqref{LES} is a sum 
\begin{equation} \label{summandsofF}
F=\bigoplus _{\ell \in \mathbb{Z}} F_\ell, \quad \quad \mbox{ where } \quad \quad 
F_\ell : HF^+(Y,\s) \to HF^+(Y_{-n},\s_k)
\end{equation}
is the homomorphism induced by $(W_{-n}, \rs _{k,\ell})$. 

Recall that when $c_1(\s)$ is torsion both $HF^+(Y,\s)$ and $HF^+(Y_{-n},\s_k)$ 
come equipped with an absolute $\mathbb{Q}$-grading $\widetilde{gr}$ lifting the relative $\mathbb{Z}$-grading
$gr$ (cf \cite{OS4}). With respect to the absolute grading $\widetilde{gr}$ the degrees of the maps $F_\ell$
on homogeneous elements are 
\begin{equation} \label{degreeoffl}
\deg F_\ell = \frac{1}{4} \left( 1 - \frac{(2k-(2\ell-1)n)^2}{n}\right)
\end{equation}
This function attains its maximum at $\ell = \frac{1}{2} - \frac{k}{n}$, though of course $\ell$ is constrained to be an integer. When $k\neq 0$ there is therefore a unique value of $\ell$ corresponding to the maximal degree shift, while for $k=0$ the maximum is attained for both $\ell = 0,1$.

To state the version of the sequence for twisted coefficients we first introduce some more notation. 
With the choice of a Seifert surface $\sigma\subset Y$ of $K$ as above, let 
$\hat{\sigma} \subset Y_0$ be the surface obtained by capping off $\sigma$ with the surgery disk. 
Set $t = P.D.([\widehat{\sigma}]) \in H^1(Y_0;\mathbb{Z})$ and let $L(t)=\mathbb{Z}[t,t^{-1}]$ 
be the ring of Laurant polynomials in $t$; equivalently $L(t)$ is the group ring on the subgroup of $H^1(Y_0)$ generated by $t$. There is a natural homomorphism $\zee[H^1(Y_0)]\to L(t)$ induced by the map $\alpha \mapsto \langle \alpha,[K]\rangle\cdot t$, or in multiplicative notation $\alpha\mapsto t^{\langle \alpha,[K]\rangle}$. (Here $[K]$ indicates the homology class in $H_1(Y_0)$ coming from the core of the surgery torus.) We use this map to endow $L(t)$ with the structure of a $\zee[H^1(Y_0)]$ module; observe that if $\tk\in Spin^c(Y_0)$ is a \spinc structure whose Chern class is dual to a multiple of $[\hat{\sigma}]$ then $L(t)$ is naturally a graded module for $R_{Y_0,\tk}$ with $\gr_{\tk}(t) = -\langle c_1(\tk)\cup t, [Y_0]\rangle$.

More generally, suppose $\s\in Spin^c(Y)$ is a \spinc structure on the original 3-manifold, and $M$ is a graded module for $R_{Y,\s}$. Then the surgery cobordism $Y\to Y_0$, equipped with some \spinc structure, induces a graded module for $R_{Y_0,\tk}$ that we denote by $M[t^{\pm 1}]$, whose underlying group is $M\otimes_{\zee[H^1(Y)]}L(t)$ and where $H^1(Y)$ acts trivially on $L(t)$.


With this understood, the next theorem can be found in \cite{OS2}.
\begin{theorem} \label{surgeryseq}
Let $Y$ be a three manifold, $K$ a nullhomologous knot in $Y$ and $M$ an $R_Y$-module. 
With the notation as above and 
for any $n>0$ there are surgery long exact sequences of $R_{Y_0}\otimes \mathbb{Z}[U]$-modules for the Heegaard Floer homology groups 
with twisted coefficients  
\begin{align} \nonumber
... & \stackrel{\ug}{\longrightarrow}  HF^+(Y_0, [\s _k];M[t^{\pm 1}] )  \stackrel{\uh}{\longrightarrow}  
HF^+(Y,\s;M)[t^{\pm 1}]  
\stackrel{\uf}{\longrightarrow} HF^+(Y_{-n}, \s_k;M)[t^{\pm 1}] \stackrel{\ug}{\longrightarrow} ... \cr 
... & \stackrel{\ug}{\longrightarrow}  HF^+(Y_0, [\s _k];M[t^{\pm 1}] )  \stackrel{\uh}{\longrightarrow}  
HF^+(Y_n,\s;M)[t^{\pm 1}]  
\stackrel{\uf}{\longrightarrow} HF^+(Y, \s_k;M)[t^{\pm 1}] \stackrel{\ug}{\longrightarrow} ... 
\end{align}
The analogous sequences for $\widehat{HF}$ are also exact. 
\end{theorem}
We shall refer to the above sequences as the {\em negative $n$} and {\em positive $n$
surgery sequences} respectively. As in \eqref{LES} we abuse notation by labelling the maps in both sequences
by the same letters.  
It is worthwhile to single out a case of special interest later on, namely the choice of 
$M=\mathbb{Z}$ with trivial $R_Y$-module structure. In this case the negative $n$ sequence becomes  
\begin{align} \label{LEStwisted}
... \stackrel{\ug}{\longrightarrow}  HF^+(Y_0, [\s _k];L(t) )  \stackrel{\uh}{\longrightarrow}  
HF^+(Y,\s)[t^{\pm 1}]  
\stackrel{\uf}{\longrightarrow} HF^+(Y_{-n}, \s_k)[t^{\pm 1}] \stackrel{\ug}{\longrightarrow} ... 
\end{align}

For any choice of $\eta \in H^2(Y_0;\mathbb{Z})$, using the 
flatness property of the Novikov ring $\R_{Y_0,\eta}$ we obtain this useful consequence of the above theorem:
\begin{cor}
Suppose $\langle \eta\cup t, [Y_0]\rangle >0$. Then for any $n>0$ there is a long exact sequence 
\begin{align} \label{LEStwistedNovikov}
\cdots \stackrel{\ug}{\longrightarrow}  HF^+(Y_0, [\s _k];\L(t))  
\stackrel{\uh\otimes 1}{\longrightarrow}  & 
HF^+(Y,\s)[t^{\pm 1}]\otimes _{R_{Y_0}} \R_{Y_0,\eta}   
\stackrel{\uf\otimes 1 }{\longrightarrow}  \cr 
& \stackrel{\uf\otimes 1 }{\longrightarrow}   HF^+(Y_{-n}, \s_k)[t^{\pm 1}]\otimes _{R_{Y_0}} \R_{Y_0,\eta}  
 \stackrel{\ug\otimes 1}{\longrightarrow} \cdots,
\end{align}
where $\L(t)$ denotes the ring of Laurent series in $t$.
\end{cor} 

There is a straightforward relationship between the exact sequences \eqref{LES} and \eqref{LEStwisted}.

\begin{prop} \label{theforce} 
Let $K$ be a nullhomologous knot in $Y$ and let 
$F : HF^+(Y,\s ) \rightarrow HF^+(Y_{-n},s_k)$ and 
$\uf:HF^+(Y,\s )[t^{\pm 1}]\rightarrow HF^+(Y_{-n},s_k)[t^{\pm 1}] $ be the maps appearing in the exact sequences 
\eqref{LES} and \eqref{LEStwisted} respectively. Let $F_i$ be the components of $F$ as in \eqref{summandsofF}. 
Then 
$$ \uf = \sum _{\ell \in \mathbb{Z}} F_\ell\otimes t^\ell $$
up to sign and overall multiplication by a power of $t$.

Moreover, when $k\ne 0$ or $\s$ is nontorsion, for all sufficiently large $n$ the only nonzero terms in the above formula are those for which $\ell = 0$ or $1$.
When $k=0$ and $\s$ is torsion, then the same is true for the restriction of $\uf$ to $HF^+_{\leq d_0}(Y,\s)$ for any fixed grading $d_0$ of $HF^+(Y,\s)$.  
\end{prop}

\begin{proof} The homomorphisms in both sequences are defined by counts of holomorphic triangles in appropriate Heegaard triple-diagrams, and the stated relationship between $\uf$ and $F$ follows from elementary considerations in these diagrams. Indeed, in notation from \cite{OS2} (see also \cite{OSsurg}), the map in the twisted sequence can be written as
\[
\uf([\x,i]) = \sum_{\psi\in\pi_2(\x,\Theta,\y)} \#\M(\psi) [\y,i-n_z(\psi)] \cdot t^{n_\gamma(\psi)},
\]
where the sum is over homotopy classes of triangles $\psi$ in a diagram $(\Sigma,\balpha,\bbeta,\bgamma,z)$ describing the natural cobordism $Y\to Y_{-n}$. In this situation we are using twisted coefficients on $Y_{-n}$ constructed by fixing a reference point $\tau$ lying on the surgery circle $\gamma_g$, such that the boundary operator in the twisted chain complex for $Y_{-n}$ records (in the power of $t$) the intersection of the $\gamma$-component of a holomorphic disk with the subvariety $V = \gamma_1\times\cdots\times \gamma_{g-1}\times\{\tau\}\subset T_\gamma\subset Sym^g(\Sigma)$. (This formal device induces trivially twisted coefficients on $Y_{-n}$.) In the formula above, the power $n_\gamma(\psi)$ is similarly the intersection of the $\gamma$-component of the boundary of $\psi$ with $V$. 

The first claim of the proposition amounts to the fact that the power of $t$ appearing above determines and is determined by the value $\langle c_1(s_z(\psi)),H(\P)\rangle$, where $H(\P)$ is the 2-dimensional homology class of the triply-periodic domain $\P$ corresponding to the generator of the 2-dimensional homology of the surgery cobordism. This in turn follows easily from inspection of the Heegaard triple itself, together with the expression for $\langle c_1(s_z(\psi)),H(\P)\rangle$ in terms of data on the Heegaard diagram obtained by Ozsv\'ath and Szab\'o (Proposition 6.3 of \cite{OS3}).

To see the remaining claims, recall that the homomorphism $F_\ell$ (corresponding to the \spinc structure on the cobordism with $\langle c_1(\rs_{k,\ell}),[S]\rangle = 2k - (2\ell -1)n$) induces a shift in degree given by \eqref{degreeoffl}. Hence for $\ell\neq 0,1$, the degree of $F_\ell$ is at least $2n$ lower than that of $F_0$ or $F_1$. When $k\neq 0$ or $\s$ is nontorsion, the group $HF^+(Y_{-n},\s_k)$ is finitely generated; hence for $n$ sufficiently large, this observation implies that only $F_0$ and $F_1$ can be nontrivial (note that for sufficiently large $n$, $HF^+(Y_{-n}, \s_k)$ is independent of $n$, c.f. \cite{OSknot}). For $k = 0$ and $\s$ torsion, the same holds after restriction to a finitely generated subgroup.
\end{proof}

\subsection{A surface cross a circle - Partially twisted coefficients} \label{partiallytwistedhf}
In this section we apply the general discussion from the previous section to the case of 
$Y=\#^{2g} (S^1 \times S^2)$ and $K=K_g=\#^g B(0,0)$ with $B(m,n)$ defined in  figure \ref{pic1}. 
\begin{figure}[htb!] 
\centering
\includegraphics[width=6cm]{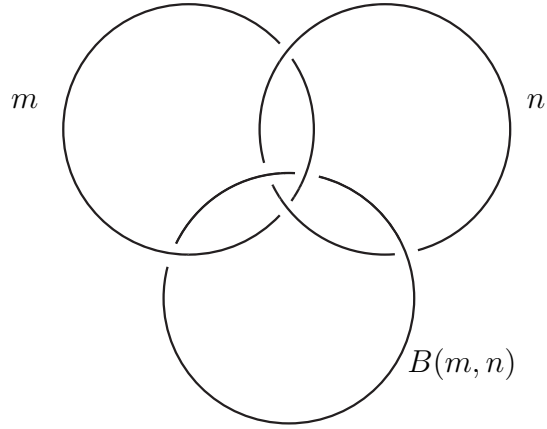}
\put(-190,120){$m$}
\put(5,120){$n$}
\put(-40,20){$B(m,n)$}
\caption{The knot $B(m,n)\subset L(m,1)\#L(n,1)$ with $m,n \in \mathbb{Z}$. }  \label{pic1}
\end{figure}
Let $\sigma _1$ be the Seifert surface for $K_1$ obtained from the obvious disk bounded by $B(0,0)$ 
in Figure \ref{pic1} by adding two 1-handles where the other two components of the Borromean rings 
intersect that disk. Let $\sigma = \sigma _g = \#^g \sigma _1$ be the choice of Seifert surface for $K_g$. 
It is then not hard to see that $Y_0 = \Sigma _g \times S^1$ and $t$ becomes the 
Poincar\'e dual of $[\Sigma _g]$. For the rest of this subsection and the next, we assume $g\geq 2$.

Let $\s$ be the unique torsion \spinc structure on $Y$. Then for $n\gg 0$ the only \spinc structure in 
the set $[\s_k]$ having nontrivial $HF^+$ consists of the unique spin$^c$-structure $\s_k\in Spin^c(\Sigma _g \times S^1)$ 
with $c_1(\s_k) = 2k P.D.([S^1])$. We shall only focus on the Novikov ring $\R_{Y_0,\eta}$ 
associated to 
$$\eta = P.D.([S^1]) $$

With these choices of $Y, K, \sigma$ the maps $F_0, F_1:HF^+(Y,\s) \to HF^+(Y_{-n},\s_k)$ appearing in 
proposition \ref{theforce} have been 
made completely explicit by the results from  \cite{us, OSknot}. Before proceeding we describe these 
maps. 

Let us use the shorthand  $\Lambda ^k$ and $\Lambda^*$ to denote $\Lambda^kH^1(\Sigma _g;\mathbb{Z})$ and 
$\Lambda^*H^1(\Sigma _g;\mathbb{Z})$ respectively. We embed
the $\mathbb{Z}[U]$-module $\Lambda^* \otimes \mathbb{Z}[U,U^{-1}]$ 
into a 2-dimensional coordinate system by placing $\Lambda^k\otimes U^\ell$ at coordinate
$(-\ell,k-g-\ell)$. We equip the coordinate plane with a $\mathbb{Q}$-grading 
$\widetilde{gr}: \mathbb{Z}^2 \to \mathbb{Q}$ by setting 
\begin{equation} \label{Qgrading}
\widetilde{gr}(i,j) = i+j.
\end{equation}
In this picture the action of $U$ can be thought of as translation by $(-1,-1)$; as usual it decreases 
grading by two. We shall write 
$$H\{ \mbox{condition on} (i,j)  \}$$
to denote the various $\mathbb{Z}[U]$ sub- and quotient modules of $\Lambda^*\otimes \mathbb{Z}[U,U^{-1}]$ 
obtained as a direct sum over all the terms in the coordinate system which reside at  
coordinates $(i,j)$ subject to the stated conditions. 
For example $H\{i<0\}$ is the submodule $\Lambda^*\otimes (U\cdot \mathbb{Z}[U])$ and 
$H\{i\ge 0\}$ is the quotient module $\Lambda^*\otimes \T$, where $\T =\mathbb{Z}[U,U^{-1}]/U\cdot \mathbb{Z}[U]$.

For $0\le d\le g-1$ we define the $\mathbb{Z}[U]$-module $X(g,d)$ as 
$$ X(g,d) = \bigoplus _{i=0}^{d} \left( \Lambda^i \otimes _\mathbb{Z} \frac{\T}{U^{i - d -1}\cdot \zee[U^{-1}]}   \right)$$
In the notation above, $X(g,d)$ is isomorphic to 
\begin{equation} \label{xgdidentification}
 X(g,d) \cong H\{i\ge 0 \mbox{ and } j<d+1-g  \}
\end{equation}
as a $\mathbb{Z}[U]$-module. We shall refer to this identification as the {\em standard embedding} of 
$X(g,d)$ into $\Lambda^*\otimes \mathbb{Z}[U,U^{-1}]$. We shall encounter \lq\lq non-standard embeddings\rq\rq of 
$X(g,d)$ as well, see Theorem \ref{groups1} below. 

It was shown in \cite{OSknot} that for $Y = \#^{2g}S^1\times S^2$ and $K = \#^g B(0,0)$ as above, and $n\gg 0$, there are 
$\Lambda^*H_1(\Sigma_g;\mathbb{Z})\otimes\mathbb{Z}[U]$-module isomorphisms
\begin{equation} \label{identifications}
HF^+(Y,\s) \cong H\{i\ge 0\} \quad \mbox{ and } \quad HF^+(Y_{-n},\s_k) \cong H\{ i\ge 0 \mbox{ and } j \ge k\} 
\end{equation}
where the action of $\Lambda^* H_1(\Sigma _g;\mathbb{Z})\cong \Lambda^*  H_1(Y;\mathbb{Z})$ on $\Lambda^k\otimes U^\ell$ is given by 
\begin{equation} \label{standardaction}
\gamma \cap (\alpha \otimes U^\ell) = \iota _\gamma \alpha \otimes U^\ell + (P.D. (\gamma) \wedge 
\alpha )\otimes U^{\ell +1} \quad \quad \gamma \in H_1(\Sigma _g;\mathbb{Z}), \, \alpha \in \Lambda ^k
\end{equation}
Here $\iota _\gamma$ is contraction with $\gamma$ and $P.D.(\gamma)$ is the Poincar\'e dual of $\gamma$ 
taken on $\Sigma _g$. By virtue of \eqref{xgdidentification} this action 
induces an action on $X(g,d)$. 
We shall refer to \eqref{standardaction} as the {\em standard action}, and use the cap product notation $\cap$ to distinguish it from actions of first homology on Floer homology that need not be ``standard'' (we use the ``dot'' notation for the latter: $h.\xi$ for $h\in H_1$, $\xi\in HF^\circ$).

To describe  the maps $F_0, F_1:H\{i\ge 0\} \to H\{i\ge 0 \mbox{ and } j\ge k\}$ (under the identifications 
\eqref{identifications}) we need a bit more notation first. Let 
$e_1,\ldots, e_{2g}$ be a symplectic basis for $H_1(\Sigma_g;\mathbb{Z})$ and set 
$\omega = e_1\wedge e_2 + \cdots + e_{2g-1}\wedge e_{2g} \in \Lambda^2H_1(\Sigma _g;\mathbb{Z})$. For a given  
$\beta \in \Lambda^1$ define $\beta \angle:\Lambda ^k \to 
\Lambda^{k-1}$ as contraction associated to $\omega$, i.e. 
$$ \beta \angle (\alpha _1\wedge ... \wedge \alpha _k) = \sum _{\ell=1}^k (-1)^\ell 
\omega( \alpha_\ell, \beta)\,  \, 
\alpha _1\wedge ... \wedge \widehat{\alpha}_\ell\wedge .. \wedge \alpha _k  $$
where $\omega(\alpha_\ell, \beta)$ refers to the natural pairing between homology and cohomology on
$\Sigma _g$. The contraction $\angle$ defined this way extends readily to a contraction 
$\angle :\Lambda ^m \otimes \Lambda^k \to \Lambda^{k-m}$ as 
$(\beta_1\wedge...\wedge \beta_m)\angle (\alpha _1\wedge ... \wedge \alpha _k) = 
\beta_1\angle (\beta_2\angle ( ... ( \beta_m\angle( \alpha _1\wedge ... \wedge \alpha _k)...))$. Let 
 $\tilde\star :\Lambda ^k \to \Lambda^{2g-k}$ be the \lq\lq Hodge-Lefschetz star operator\rq\rq  associated to 
$\omega$ and defined as 
$$ \tilde \star \, \alpha = \frac{1}{g!} \,\,   \alpha \angle  \omega^g $$
where we have by abuse of notation used $\omega$ to also denote 
$e^1\wedge e^2 + ... + e^{2g-1}\wedge e^{2g} \in \Lambda^2$, which is the dual 
of the symplectic form $\omega$ from earlier. Here $e^i\in H^1(\Sigma _g;\mathbb{Z})$, $i=1,...,2g$  
is the dual symplectic basis of $e_i\in H_1(\Sigma _g;\mathbb{Z})$, i.e. $e^i(e_j) = \delta_{ij}$. 
Let 
\begin{align} \nonumber
\pi _k :  H\{i \ge 0 \} & \to H\{i\ge 0 \mbox{ and } j \ge k\} \cr
\pi_{i\ge 0} :  \Lambda^*\otimes \mathbb{Z}[U,U^{-1}]  & \to H\{i\ge 0\} \cr
\pi_{j\ge 0} :  \Lambda^*\otimes \mathbb{Z}[U,U^{-1}]  & \to H\{j\ge 0\} 
\end{align}
be the natural projection maps and let 
$J:H\{i\ge 0\} \to H\{j\ge 0\}$ be the map 
$$ J(xU^\ell ) = \pi_{j\ge0}  \left( 
(-1)^{k+g-1} \, \exp(2\omega U) \angle \, (\tilde \star \, x) \, U^{g+\ell - k}\right) 
\quad \mbox{ when } \quad x\in \Lambda^k,$$
where by convention, contraction with $U^n$ is taken to mean multiplication by $U^{-n}$.
With this understood, it was shown in \cite{us} that 
$F_0, F_1:H\{i\ge0 \} \to H\{i\ge0 \mbox{ and }j\ge k\}$ are given by 
\begin{align} \nonumber 
F_0 = \left\{ 
\begin{array}{cl}
\pi_k & ; \, k \le 0 \cr
\pi_k \circ (U^{-k}\, J ) & ; \, k >0 
\end{array}
\right.\quad \quad \mbox{ and } \quad \quad 
F_1= \left\{ 
\begin{array}{cl}
\pi_k \circ (U^{-k} \, J ) & ; \, k \le 0 \cr
\pi_k & ; \, k >0 
\end{array}
\right.
\end{align}
With all these preliminaries out of the way and with our notation in place, we now turn to the actual 
calculations of the twisted Heegaard Floer groups of $\Sigma _g \times S^1$. 
The adjunction inequality implies that for any 
spin$^c$-structure $\s$ on $\Sigma_g \times S^1$ which is not among the $\s_k$, the associated Heegaard Floer 
groups $HF^+(\Sigma _g\times S^1,\s;M)$ vanish (for any coefficient module $M$);
the same is true for $\s=\s_k$ when $|k|\ge g$. 
The remaining spin$^c$-structures $\s_k$ with $|k|\le g-1$ give rise to nontrivial 
Heegaard Floer groups as the next theorem explains. 
\begin{theorem} \label{groups1}
Pick an integer $k$ with $|k|\le g-1$. If $k\ne 0$ choose $\Lambda$ to be either $L(t)$ or $\L(t)$ and 
if $k=0$ choose $\Lambda = \L(t)$. 
Then the Heegaard Floer homology groups $HF^+(\Sigma _g \times S^1,\s_k;\Lambda)$ 
are isomorphic to 
$$  HF^+(\Sigma _g \times S^1,\s_k;\Lambda) \cong X(g,d)\otimes \Lambda \quad \quad \mbox{ with } \quad \quad d= g-1-|k|$$
as $\mathbb{Z}[H^1(\Sigma_g\times S^1;\mathbb{Z})]\otimes _\mathbb{Z} \mathbb{Z}[U]$-modules. 
The action of $\Lambda^*H_1(\Sigma_g;\mathbb{Z})\subset \Lambda^*H_1(\Sigma_g\times S^1;\mathbb{Z})$ on 
$X(g,d)$ is 
the restriction of the standard action \eqref{standardaction} under the non-standard embedding of $X(g,d)$ into 
$\Lambda^*H^1(\Sigma_g;\mathbb{Z})\otimes \mathbb{Z}[U,U^{-1}]\otimes \Lambda$ given by 
$$ x\mapsto x+ \pi_{-|k|} \left( \sum _{\ell \ge 1} (-t\,  U^{|k|} J)^\ell \, x \right) .$$
%
\end{theorem}
\begin{proof}
The proof of the theorem follows slightly different arguments depending on whether $k\ne 0$ or $k=0$. 
We first address the former.

{\em Case of $k\ne 0$ }
For concreteness let us assume $k<0$. Choose $\Lambda =L(t)$ for now and consider the exact sequence \eqref{LEStwisted}.  
By proposition \ref{theforce} the map $\uf$ equals $F_0+ t F_1 $ once $n$ is chosen sufficiently large (which 
we assume tacitly throughout). 
It follows from \eqref{degreeoffl} that $\deg F_0 = \deg F_1 - 2k$ and thus $\deg F_0  > \deg F_1$ when $k<0$.  
Since $F_0$ is clearly surjective 
and $\deg F_1<\deg F_0$ we see that $\uf$ is also surjective. Moreover the kernel of $\uf$ 
is generated by elements of the form 
$$ \kerr \uf  = \left\langle \pi_{k} \left.\left( \sum _{\ell\ge 1} (-t\,  U^{-k} J)^\ell \, x \, \right) \right| \, 
x \in H\{ i\ge 0 \mbox{ and } j<k\}  \right\rangle $$
Notice that the sum $\sum_{\ell \ge 1} (-t\,  U^{-k} J)^\ell \, x$ is finite and thus well defined. 
Projection onto the homogeneous term of highest degree establishes the isomorphism  
$\kerr \uf \cong X(g,d)$ with $d=g-1-|k|$. Since the sequence \eqref{LEStwisted} is equivariant with 
respect to the $H_1(\Sigma _g;\mathbb{Z})$-action, the proposition (for the case $k<0$ and $\Lambda =L(t)$) 
follows. The results with $\Lambda =\L(t)$ follows from the result for $L(t)$ by tensoring 
with $\R_{Y_0,\eta}$ and using the flatness of $\R_{Y_0,\eta}$. The case of $k>0$ can be proved analogously, or  by appeal to conjugation invariance (Theorem \ref{pertconjinvthm}).

{\em Case of $k=0$ } 
Consider once more the sequence \eqref{LEStwisted} and note that upon restriction to a given grading we may take $\uf = F_0 + tF_1$ according to 
Proposition \ref{theforce}. The one key difference from the case of $k\ne 0$ 
is that the degrees of $F_0$ and $F_1$ are now equal. 

We begin by showing that $\uf$ is again surjective: for a given $y\in H\{i\ge 0 \mbox{ and } j\ge 0\}$ let 
$x_y\in H\{i\ge 0\}\otimes \Lambda _\eta$ be 
$x_y= \pi_{i\ge 0} \left( \sum _{\ell\ge 0} (-tJ)^\ell y \right) $. Then 
$$ \uf (x_y) = \pi_0(\mbox{id} + tJ)\left(\pi_{i\ge 0}  \sum _{\ell\ge 0} (-tJ)^\ell y \right)= 
\pi_0(\mbox{id} + tJ)\left(  \sum _{\ell\ge 0} (-tJ)^\ell y \right) = \pi_0(y) = y $$
To determine the kernel of $\uf$ pick a kernel element 
$\xi  = \xi_0 + \xi_1 t + \xi_2t^2+ ... \in \kerr (\uf)$. Such an element $\xi$ is then subject to the 
infinite system of equations
\begin{align} \label{infinitesystem}
\pi_0(\xi_0) & = 0 \cr
\pi_0(\xi_1 + J(\xi_0))  & = 0 \cr 
& \vdots \cr
\pi_0(\xi_k + J(\xi_{k-1}))& = 0 \cr 
& \vdots 
\end{align}
The equation $\pi_0(\xi_0)=0$ implies $\xi_0 \in H\{i\ge 0 \mbox{ and } j<0\}$. The second equation 
determines the $H\{i\ge 0 \mbox{ and } j\ge 0\}$-component of $\xi_1$ uniquely but imposes no condition  
on the $H\{i\ge 0 \mbox{ and } j< 0\}$-component of $\xi_1$. The same holds true for all $\xi_k$, $k\ge1$:
\begin{itemize}
\item  The $H\{i\ge 0 \mbox{ and } j\ge 0\}$-{component of } $\xi_k$ is determined by $\xi_{k-1}$. 
\item  The $H\{i\ge 0 \mbox{ and } j< 0\}$-{component of } $\xi_k$  can be chosen arbitrarily.  
\end{itemize}
This immediately shows that the kernel of $\uf$ is isomorphic (but {\sl not} equal!) to 
$H\{i\ge 0 \mbox{ and } j<0\}\otimes \L(t)$. 
As the above system shows, the isomorphism 
$H\{i\ge 0 \mbox{ and } j<0\}\otimes \L(t) \to \kerr(\uf)\subset H\{i\ge0\}\otimes \L(t)$ is given by
the $\L(t)$-equivariant map
$$ \xi \mapsto \xi + \pi_{0} \left( \sum _{\ell \ge 1} (-tJ)^\ell  \xi \right) $$
\end{proof}
\begin{remark} \label{H1-action}
Consider the embedding $X(g,d) \hookrightarrow 
\Lambda^*H^1(\Sigma_g;\mathbb{Z})\otimes \mathbb{Z}[U,U^{-1}]\otimes \Lambda$ from theorem \ref{groups1}:
$$ x\mapsto x - \pi_{k} \left( tU^{-k}J - (t U^{-k}J)^2 x + (t U^{-k}J)^3 x  - ...  \right) $$
%
Notice that the induced action by $H_1(\Sigma_g)$ on $X(g,d)$ is standard in the lowest power of $t$ but typically has 
nonzero ``correction terms'' involving higher powers of $t$. 
However, when $3|k|>g-2$ then all of the terms $(tU^{-k}J)^\ell$ for $\ell \ge 1$ lie in the 
kernel of $\pi_k$ showing that in that range the 
$\Lambda^*H_1(\Sigma_g;\mathbb{Z})$-action has no correction terms. This was already observed by 
Ozsv\'ath-Szab\'o \cite{OSknot} in the case of $\mathbb{Z}$ coefficients.
\end{remark}

\begin{remark}
The isomorphism $HF^+(\Sigma_g\times S^1,\s_k;\Lambda)\cong X(g,d)\otimes \Lambda$ from theorem \ref{groups1}
does not extend to the case of $k=0$ and $\Lambda=L(t)$. With $k=0$ and $\Lambda=L(t)$ the
 infinite system \eqref{infinitesystem}
becomes a finite system which terminates with the equation $\pi_0(J(\xi_m))=0$ for some choice of $m\in \mathbb{N}$. 
This equation breaks the symmetry of the system and imposes additional restraints not satisfied by all 
elements of the form $\pi_0 \sum_{\ell\ge 0} (-tJ)^\ell x$ with  $x\in X(g,g-1)\otimes L(t)$. 
\end{remark}

\begin{remark} It was seen in the proof of the theorem that the homomorphism $F: HF^+(Y,\s; \Lambda)\to HF^+(Y_{-n}, \s_k; \Lambda)$ is surjective in all cases, so that $HF^+(\Sigma\times S^1,\s_k; \Lambda)$ can be thought of as a submodule of $HF^+(Y,\s;\Lambda) = H\{i\geq 0\}\otimes \Lambda$. The latter carries a grading with respect to which $t\in \Lambda$ ($ = L(t)$ or $\L(t)$) carries degree $0$, so we can use this to impose a similar grading on $HF^+(\Sigma\times S^1,\s_k;\Lambda) = X(g,d)\otimes \Lambda$. Equivalently, we grade the latter group by lifting the natural grading on $X(g,d)$, induced by the standard embedding. This grading lifts the relative cyclic grading obtained by forgetting the grading on $R_{Y_0}$ in the definition of twisted-coefficient Floer homology, and has the property, for example, that homogeneous summands are $R_{Y_0}$-submodules. However, it is no longer the case that the action by $H_1(Y_0)$ decreases degree by $1$, or is even homogeneous. We will refer to this alternative grading as the {\em height} in $HF^+(\Sigma\times S^1,\s_k;\Lambda)$.
\end{remark}

In the next section we will have occasion to consider the relative {\OS} invariant of the 4-manifold $\Sigma\times D^2$, for which the next result is central.

\begin{theorem} \label{cobordisminducedmaps1}
Consider the cobordism $W$ from $\Sigma _g \times S^1$ to $S^3$ obtained by removing a small 4-ball from 
$\Sigma _g \times D^2$. For $|k|\le g-1$ let $\rs _k \in Spin^c(W)$ be the unique spin$^c$-structure on $W$ which restricts to 
$\s_k$ on $\Sigma _g \times S^1$. If $k=0$ let $\Lambda=\L(t)$ and if $k\ne 0$ choose $\Lambda$ to be 
either $L(t)$ or $\L(t)$. 
Then the component of the map $\mathcal{F}_k : HF^+(\Sigma_g \times S^1,\s_k;\Lambda) \to 
HF^+(S^3)\otimes \Lambda$ induced by $(W,\rs_k)$ with image in the lowest-degree part $HF^+_0(S^3)\otimes \Lambda$ is given by projection onto the summand of lowest height, corresponding to $H\{(0,-g)\}\otimes \Lambda \cong \Lambda ^0H^1(\Sigma) \otimes U^0 \otimes \Lambda\subset X(g,d)\otimes \Lambda$. 
\end{theorem}
\begin{proof}
We decompose the cobordism $W$ as $W=W_0 \cup W_1\cup ... \cup W_{2g}$ where $W_0$ is the cobordism from 
$\Sigma _g \times S^1$ to $Y=\#^{2g}(S^1\times S^2)$ obtained by attaching a 0-framed 2-handle to the latter 
along the knot $K_g$. The orientation on $W_0$ is the one that induces the orientations 
$\partial W_0 = -(\Sigma _g \times S^1) \sqcup Y$ on its boundary components. The cobordisms 
$W_i$, $i=1,...,2g$ are obtained by the obvious 3-handle additions corresponding to the $2g$ 1-handles of $Y$. 

As explained in section \ref{cobordismsec}, the map $\mathcal{F}_k$ can be calculated by separately calculating the 
contribution from each of the maps $\mathcal{F}_{W_i}$ induced by $W_i$ (the spin$^c$-structure on 
$W_i$ is the restriction of $\rs_k|_{W_i}$ which we omit from the notation for simplicity). 

The map $\mathcal{F}_{W_0}$ is just the map $\uh$ from the sequence \eqref{LEStwisted}, 
it maps $HF^+(\Sigma_g\times S^1,\s_k;\Lambda)$ isomorphically onto the kernel of $\uf$. The latter 
kernel was explicitly identified in the proof of theorem \ref{groups1} and equals the image of the embedding 
of $H\{i\ge 0 \mbox{ and } j <-|k|\} \otimes \Lambda \hookrightarrow \Lambda ^*
\otimes \mathbb{Z}[U^{-1}]\otimes\Lambda$ given by 
$$ x \mapsto x + \pi_k \left(\sum _{\ell\ge 1}  (-tU^{-k}J)^\ell x \right) $$
Indeed, under the identification of 
$HF^+(\Sigma_g \times S^1,\s_k;\Lambda)$ with $H\{i\ge 0 \mbox{ and } j <-|k|\} \otimes \Lambda$ 
(theorem \ref{groups1} and \eqref{xgdidentification}) this 
embedding precisely corresponds to the map $\mathcal{F}_{W_0}$. 

It is a simple matter to see that the homomorphism in Floer homology induced by the composition of 3-handle cobordisms $\#^{2g} S^1\times S^2 \to S^3$ is given by projection onto the lowest-degree factor (and shifting degree up by $g$). The result follows from the above description of the image of $HF^+(\Sigma\times S^1,\s_k;\Lambda)$ in $HF^+(\#^{2g}S^1\times S^2)$.


\end{proof}

\subsection{A surface cross a circle - Universally twisted coefficients}
We will have need for a limited amount of information on the Floer homology of $\Sigma\times S^1$ with ``universal'' coefficients, i.e., coefficients in the group ring $R_{\Sigma\times S^1}$. Continuing our notation from the last section, we let $Y$ be the manifold $Y=\#^{2g}(S^1\times S^2)$. 
An easy application of Theorem \ref{surgeryseq} (for surgery
on the unknot in $S^3$ and with $n=1$) and the connected sum formula for $HF^+$ 
and $\widehat{HF}$ (cf. \cite{OS2}) yields 
$$ \widehat{HF}(\#^{2g}(S^1\times S^2),\s_0;R_Y) \cong \mathbb{Z}_{(-g)} \quad \quad \quad 
HF^+(\#^{2g}(S^1\times S^2),\s_0;R_Y) \cong  \T_{-g}$$
where $\T_n = \zee[U,U^{-1}]/U\cdot \zee[U]$ as before, graded such that the summand of lowest degree lies in degree $n$. Since 
$HF^+(\#^{2g}(S^1\times S^2),\s;R_Y)$ and $\widehat{HF}(\#^{2g}(S^1\times S^2),\s;R_Y)$ are zero for all 
spin$^c$-structures $\s \ne \s_0$, we shall drop the spin$^c$-structure from our notation. Also, we shall 
drop the 3-manifold from our notation for the knot Floer homology groups whenever there is not risk 
of confusion. 


\begin{lemma} \label{b00universal}
Let $g\ge 1$, set $Y=\#^{2g}(S^1\times S^2)$ and let $K_g$ be the nullhomologous knot 
$K_g=\#^g B(0,0)\subset Y$. 
Then for $j\in \{-g,\ldots,g\}$, the twisted knot Floer homology $\widehat{HFK}(K_g, j;R_Y)$ is a free module over $R_Y$ having rank $2g\choose g+j$ and supported in degree $j$,
and is zero for all other values of $j$.
\end{lemma}
The proof of this lemma relies on a filtered version of theorem \ref{surgeryseq}. 
\begin{theorem}[{\OS} \cite{OSknot}]  \label{surgeryknotseq}
Let $K,L \subset Z$ be two nullhomologous knots with linking number 0. Let $K_0$, $K_1$ and $K_{-1}$ be 
the knots in 
$Z_0(L)$, $Z_1(L)$ and $Z_{-1}(L)$ induced by $K$ where $Z_\ell(L)$ is the result of $\ell$-framed surgery on $L$. 
Then for any $\s \in Spin^c(Z)$ and any $R_Z$-module $M$ there are exact sequence of $R_{Z_0(L)}$-modules
\begin{align} \nonumber 
... \to \widehat{HFK}(K,\s,j;M)[t^{\pm 1}]\to \widehat{HFK}(K_0,[\s_k],j;M[t^{\pm 1}])\to
\widehat{HFK}(K_1,\s_k,j;M)[t^{\pm 1}]\to ... \cr
... \to \widehat{HFK}(K,\s,j;M)[t^{\pm 1}]\to \widehat{HFK}(K_{-1},[\s_k],j;M[t^{\pm 1}])\to
\widehat{HFK}(K_0,\s_k,j;M)[t^{\pm 1}]\to ...
\end{align}
\end{theorem}\hfill$\Box$

\begin{proof}[Proof of lemma \ref{b00universal}] 
Lemma \ref{b00universal} follows from repeated applications of theorem 
\ref{surgeryknotseq} to various triples of knots. Our proof is a straightforward 
adaptation of the $\mathbb{Z}$-coefficient proof first obtained by Ozsv\'ath and Szab\'o in \cite{OS6}. 
We first consider the case of  $g=1$. 

The three knots $B(\infty,1)\to B(0,1)\to B(1,1)$ fit into the positive $n$ surgery sequence from 
theorem \ref{surgeryknotseq}. It is easy to see that $B(\infty,1)$ is the unknot in $S^3$ while 
$B(1,1)$ is the right-handed trefoil. Thus  
\begin{align} \nonumber
\widehat{HFK}(B(\infty,1),j)& \cong \left\{
\begin{array}{cl}
\mathbb{Z}_{(0)} & \quad ; \quad j=0 \cr
0 &  \quad ; \quad j\ne 0
\end{array}
\right. \cr
& \cr\end{align} 
and
\begin{align}\nonumber
\widehat{HFK}(B(1,1),j)& \cong \left\{
\begin{array}{cl}
\mathbb{Z}_{(0)} & \quad ; \quad j=1 \cr
\mathbb{Z}_{(-1)} & \quad ; \quad j=0 \cr
\mathbb{Z}_{(-2)} & \quad ; \quad j=-1 \cr
0 &  \quad ; \quad j\ne 0,\pm 1
\end{array}
\right. 
\end{align}
where a subscript $(n)$ indicates that the corresponding module is supported in degree $n$.

Using these in the surgery sequence leads to 
$$ \widehat{HFK}(B(0,1),j;R_{S^1\times S^2}) \cong \left\{
\begin{array}{cl}
(R_{S^1\times S^2})_{\left(\frac{3}{2}\right)} & \quad ; \quad j=1 \cr
(R^2_{S^1\times S^2})_{\left(\frac{1}{2}\right)} & \quad ; \quad j=0 \cr
(R_{S^1\times S^2})_{\left(-\frac{1}{2}\right)} & \quad ; \quad j=-1 \cr
0 &  \quad ; \quad j\ne 0, \pm 1
\end{array}
\right. $$
In a similar vein using the negative $n$ surgery sequence from theorem \ref{surgeryknotseq} for the triple 
$B(\infty,-1)\to B(-1,-1)\to B(0,-1)$ (and observing that $B(-1,-1)$ is the left-handed trefoil) leads to 
$$ \widehat{HFK}(B(0,-1),j;R_{S^1\times S^2}) \cong \left\{
\begin{array}{cl}
(R_{S^1\times S^2})_{\left(\frac{1}{2}\right)} & \quad ; \quad j=1 \cr
(R^2_{S^1\times S^2})_{\left(-\frac{1}{2}\right)} & \quad ; \quad j=0 \cr
(R_{S^1\times S^2})_{\left(-\frac{3}{2}\right)} & \quad ; \quad j=-1 \cr
0 &  \quad ; \quad j\ne 0, \pm 1
\end{array}
\right. $$
For our next set of surgery sequences note that $B(0,\infty)$ is the unknot in $S^1\times S^2$ and therefore
$$ \widehat{HFK}(B(0,\infty),j;R_{S^1\times S^2} ) \cong \left\{
\begin{array}{cl}
\mathbb{Z}_{\left( -\frac{1}{2}\right)} & \quad ; \quad j=0 \cr
0 &  \quad ; \quad j\ne 0
\end{array}
\right. $$
where $\mathbb{Z}$ is the trivial $R_{S^1\times S^2}$-module. 
Using the negative $n$ surgery sequence on the triple $B(0,\infty)\to B(0,-1)\to B(0,0)$ for $j=0$ 
shows that $\widehat{HFK}_{(-1)}(B(0,0),0;R_{\#^2 (S^1\times S^2)}) = 0$. 
The positive $n$ surgery sequence for the triple $B(0,\infty)\to B(0,0) \to B(0,1)$ for $j=\pm 1$ leads 
to 
\begin{align} \nonumber
\widehat{HFK}(B(0,0),j;R_{\#^2 (S^1\times S^2)}) \cong (R_{\#^2 (S^1\times S^2)})_{(j)} \quad \quad j=\pm 1
\end{align}
while $j=0$ yields the sequence
$$ 0 \to \widehat{HFK}_{(0)}(B(0,0),0;R_{\#^2 (S^1\times S^2)}) \to 
(R^2_{\#^2 (S^1\times S^2)})_{\left(-\frac{1}{2} \right)} \to 
(R_{S^1\times S^2})_{\left(-\frac{1}{2} \right)} \to 0 $$
Let $s_1,s_2$ be two generators of $H^1(\#^2(S^1\times S^2);\mathbb{Z})$, then we can write the above as 
$$ 0 \to \widehat{HFK}_{(0)}(B(0,0),0;R_{\#^2 (S^1\times S^2)}) \to 
\mathbb{Z}^2[s_1^{\pm 1}, s_2^{\pm 1}] _{\left(-\frac{1}{2} \right)} \stackrel{f}{\to} 
\mathbb{Z}[s_2^{\pm 1}] _{\left(-\frac{1}{2} \right)} \to 0 $$
Notice that $f$ factors through the quotient 

\centerline{
\xymatrix{
\mathbb{Z}^2[s_1^{\pm 1}, s_2^{\pm 1}] \ar[r]^{f} \ar[d]_\pi  &   \mathbb{Z}[s_2^{\pm 1}]  \\
\frac{\mathbb{Z}^2[s_1^{\pm 1}, s_2^{\pm 1}]}{(s_1-1)\cdot \mathbb{Z}^2[s_1^{\pm 1}, s_2^{\pm 1}]} 
\cong \mathbb{Z}^2[s_2^{\pm 1}]
\ar[ru]_{\tilde f} &  \\
}
}
\vskip1mm
\noindent Here $\pi$ is the map which sends a pair of polynomials 
$(p(s_1,s_2),q(s_1,s_2))\in \mathbb{Z}^2[s_1^{\pm 1},s_2^{\pm 1}]$ to 
$(p(1,s_2),q(1,s_2)\in \mathbb{Z}^2[s_2^{\pm 1}]$. 
Regarding $\tilde f$ as a $\mathbb{Z}[s_2^{\pm 1}]$-module homomorphism between free 
$\mathbb{Z}[s_2^{\pm 1}]$ modules we see that $\kerr(\tilde f) \cong \mathbb{Z}[s_2^{\pm 1}]$ 
inducing a splitting of the domain of $\tilde f$ as $\mathbb{Z}^2[s_2^{\pm 1}]\cong \kerr (\tilde f) \oplus 
\mathbb{Z}[s_2^{\pm 1}]$, the latter two summands generated by pairs of polynomials 
$(a_1(s_2),0)$ and $(0, a_2(s_2))$  with $a_i(s_2)\in \mathbb{Z}[s_2^{\pm 1}]$. 
Setting $a_i(s_1,s_2) = a_i(s_2)$, it becomes an easy exercise to see that the two pairs of polynomials 
$(a_1(s_1,s_2),0)$ and $(0,a_2(s_1,s_2))$ generate 
$\mathbb{Z}^2[s_1^{\pm 1},x_2^{\pm 1}]$. The thus induced splitting of $\mathbb{Z}^2[s_1^{\pm 1},s_2^{\pm 1}]$
is respected by $\pi$ which shows that 
$$\kerr (f) = \pi^{-1}(\kerr \tilde f) \cong 
\mathbb{Z}[s_1^{\pm 1}, s_2^{\pm 1}] \oplus ((s_1-1)\cdot \mathbb{Z}[s_1^{\pm 1}, s_2^{\pm 1}])
\cong R^2_{\#^2 (S^1\times S^2)}$$
completing the proof of lemma \ref{b00universal} when $g=1$. The case of $g>1$ follows from this and 
the connected sum formula for knot Floer homology \cite{OSknot}.
\end{proof}
The results of Lemma \ref{b00universal} can be rewritten in a more concise way as follows: let $M$ 
be the free $R$-module of rank $2g$ (with $R$ still denoting $R_Y$). Then for $j=-g,...,g$ we have $\widehat{HFK}(B(0,0),j;R) \cong \Lambda^{g+j} M$
supported entirely in grading $j$. 

Recall that there is a spectral sequence whose $E^2$ and $E^\infty$ terms equal 
$\widehat{HFK}(K_g,j;R)$ and $\widehat{HF}(\#^{2g}(S^1\times S^2);R) \cong \mathbb{Z}_{(-g)}$ respectively. 
Since the grading of any term in $\widehat{HFK}(K_g,j;R)$ equals $j$, it follows that 
this spectral sequence collapses after the $E^2$ term, in particular the only nonzero differentials in the
spectral sequence are those on the $E^1$ level: 
\begin{equation} \label{Rdifferential} \partial_v: \Lambda^{\ell+1}M \to \Lambda^\ell M \quad \quad \quad \ell = 0,...,2g-1\end{equation}
%
%
In particular, we infer that the chain complex 
\begin{equation}\label{freeres} 0 \to \Lambda^{2g}M\stackrel{\partial_v}{\to} \Lambda^{2g-1}M \stackrel{\partial_v}{\to} \cdots 
\stackrel{\partial_v}{\to} \Lambda^1M \stackrel{\partial_v}{\to} \Lambda^0 M \to 0  \end{equation}
is a (minimal) free resolution of $\mathbb{Z}$ in the category of $R$-modules. 

As before, knowledge of the knot Floer homology of $K_g$ allows calculation of the Floer homology of large integer surgeries along $K_g$.

\begin{prop} \label{hfhatuniversal}
Choose integers $g, n$ with $g\ge 2$  
and $n\gg 0$. Let $\varepsilon : R_Y\to \mathbb{Z}$ be the 
augmentation homomorphism sending each element of $H^1(Y)$ to $1$. Then 
\begin{align} \nonumber
\widehat{HF}(Y_{n} ,\s_{\pm(g-1)};R_Y) & \cong R_{(g-1)} \oplus R_{(g-2)} \oplus \mathbb{Z}_{(-g)} \cr
HF^+(Y_{-n},\s_{\pm (g-1)};R_Y) & \cong \kerr_{(-g+1)} (\varepsilon) \oplus \T_{-g+2} \cr
HF^+_d(Y_{-n},\s_k;R_Y) & \cong 0  
\end{align}
where the third line is valid for all $d<-|k|$ and $|k|\ne g-1$. 
The grading used is the absolute $\mathbb{Q}$-grading for ${HF}^\circ(Y_{\pm n},\s_k;R_Y)$
shifted by $\frac{n-(2k-n)^2}{4n}$ for convenience.
\end{prop}
\begin{proof} Just as in the case of $L(t)$ coefficients considered previously, we can obtain information on the fully-twisted Floer homology of large-$n$ surgeries from the knot complex. To be explicit, we have from \cite{OSknot},
\begin{align} \label{hfandhfkuniversal}
\widehat{HF}(Y;R_Y) & \cong H\{i=0\}  \cong H\{j=0\} \cr
HF^+(Y;R_Y) & \cong H\{ i\ge 0\} \cr
\widehat{HF}(Y_n,\s_k;R_Y) & \cong H\{\max (i,j-k) = 0 \} \cr
\widehat{HF}(Y_{-n},\s_k;R_Y) & \cong H\{\min (i,j-k) = 0 \} \cr
HF^+(Y_{-n},\s_k;R_Y) & \cong H\{\min(i,j-k) \ge 0 \}
\end{align}
in the obvious notation. For each of the cases above there is a spectral sequence converging to  the desired homology coming from the filtration on the knot complex given by $i+j$; in each case the $E^1$ term is the appropriate sub-quotient of $\widehat{HFK}(K_g; R_Y)\otimes \zee[U,U^{-1}]$. In the untwisted case there are no further differentials in this spectral sequence. In the case at hand that is no longer true, but for dimensional reasons there can be no differentials past $d^1$.

The latter differential decomposes into ``vertical'' and ``horizontal'' components; the vertical component is $\partial_v$ from \eqref{Rdifferential} and we write $\partial_h$ for the horizontal component. Note that in any given row, $\partial_h$ also gives the maps in a free resolution of $\zee$ over $R_Y$.  See figure \ref{pic2} for an illustration of these concepts.
\begin{figure}[htb!] 
\centering
\includegraphics[width=14cm]{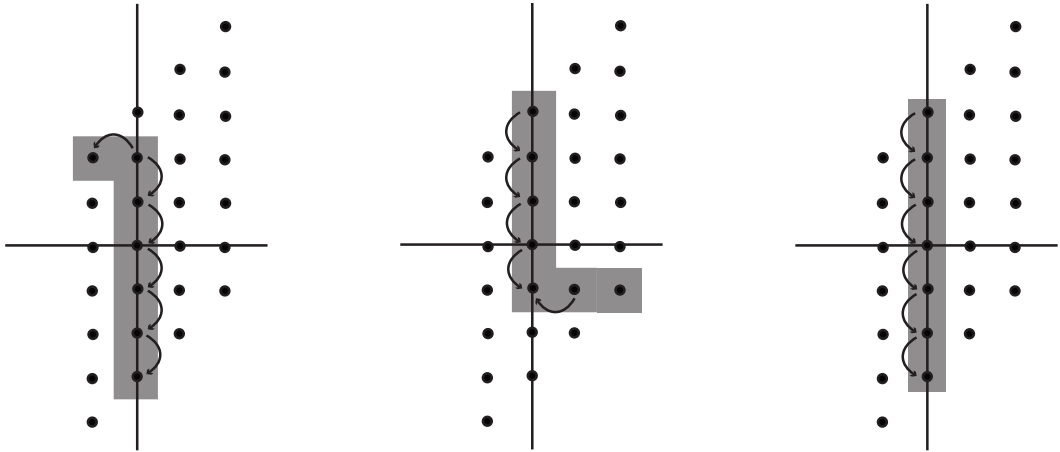}
\caption{The complexes $C\{\max (i,j-k)=0\}$ (left) and $C\{\min(i,j-\ell)=0\}$ (center) and 
$C\{i=0\}$ (right) for $g=3$ and 
$k=g-1$, $\ell=-g+2$. The vertical arrows indicate the nontrivial components $\partial_v$ of $d^1$ while 
the horizontal arrows represent $\partial _h$. All other $d^r$ 
for $r\ge 2$ are zero for dimensional reasons. The homology of $C\{i=0\}$ is $\mathbb{Z}_{(-g)}$ 
occurring at coordinate $(0,-g)$.}  \label{pic2}
\end{figure}
Comparing the homology of $C\{i=0\}$ to that of $C\{i\ge 0\}$ one finds that 
\begin{align} \label{comparingimages}
\imm (\partial_v : \Lambda ^{k} M \otimes U^\ell \to \Lambda ^{k-1} M \otimes U^\ell ) & = 
\imm (\partial _h : \Lambda ^{k-2}\otimes U^{\ell-1} \to \Lambda ^{k-1} M \otimes U^\ell  ) \cr
\kerr (\partial_v : \Lambda ^{k} M \otimes U^{\ell -1} \to \Lambda ^{k-1} M \otimes U^{\ell -1} ) & = 
\kerr (\partial _h : \Lambda ^{k}\otimes U^{\ell -1} \to \Lambda ^{k+1} M \otimes U^{\ell}  ) 
\end{align}
for all $\ell \le 0$ and for any $k$.

To see that the first isomorphism of the statement holds, we need to compute the homology
$H\{\max (i, j-g+1) = 0 \} $ (where we have chosen to consider the spin$^c$-structure $\s_{g-1}$ for concreteness).
Since most of $C\{\max (i,j-g+1) =0\}$ agrees with $C\{i=0\}$ (figure \ref{pic2}), most of their homologies 
agree as well. In fact since $C\{ \max (i,j-g+1)=0 \mbox{ and } j\le g-2\} \cong 
C\{i=0 \mbox{ and } j\le g-2\}$ we find that 
$$ \oplus _{d\le g-3} \widehat{HF}_{(d)} (Y_n,\s_{g-1};R_Y) \cong 
 \oplus _{d\le g-3} \widehat{HF}_{(d)} (Y;R_Y) \cong 
\mathbb{Z}_{(-g)} $$ 
When $d=g-1$ the homology of $C_{(d)}\{\max(i,j-g+1)=0\}$ is (with the help of \eqref{comparingimages})
$$\kerr \left( \Lambda ^{g-1}M \to \Lambda ^{g-2}M \oplus (\Lambda^gM\otimes U) \right)  
= \imm \left( \Lambda ^g M  \to \Lambda ^{g-1}M \right) 
\cong \Lambda ^g M \cong R_Y $$
In the remaining case of $d=g-2$ the relevant homology group is 
$$ \frac{(\Lambda ^gM \otimes U) \oplus \kerr (\Lambda ^{g-2}M \to \Lambda ^{g-3}M)}{\imm \left(\Lambda ^{g-1}M \to 
(\Lambda ^gM\otimes U)\oplus \Lambda ^{g-2}M\right)} \cong \Lambda ^gM \otimes U \cong R_Y$$
as follows from \eqref{comparingimages}. 

To deduce the second isomorphism from the proposition we next calculate the homology $H\{\min(i,j+g-1)\ge 0\}$
(where we have chosen $k=-g+1$). Observe that the portions of the complexes 
$C\{\min(i,j+g-1)\ge 0\}$ and $C\{i\ge 0\}$ lying in grading  $gr \ge -g+2$ are isomorphic thus leading
to the summand $\T_{-g+2}$ of 
$H\{\min(i,j+g-1)\ge 0\}$. In the remaining grading $-g+1$ we have from \eqref{Rdifferential} and \eqref{freeres}
\begin{align}\nonumber 
HF^+(Y_{-n},\s_{-g+1};R_Y) & \cong \frac{\Lambda ^1M}
{\mbox{Im}\left( \partial_v\oplus \partial_h : \Lambda ^2M \oplus (\Lambda ^0M\otimes U) \to \Lambda^1 M\right)} \cr
& \cong \frac{\Lambda ^1M}
{\mbox{Im}\left( \partial_v: \Lambda ^2M  \to \Lambda^1 M\right)} \cr
& \cong \mbox{Im}\left( \partial _v :\Lambda^1M \to \Lambda^0 M  \right) \cr
& \cong \kerr (\varepsilon : \Lambda ^0 M \to \mathbb{Z}) 
\end{align}
as claimed. 
Finally, the third isomorphism of the proposition follows directly from \eqref{hfandhfkuniversal}. 
\end{proof}

As in the previous subsection, our interest in the fully-twisted Floer homology of $\Sigma_g\times S^1$ is focused on that component that maps nontrivially to the lowest-degree Floer homology of $S^3$ under the cobordism $\Sigma\times D^2\setminus B^4$. To that end, we have the following. 

\begin{theorem} \label{universalgroups1}
Assume $g\ge 2$ 
and choose an integer $k$ with $|k|\le g-1$. 
Consider the tail portion of the negative $n$ surgery sequence \eqref{LEStwisted} (with $n\gg 0$ and 
with grading conventions 
as in proposition \ref{hfhatuniversal})
$$... \to HF_{(-g+1)}^+(Y_{-n},\s_k;R)[t^{\pm 1}]\stackrel{G}{\to} 
HF^+_{(-g+1/2)}(\Sigma _g \times S^1,\s_k;R_{\Sigma _g \times S^1})\stackrel{H}{\to} 
HF^+_{(-g)}(Y; R)[t^{\pm 1}]\to 0$$ 
where the subscript $(-g+1/2)$ in the middle term serves merely as a label rather than a reference to a 
grading (except in the case when $k=0$ where it coincides with the $\mathbb{Q}$-grading). Then 
$$ HF ^+_{(-g+1/2)}(\Sigma _g \times S^1,\s_k;R_{\Sigma _g \times S^1}) \cong 
\left\{ 
\begin{array}{cl}
L(t) & \quad ; \quad |k| \ne g-1   \cr
& \cr
R_{\Sigma _g \times S^1} & \quad ; \quad  |k| =  g-1 
\end{array} \right.
$$
Moreover, $H$ is an isomorphism when $|k|\ne g-1$ while in the case $|k|=g-1$, $H$ is the quotient map   
$$ R_{\Sigma_g\times S^1}\to L(t)$$
induced by $\alpha\mapsto t^{\langle \alpha, [K_g]\rangle}$ for $\alpha\in H^1(\Sigma_g\times S^1)$. 
\end{theorem}
As previously we shall refer to 
$HF^+_{(-g+1/2)}(\Sigma _g \times S^1,\s_k;R_{\Sigma _g \times S^1})$ as the submodule of {\em lowest height}.

\begin{proof}
The claim of the theorem for $|k|\ne g-1$ follows at once from proposition \ref{hfhatuniversal} 
and the negative $n$ surgery sequence \eqref{LEStwisted}. 

The same proposition and surgery sequence for $k=-g+1$ lead to the exact sequences
\[
0\to \left( \kerr (\varepsilon)\right)[t^{\pm 1}]_{(-g+1)} \to HF^+_{(-g+1/2)}(\Sigma _g \times S^1,\s_{-g+1};R_{\Sigma _g \times S^1}) 
\to L(t)_{(-g)} \to 0 \]
in the lowest degree, and 
\[
0 \to HF^+_{h}\Sigma _g \times S^1,\s_{-g+1};R_{\Sigma _g \times S^1}) \to L(t)_{(-g+2\ell)} \stackrel{\cong}{\to}
L(t)_{(-g+2\ell)} \to 0 
\]
for summands $HF^+_{h}\Sigma _g \times S^1,\s_{-g+1};R_{\Sigma _g \times S^1})$ of non-minimal height $h$ and any $\ell \ge 1$. We conclude from this that 
$$ HF^+(\Sigma _g \times S^1,\s_{-g+1};R_{\Sigma _g \times S^1}) = 
HF^+_{(-g+1/2)}(\Sigma _g \times S^1,\s_{-g+1};R_{\Sigma _g \times S^1})$$
and since the action of $\mathbb{Z}[U]$ on either side is trivial we also obtain the isomorphism 
$$\widehat{HF}(\Sigma _g \times S^1,\s_{-g+1};R_{\Sigma _g \times S^1})\cong HF^+(\Sigma _g \times S^1,\s_{-g+1};R_{\Sigma _g \times S^1})^{\oplus 2},
$$
i.e., $\widehat{HF}$ is two adjacent copies of $HF^+$. 
Unfortunatly the negative $n$ surgery sequence doesn't allow for an effective calculation of 
$\widehat{HF}(\Sigma _g \times S^1,\s_{-g+1};R_{\Sigma _g \times S^1})$, it leads to 
a nontrivial extension problem. 
However, the positive $n$ surgery sequence \eqref{LEStwisted} 
for $\widehat{HF}$ in conjunction with 
proposition \ref{hfhatuniversal}  yield the exact sequence 
\[
0\to 
 \widehat{HF} _{\left( -g+\frac{1}{2} \right)}(\Sigma_g \times S^1,\s_{g-1};R_{\Sigma _g \times S^1}) \to (R_{(g-1)}\oplus R_{(g-2)}\oplus\zee_{(-g)})[t^{\pm 1}] \rTo  \mathbb{Z}_{(-g)}[t^{\pm 1}]\to 0
\]
(where $R = R_{\#^{2g}S^1\times S^2}$).  It follows from the sequence in $HF^+$ that the last map is an isomorphism on the $\zee_{(-g)}[t^{\pm 1}]$ factors, and the result follows.

The case of $k=-g+1$ follows a similar argument.  
\end{proof}
\section{Product Formulae}\label{productformulasec}

We can now piece together the ingredients of the preceding sections to deduce the results stated in the introduction. The conceptual plan is reasonably straightforward: if $X = M_1\#_\Sigma M_2$ is a fiber sum as in the introduction, then the {\OS} invariants for $X$ are given by a pairing between relative invariants for $Z_i = M_i\setminus (\Sigma\times D^2)$, after perturbing by a class $\eta\in H^2(X;\arr)$ that restricts nontrivially to $\Sigma \times S^1$, according to Theorem \ref{pertproductthm}. 

To determine the relative invariants $\Psi_{Z_i,\eta}$ of the pieces and obtain a formula for $OS_X$ in terms of $OS_{M_i}$, we observe that the {\OS} invariants of the $M_i$ are themselves determined by the pairing between $\Psi_{Z_i,\eta}$ and the relative invariant of $\Sigma\times D^2$, again using the perturbed version of Floer theory since, even if $b^+(Z_i)>0$, we have $b^+(\Sigma\times D^2) = 0$ and Theorem \ref{productthm} need not apply. Hence we need to understand the perturbed relative invariant $\Psi_{\Sigma\times D^2, \eta}$, as well as the relevant pairing on Floer homology. 

Now, it is easy to see that the coefficient module for $\Sigma\times S^1$ induced by $\Sigma\times D^2$ (with a 4-ball removed) is $\zee[K(\Sigma\times D^2)] = L(t)$, where $t\in H^1(\Sigma\times S^1;\zee)$ is a generator Poincar\'e dual to $\Sigma\times pt$. There is little choice in the perturbation $\eta$ on $\Sigma\times D^2$; namely we take $\eta$ to be (a positive multiple of the) Poincar\'e dual to the relative class $pt\times D^2$, which has $\langle t\cup \eta, [\Sigma\times S^1]\rangle = 1$. Thus the Novikov completion of $L(t)$ with respect to $\eta$ is the ring $\L(t)$ of Laurent series in $t$.

Assuming $[\Sigma]$ to be (rationally) nontrivial in $M_1$ and $M_2$, we can extend $\eta$ to $M_i$, and consider the relative invariants of the complements $Z_i$. In particular, if $\K(Z_i,\eta)$ is the module for $\R_{\Sigma\times S^1,\eta}$ induced by $(Z_i,\eta)$, we are interested in the pairing
\begin{align*}
HF^-_\Dot(\Sigma\times S^1, \s; \K(Z_i,\eta))\otimes \overline{HF^-_\Dot(-\Sigma\times S^1, \s; \L(t))} &\to \K(M_i,\Sigma\times S^1,\eta)\\
\xi_1\otimes \xi_2 & \mapsto \langle \tau^{-1}(\xi_1),\,\xi_2\rangle
\end{align*}
between the perturbed Floer homologies. In fact, more specifically we are interested in the homomorphism $HF^-_\Dot(\Sigma\times S^1, \s; \K(Z_i,\eta))\to \K(M_i, \Sigma\times S^1, \eta)$ induced by the pairing above when $\xi_2$ is equal to $\Psi_{\Sigma\times D^2, \eta}$, the relative invariant for $\Sigma\times D^2$. 

Now, it is a simple exercise to see that $K(M_i, \Sigma\times S^1)$ is cyclic, generated by the Poincar\'e dual of $[\Sigma]$ in $H^2(M_i)$. Since $\Sigma$ is assumed to represent a non-torsion class in each of $M_1$ and $M_2$, then, we have $\K(M_i,\Sigma\times S^1,\eta) \cong \L(t)$. There is a natural surjection 
\[
\rho: K(Z_i)\cong \frac{H^1(\Sigma\times S^1)}{H^1(Z_i)}\to \frac{H^1(\Sigma\times S^1)}{H^1(Z_i) + H^1(\Sigma\times D^2)} \cong K(M_i, \Sigma\times S^1),
\]
inducing a surjection $\rho: \K(Z_i,\eta) \to \L(t)$, and a commutative diagram
\begin{equation}\label{redcoeffdiag}
\begin{diagram}
HF^-_\Dot(\Sigma\times S^1; \K(Z_i, \eta)) & \rTo^{\langle\tau^{-1}(\cdot),\Psi_{\Sigma\times D^2, \eta}\rangle} & \K(M_i, \Sigma\times S^1, \eta) \cong \L(t)\\
\dTo_{\rho_*} && \dTo^{id}\\
HF^-_\Dot(\Sigma\times S^1; \L(t)) &\rTo^{\langle\tau^{-1}(\cdot),\Psi_{\Sigma\times D^2, \eta}\rangle}  & \L(t)
\end{diagram}
\end{equation}
(c.f. Lemmas \ref{lemma1} and \ref{lemma2}). We will see that the arrow on the bottom of this diagram is determined essentially uniquely by algebraic considerations. Hence, determining the pairing mentioned above is equivalent to understanding the change-of-coefficient map $\rho_*$, and the relative invariant $\Psi_{\Sigma\times D^2,\eta}$. 

Naturally, we cannot hope to do better than determining these objects up to a unit in $\L(t)$. Since units abound in a power series ring, this is not necessarily sufficient. However, we know that all the algebra we must use in the context of Novikov rings is induced from corresponding algebra over ordinary group rings: that is, the perturbed case is an obvious Novikov completion of the unperturbed case. Since $L(t)$ has many fewer units than $\L(t)$, this is a useful observation: we work initially in twisted, but unperturbed, coefficients.

\subsection{Relative invariants in case $g = 1$} 

Let $M$ be a closed $4$-manifold with $b^+(M)\geq 1$ containing a smoothly embedded torus $T\hookrightarrow M$ with trivial normal bundle. We assume that $[T]$ is an element of infinite order in $H_2(M)$. Write $T\times D^2$ for a tubular neighborhood of $T$, and let $Z = M\setminus (T\times D^2)$ be the complement of this neighborhood. We wish to understand the relationship between the {\OS} invariants of $M$ and the relative invariant of $Z$. If $b^+(M) = 1$, we will be interested in the invariant $\O_{M,T\times S^1, \eta}$, where $\eta\in H^2(M;\arr)$ is a class that restricts to $T\times D^2$ as a nonzero multiple of the Poincar\'e dual of the relative class $[pt \times D^2]$. 

Recall the following result of Ozsv\'ath and Szab\'o.

\begin{theorem}[\cite{OS4}]\label{OStwistT3} The twisted Heegaard Floer homology $HF^+(T^3, \s; R_{T^3})$ is trivial unless $\s$ is equal to the unique \spinc structure $\s_0$ with $c_1(\s_0) = 0$. In this case, there is an isomorphism 
\[
HF^+(T^3, \s_0;R_{T^3}) = \T_{1/2}\oplus \ker(\varepsilon),
\]
where $\T_{1/2}$ is the $\zee[U]$-module $\zee[U,U^{-1}]/ U\cdot\zee[U]$, graded so that its homogeneous summand of least degree lies in dimension $1/2$, and 
\[
\varepsilon: R_{T^3} = \zee[H^1(T^3)]\to \zee
\]
is the augmentation homomorphism that maps each element of $H^1(T^3)$ to $1$. In the above, $\ker(\varepsilon)$ lies in degree $-1/2$.
\end{theorem}

In particular, the reduced Floer homology in the fully-twisted case is $HF^+_{red}(T^3, \s_0; R_{T^3}) = \ker(\varepsilon)$, lying entirely in degree $-1/2$, where $\s_0$ is the torsion \spinc structure.

\begin{prop}\label{redcoeffprop} Let $Z = M\setminus (T\times D^2)$ be the complement of an essentially embedded torus in a 4-manifold as above, and let $K(Z) = \ker(H^2(Z, \partial Z)\to H^2(Z))$ as usual, so that $K(Z) \cong H^1(T^3)/ H^1(Z)$. Then $HF^+_k(T^3, \s_0; \zee[K(Z)]) = 0$ if $k < -1/2$, and there is an isomorphism
\[
HF^+_{-1/2}(T^3, \s_0; \zee[K(Z)]) \cong \ker(\varepsilon)\otimes_{R_{T^3}} \zee[K(Z)].
\]
The change-of-coefficient map $\rho_*: HF^+(T^3, \s_0; \zee[K(Z)])\to HF^+(T^3, \s_0; L(t))$ is given by the natural map
\[
id \otimes \rho: \ker(\varepsilon)\otimes_{R_{T^3}} \zee[K(Z)]\to \ker(\varepsilon)\otimes_{R_{T^3}} L(t)
\]
\end{prop}

\begin{proof} For an $R_{T^3}$-module $M$, there is a ``first quadrant'' universal coefficients spectral sequence converging to $HF^+(T^3, \s_0; M)$, whose $E_2$ term is $\tor^j(HF^+_i(T^3; R_{T^3}), M)$. In particular, the group in lowest total degree in the $E_2$ term is 
\[
\tor^0(HF^+_{-1/2}(T^3; R_{T^3}), M) = HF^+_{-1/2}(T^3; R_{T^3})\otimes M = \ker(\varepsilon)\otimes M.
\]
From the structure of the differentials in the spectral sequence this group must survive as the lowest-degree part of $E_\infty$, which proves the first statement. 

 If $\rho: M\to M'$ is a module homomorphism then we get a corresponding  map of spectral sequences, for which the map on the $E_2$ term is $id\otimes \rho$ on the $j = 0$ row. The second statement of the theorem follows as before, since we consider only the bottom-degree groups.
\end{proof}

Note that although the fully-twisted Floer homology for $T^3$ in dimension $-1/2$ is precisely equal to the reduced Floer homology, the same is not true in other coefficient systems (indeed, with untwisted coefficients, the reduced Floer homology is trivial). For example, there is an isomorphism 
\[
HF^+_{-1/2}(T^3,\s_0; L(t)) = \ker(\varepsilon)\otimes L(t) =\zee\oplus \zee\oplus L(t),
\]
and in this decomposition only the $L(t)$ factor lies in the reduced Floer homology. However, there is a natural projection $HF^+\to HF^+_{red}$ for any 3-manifold and any coefficient module. When we apply $\rho_*$ we often implicitly compose with this projection without including it in the notation, and hope this will not cause confusion; note that this problem disappears when we pass to perturbed Floer homology.

In the special case $K(Z) = H^1(T^3)$ (equivalently, the restriction $H^1(Z)\to H^1(T^3)$ is trivial), we can identify the change-of-coefficient map (in reduced Floer homology) with the surjection
\begin{align*}
H: \ker(\varepsilon)&\to L(t)\\
a(r,s,t) &\mapsto \frac{a(1,1,t)}{t-1}.
\end{align*}
Indeed, thinking of $\zee[K(Z)] = R_{T^3}$ as the ring of Laurent polynomials in three variables $r, s, t$, we have that $\ker(\varepsilon)$ is the ideal generated by the three elements $r-1$, $s-1$, and $t-1$. In particular, if $a(r,s,t)\in \ker(\varepsilon)$, then $a(1,1,t)$ is divisible by $t-1$; the given map is uniquely determined up to units in $R_{T^3}$.

With this understanding of the change of coefficients in the unperturbed case, we can now introduce a perturbation $\eta$. As before, we choose any $\eta\in H^2(M;\arr)$ such that the restriction of $\eta$ to $T^2\times D^2$ is Poincar\'e dual to $pt\times D^2$. To understand the relative invariant $\Psi_{T^2\times D^2,\eta}$, observe first that $HF^+(T^3,\s;\M)$ is trivial for any $\R_{T^3}$-module $\M$ unless $\s = \s_0$, the unique \spinc structure with $c_1(\s_0) = 0$. Thus from now on we consider only \spinc structures on $M$ that restrict to $\s_0$ on $T^3$, which means also that $c_1(\s)|_{T^2\times D^2} = 0$. It is straightforward to see that in the fully-twisted case, $HF^+(T^3,\s_0; \R_{T^3,\eta}) \cong \R_{T^3,\eta}$, using Lemma \ref{flatlemma} and Theorem \ref{OStwistT3}.

Consider the complement $Z = M\setminus (T^2\times D^2)$ as a cobordism $S^3\to T^3$ by removing a 4-ball (we still use the symbol $Z$ for this cobordism). In this situation the diagram \eqref{pertdiag} becomes
\begin{diagram}
HF^-(S^3) & \rTo^{F^-_Z} & HF^-(T^3; \zee[K(Z)])\\
&\rdTo_{F^-_{Z,\eta}} & \dTo_{i_*}\\
&& HF^-_\Dot(T^3; \K(Z,\eta)),
\end{diagram}
where $i_*$ is the homomorphism induced by the natural inclusion $i: \zee[K(Z)]\to \K(Z,\eta)$.  If $K(Z) = H^1(T^3)$ (e.g., if the complement of $T$ in $M$ is simply-connected), then in the lowest nontrivial degree, $i_*$ is a map $\ker(\varepsilon)\to \R_{T^3}$. This is induced by the natural homomorphisms
\begin{diagram}
HF^-(T^3;R_{T^3}) & \rTo & HF^-(T^3; R_{T^3})\otimes \R_{T^3,\eta} & \stackrel{\sim}{\longto} & HF^-(T^3; R_{T^3}\otimes \R_{T^3,\eta}) &\rTo& HF^-_\Dot(T^3; \R_{T^3,\eta})\\
[x] & \rMapsto& [x]\otimes 1 &\rMapsto&&& [x\otimes 1],
\end{diagram}
so we can think of $i_*$ as the homomorphism mapping a Laurent polynomial $a(r,s,t)\in \ker(\varepsilon)$ into the Laurent series ring $\R_{T^3,\eta}$ by the natural inclusion.

Consider the diagram
\begin{equation}\label{twistpairdiag}
\begin{diagram}
HF^+_{-1/2}(T^3;R_{T^3}) & \rTo & L(t)\\
\dTo_{i_*} && \dTo\\
HF^+_{-1/2}(T^3; \R_{T^3,\eta}) & \rTo & \L(t)\\
\dTo_{\rho_*} &\ruTo_{\mbox{\tiny $\langle\cdot, \Psi_{T^2\times D^2,\eta}\rangle$}} & \\
HF^+_{-1/2}(T^3; \L(t)) 
\end{diagram}
\end{equation}
where the upper arrow is $a \mapsto \langle a, F^-_{T^2\times D^2}(\Theta^-)\rangle$, and the middle arrow is $b \mapsto \langle b, \Psi_{T^2\times D^2, \eta}\rangle$. In the unperturbed  case, we have $\langle a, F^-_{T^2\times D^2}(\Theta^-)\rangle = \langle \rho_*(a), F^-(\Theta^-)\rangle  = \langle F^+_{T^2\times D^2}(\rho_*(a)), \Theta^-\rangle$ where $\rho_*$ is the projection to $L(t)$ coefficients. 
From previous work we know that in $L(t)$ coefficients, $F^+_{T^2\times D^2}$ induces a surjection to $L(t)$. It follows from (the unperturbed analog of) \eqref{redcoeffdiag} and Proposition \ref{redcoeffprop} that the upper arrow in \eqref{twistpairdiag} is given by
\[
a \mapsto \langle a, F^-_{T^2\times D^2}(\Theta^-)\rangle = H(a),
\]
where $H: \ker(\varepsilon)\to L(t)$ is the homomorphism introduced above.

On the other hand, if we think of $a\in \ker(\varepsilon)$ as a Laurent polynomial $a(r,s,t)$ then $i_*(a) = a$. The coefficient-change map $\rho_*:HF^+(T^3; \R_{T^3,\eta})\cong \R_{T^3,\eta}\to HF^+(T^3;\L(t))\cong \L(t)$ is necessarily the reduction $b(r,s,t)\mapsto b(1,1,t)$, so that $\rho_*(i_*(a)) = a(1,1,t)$. 

Thus, $\Psi_{T^2\times D^2,\eta}(1)$ is a generator of $HF^-_{-1/2}(T^3;\L(t))$, satisfying the property that for $a(r,s,t) \in \ker(\varepsilon)$,
\[
\langle a, \Psi_{T^2\times D^2, \eta}(1)\rangle = H(a) = \frac{a(1,1,t)}{t-1},
\]
up to a unit in $L(t)$. Identifying $HF^\pm_{-1/2}(T^3; \L(t)) = \L(t)$, the diagonal arrow in the preceding diagram can be taken to be multiplication by $\Psi_{T^2\times D^2,\eta}(1)$. We conclude:

\begin{prop}\label{g1relinvtprop} The relative invariant $\Psi_{T^2\times D^2,\eta}$ can be identified with the map $\A(T^2\times D^2)\to HF^-(T^3,\s_0;\L(t))\cong \L(t)$ whose value on the element $1$ is 
\[
\Psi_{T^2\times D^2,\eta}(1) = \frac{1}{t-1} 
\]
up to multiplication by $\pm t^n$, and which vanishes on elements of $\A(T^2\times D^2)$ having nonzero degree.  

Let $M$ be a closed 4-manifold containing an embedded torus $T$ with trivial normal bundle, and $\eta\in H^2(M;\arr)$ a class with $\int_T\eta >0$. Let $\s$ be a \spinc structure on $M$, and write $Z = M\setminus (T\times D^2)$. If $\langle c_1(\s) , [T]\rangle \neq 0$, then the relative invariant $\Psi_{Z,\eta,\s}$ and the closed invariant $\O_{M,T\times S^1, \eta,\s}$ both vanish.

If $\langle c_1(\s), [T]\rangle = 0$, then the relative invariant $\Psi_{Z,\eta,\s}$ takes values in the Novikov ring $\K(Z,\eta)$. Furthermore, the value of the perturbed {\OS} invariant  on a class $\alpha_1\otimes \alpha_2$ with $\alpha_1\in \A(Z)$, $\alpha_2\in \A(T^2\times D^2)$, is given up to multiplication by $\pm t^n$ by
\[
\O_{M, T^3,\eta,\s}(\alpha_1\otimes \alpha_2) = \langle \tau^{-1}\Psi_{Z,\eta,\s}(\alpha_1), \Psi_{T^2\times D^2,\eta,\s}(\alpha_2)\rangle =\left\{\begin{array}{ll}  \displaystyle\frac{1}{t-1} \rho(\Psi_{Z,\eta,\s}(\alpha_1)) & \mbox{if $\alpha_2 = 1$} \vspace{1ex}\\ 0 & \mbox{otherwise,}\end{array}\right.
\]
where $\rho$ is the natural map $\K(Z,\eta)\to \L(t)$ induced by the projection $K(Z)\to K(M, T^3) \cong \zee$.
\end{prop}

\begin{proof} First, observe that since the Floer homology $HF^-_\Dot(T^3,\s_0;\M)$ is supported entirely in degree $-1/2$ for any $\R_{T^3,\eta}$-module $\M$, the elements of $\A(T^2\times D^2)$ having nonzero degree must act trivially. Thus, only $\Psi_{T^2\times D^2,\eta}(1)$ can be nontrivial.

 If $H^1(Z)\to H^1(T^3)$ is trivial, then $\K(Z,\eta) = \R_{T^3,\eta}$ and the result follows from the preceding discussion. In the general case, observe that since $HF^-_\Dot(T^3,\s_0; \R_{T^3,\eta})\cong \R_{T^3,\eta}$ is free over $\R_{T^3,\eta}$, we have 
\[
HF^-_\Dot(T^3,\s_0;\K(Z,\eta))\cong \R_{T^3,\eta} \otimes_{\R_{T^3,\eta}} \K(Z,\eta) = \K(Z,\eta).
\]
Likewise, by identifying $HF^+(T^3, \s_0;\zee[K(Z)])$ with $\ker(\varepsilon)\otimes \zee[K(Z)]$ as in Proposition \ref{redcoeffprop}, it is straightforward to see that the pairing with $\Psi_{T^2\times D^2,\eta}$ behaves as indicated.
\end{proof}

\subsection{Relative invariants in case $g>1$}

We follow an outline similar to the previous subsection; as before, we begin with twisted but unperturbed coefficients.

Let $\Sigma\hookrightarrow M$ be an embedded surface of square 0 and genus $g$ and $\s\in Spin^c(M)$ a \spinc structure. We have seen that unless $c_1(\s)|_{\Sigma\times S^1}$ is Poincar\'e dual to $2k[pt \times S^1]$ with $|k|\leq g-1$, the Floer homology $HF^+(\Sigma\times S^1,\s; N) = 0$ for any $R_{\Sigma\times S^1}$-module $N$ (and similarly after perturbation), forcing $\O_{M,\Sigma\times S^1,\eta,\s} = 0$ for such $\s$. Thus we suppose that the restriction of $\s$ has the indicated form; we write $\s_k$ for the \spinc structure on $\Sigma\times S^1$ with $ c_1(\s_k) = 2k\,PD[S^1]$.

The following are easy consequences of Theorems \ref{groups1} and \ref{cobordisminducedmaps1}.

\begin{lemma} Let $N$ be a module for $R_{\Sigma\times S^1}$. Then when $|k| = g-1$, we have an isomorphism 
\[
HF^+_{red}(\Sigma\times S^1, \s_k; N) = HF^+(\Sigma\times S^1, \s_k; N)  = N.
\]

If $Z = M\setminus (\Sigma\times D^2)$ is the complement of a surface representing a class of infinite order in $H_2(M;\zee)$, then the homomorphism
\[
\rho_*: HF^+_{red}(\Sigma\times S^1,\s_k; \zee[K(Z)])\to HF^+_{red}(\Sigma\times S^1, \s_k; L(t))
\]
is equal to the projection $\rho: \zee[K(Z)]\to \zee[K(M, \Sigma\times S^1)] = L(t)$.
\end{lemma}

\begin{lemma} Let $N$ be a module for $R = R_{\Sigma\times S^1}$, and suppose $|k|< g-1$. Then there is an isomorphism
\[
HF^+_{bot}(\Sigma\times S^1,\s_k; N) = N\otimes_R L(t),
\]
where $HF^+_{bot}$ denotes the component in the lowest height, and furthermore, this summand lies in the reduced Floer homology. Here $t\in H^1(\Sigma\times S^1;\zee)$ is dual to $[\Sigma]$ and $L(t)$ is an $R$-module in the usual way. 

In particular, if $Z = M\setminus (\Sigma\times D^2)$ as above then for $|k|< g-1$,
\[
HF^+_{bot}(\Sigma\times S^1, \s_k; \zee[K(Z)]) = L(t),
\]
and the map $\rho_*: HF^+_{bot}(\Sigma\times S^1, \s_k; \zee[K(Z)])\to HF^+_{bot}(\Sigma\times S^1, \s_k; L(t))$ is the identity.
\end{lemma}

As in the previous section, ``lowest height'' refers to the $R_{\Sigma\times S^1}$-submodule of the Floer homology that maps to the component of lowest degree in the Floer homology of $\#^{2g}S^1\times S^2$.

The results above follow easily from the fact that for $g>1$,
\[
HF^+_{bot}(\Sigma\times S^1, \s_k; R) = \left\{\begin{array}{ll} R & \mbox{if $|k| = g-1$}\\ L(t) & \mbox{if $|k|<g-1$,}\end{array}\right.
\]
together with formal arguments as in the proof of Proposition \ref{redcoeffprop}.

Applying a perturbation $\eta\in H^2(M;\arr)$ restricting as a positive multiple of the Poincar\'e dual of $[pt\times D^2]$ on $\Sigma\times D^2$ as before, we obtain:

\begin{prop}\label{gengrelinvtprop} For the \spinc structure $\s_k$ on $\Sigma\times D^2$ with $c_1(\s_k)$ Poincar\'e dual to $2k[pt\times D^2]$, the relative invariant $\Psi_{\Sigma\times D^2,\eta, \s_k}$ is a linear map $\A(\Sigma\times D^2)\to HF^-_\Dot(\Sigma\times S^1, \s_k; \L(t))$ whose value $\Psi_{\Sigma\times D^2,\eta, \s_k}(1)$ lies in the summand of maximal height. Furthermore, there is a natural identification $HF^-_{top}(\Sigma\times S^1,\s_k;\L(t)) \cong \L(t)$ such that
\[
\Psi_{\Sigma\times D^2,\eta, \s_k}(1) = 1
\]
up to multiplication by $\pm t^n$. More generally, if $\alpha\in \A(\Sigma\times D^2)$ then $\Psi_{\Sigma\times D^2, \eta, \s_k}(\alpha ) = \alpha. 1$, where the right hand side makes use of the action of $\Lambda^*H_1(\Sigma\times S^1;\zee)\otimes \zee[U]$ on $HF^-_\Dot(\Sigma\times S^1,\s_k;\L(t))$.

Let $Z: S^3\to \Sigma\times S^1$ be a cobordism and $\eta\in H^2(Z;\arr)$ a class restricting to $\Sigma\times S^1$ as a positive multiple of the Poincar\'e dual of $pt\times S^1$. Let $\s$ be a \spinc structure on $Z$ restricting to the \spinc structure $\s_k$ on $\Sigma\times S^1$, where $|k| = g-1$. Then $\Psi_{Z,\eta,\s}$ takes values in the Novikov ring $\K(Z,\eta)$.

Finally, let $M$ be the closed manifold obtained by gluing $\Sigma\times D^2$ to $Z$, and filling in the other boundary component of $Z$ by a 4-ball. Extend $\eta$ across $\Sigma\times D^2$ and $B^4$ in the unique way to give a class $\eta\in H^2(M;\arr)$. For a \spinc structure $\s$ with $c_1(\s)|_{\Sigma\times D^2}$ Poincar\'e dual to $\pm(2g-2)[pt\times D^2]$, the value of the perturbed {\OS} invariant $\O_{M,\Sigma\times S^1,\eta,\s}$ on $\alpha_1\otimes \alpha_2 \in \A(Z)\otimes \A(T^2\times D^2)$ is given by
\[
\O_{M,\Sigma\times S^1,\eta,\s}(\alpha_1\otimes \alpha_2) = \langle \tau^{-1}\Psi_{Z,\eta,\s}(\alpha_1), \Psi_{\Sigma\times D^2, \eta,\s}(\alpha_2)\rangle =\left\{\begin{array}{ll} \rho(\Psi_{Z, \eta,\s}(\alpha_1)) & \mbox{if $\alpha_2 = 1$} \\ 0 & \mbox{otherwise,}\end{array}\right.
\]
where $\rho$ is the natural map $\K(Z,\eta)\to \L(t)$ induced by the projection $K(Z)\to K(M,\Sigma\times S^1)\cong \zee$.
\end{prop}

\begin{proof} We have seen that the homomorphism $F^+_{\Sigma\times D^2,\s_k}: HF^+(\Sigma\times D^2, \s_k; L(t)) \to HF^+(S^3; L(t))$ can be identified with the projection of $X(g,d)\otimes L(t)$ onto the summand having minimal height. Combining this with the preceding lemmas, we have that for $x\in HF^+(\Sigma\times S^1,\s_k; \zee[K(Z)])$,
\[
 \langle x, F^-_{\Sigma\times D^2}(\Theta^-)\rangle = \langle F^+(x), \Theta^-\rangle =\left\{\begin{array}{ll} \rho(x) & \mbox{if $\rho(x)$ has minimal height} \\ 0 & \mbox{otherwise,}\end{array}\right.
\]
The analog of diagram \eqref{twistpairdiag} with $\Sigma\times S^1$ in place of $T^3$ implies that $\Psi_{\Sigma\times D^2,\eta,\s_k}(1) = F^-_{\Sigma\times D^2, \eta,\s_k}(\Theta^-)$ is a class pairing with $x$ to give $\rho(x)$ when the latter has minimal height, and $0$ otherwise. It is not hard to see that the pairing between $HF^+$ and $HF^-$ can be nontrivial only on elements of complementary height, and we have seen that the summands in minimal and maximal height in the Floer homology are $HF^+_{bot}(\Sigma\times S^1, \s_k; \L(t))\cong \L(t)$ and $HF^-_{top}(\Sigma\times S^1,\s_k;\L(t))\cong \L(t)$. Hence the $\L(t)$-valued pairing between these groups must be given by multiplication of Laurent series (up to a unit in $L(t)$, since the pairing is induced from the unperturbed situation). The above can therefore be interpreted as the statement that $\Psi_{\Sigma\times D^2, \eta,\s_k}(1) = 1$.

Since $\A(\Sigma\times S^1)\to \A(\Sigma\times D^2)$ is surjective, the statement $\Psi_{\Sigma\times D^2,\eta,\s_k}(\alpha) = \alpha. 1$ follows from the naturality of cobordism-induced homomorphisms under the action of $H_1$.

The remaining statements follow from the preceding lemma.
\end{proof}

The expression of the relative invariant $\Psi_{Z,\eta,\s}$ where $c_1(\s)|_{\Sigma\times S^1}= PD(2k[pt\times S^1])$ with $|k|< g-1$ is somewhat more complicated; in principle it may take as a value any element of the Floer homology $HF^-_\Dot(\Sigma\times S^1, \s_k; \K(Z,\eta))$, which in the case at hand is not a cyclic module. However, it is still possible to express the relative invariant for $Z$ (after applying $\rho$) in terms of the absolute invariants for $M$. To do so, we make use of the structure of the Floer homology of $\Sigma\times S^1$ deduced previously.

First, recall that the graded group $X(g,d)$ is equipped with a ``standard'' action of $\Lambda^*H_1(\Sigma)\otimes \zee[U]$. In fact, suppose $\B_{g,d} = \{\beta\}\subset X(g,d)$ is a basis for $X(g,d)$ as a free abelian group, with each $\beta$ a homogeneous element. Then it is easy to see that there is a uniquely determined collection of elements $\{\tbeta\}\subset \Lambda^*H_1(\Sigma)\otimes\zee[U]$ (lying in degrees $\leq 2d$) with the property that $\tbeta\cap \gamma = \delta_{\beta\gamma}\cdot 1$ for $\beta,\gamma\in \B_{g,d}$, where $\cap$ is the standard action, $\delta_{\beta\gamma}$ is the Kronecker delta, and $1$ denotes a fixed generator in lowest degree for $X(g,d)$.

We have seen in Theorem \ref{groups1} that there is an isomorphism $HF^-_\Dot(\Sigma\times S^1, \s_k; \L(t))\cong X(g,d)\otimes \L(t)$, and furthermore that the action of $\Lambda^*H_1(\Sigma\times S^1)\otimes\zee[U]$ agrees with the $\L(t)$-linear extension of the standard action to leading order in $t$ (and that the class $pt\times S^1$ acts trivially). Suppose first that $k\neq 0$, so that the variable $t$ carries a nonzero degree. The pairing $\langle \cdot, \cdot\rangle$ between $HF^+_\Dot$ and $HF^-_\Dot$ is nontrivial only on elements of complementary height, and is induced from the untwisted pairing since in our situation the twisting is trivial. Writing $\Xi$ for the generator in highest degree given by $\Psi_{\Sigma\times D^2, \eta,\s_k}(1)$, it follows for dimensional reasons that if the degree of $\tbeta$ is equal to the height of $x \in X(g,d)$, then
\[
\langle \tau^{-1}(\tbeta. x) , \Xi\rangle = \langle \tau^{-1}(\tbeta \cap x), \Xi\rangle
\]
(identifying $x$ with $x\otimes 1\in X(g,d)\otimes \L(t) = HF^-(\Sigma\times S^1,\s_k;\L(t))$).

On the other hand, if $k = 0$ then corrections to the $H_1$ action $\tbeta.x$ appear in the same degree as $\tbeta\cap x$. Since the action is standard to leading order in $t$, however, we can say that for basis elements $\beta,\gamma$ as previously the action satisfies $\tbeta. \gamma = \delta_{\beta\gamma}\, u_\beta(t)\cdot 1$, where $u_\beta\in \L(t)$ is  
monic (in the sense that its constant coefficient equals 1, c.f. Remark \ref{H1-action})  and hence a unit in $\L(t)$ (here we observe that if $\tbeta\cap \gamma = 0$ then also $\tbeta. \gamma = 0$). When $k = 0$, then, the above becomes
\[
\langle \tau^{-1}(\tbeta. x) , \Xi\rangle = u_\beta \langle \tau^{-1}(\tbeta\cap x),\Xi\rangle.
\] 

Now suppose $\xi$ is an element of a given height in $HF^-_\Dot(\Sigma\times S^1, \s_k;\L(t))$. Then fixing the basis $\B_{g,d} = \{\beta\}$ as previously, we can express $\xi$ in terms of $\{\beta\}$ by
\[
\xi = \left\{\begin{array}{ll} \ds\sum_\beta \langle \tau^{-1}(\tbeta. \xi), \Xi\rangle\cdot \beta & \mbox{if $k \neq 0$}\vspace{1ex}\\
\ds\sum_\beta \langle \tau^{-1}(\tbeta. \xi), \Xi\rangle\cdot u_\beta^{-1}\, \beta & \mbox{if $k = 0$},\end{array}\right.
\]
where the sum is over basis elements $\beta$ having the degree equal to the height of $\xi$.

Applying this idea to the case $\xi = \rho(\Psi_{Z,\eta,\s}(\alpha_1))$ leads to the following.

\begin{prop}\label{gengrelinvtprop} Let $M$ be a closed $4$-manifold containing an embedded surface $\Sigma$ of genus $g>1$ and trivial normal bundle, and $\eta\in H^2(M;\arr)$ a class with $\int_\Sigma \eta >0$. Write $Z = M\setminus (\Sigma\times D^2)$, and let $\s$ be a \spinc structure on $M $ restricting to $\s_k$ on $\Sigma\times S^1$. 

If $0<|k|<g-1$, then for an element $\alpha\in \A(Z)$, the reduced relative invariant $\rho(\Psi_{Z,\eta,\s}(\alpha))$ is given in terms of a basis $\{\beta\}$ for $HF^-_\Dot(\Sigma\times S^1,\s_k;\L(t))$ by
\[
\rho(\Psi_{Z,\eta,\s}(\alpha)) = \sum_\beta \O_{M,\Sigma\times S^1,\eta,\s}(\alpha\otimes \tbeta)\cdot \beta,
\]
where $\{\tbeta\}\subset \A(\Sigma\times D^2)$ are elements dual in the above sense to the basis $\{\beta\}$. If $k = 0$, then 
\[
\rho(\Psi_{Z,\eta,\s}(\alpha)) = \sum_\beta \O_{M,\Sigma\times S^1,\eta,\s}(\alpha\otimes \tbeta)\cdot u_\beta^{-1}\,\beta,
\]
where $u_\beta\in \L(t)$ are units depending only on the basis $\{\beta\}$.
\end{prop}

In the expressions above, $\alpha\otimes \tbeta$ is shorthand for the image of that element under the natural map $\A(Z)\otimes \A(\Sigma\times D^2)\to \A(M)$.

\begin{proof} When $k\neq 0$, we expand $\rho(\Psi_{Z,\eta,\s}(\alpha))$ in the basis $\beta$ as indicated previously:
\begin{eqnarray*}
\rho(\Psi_{Z,\eta,\s}(\alpha)) &=& \sum_\beta \langle\tau^{-1}(\tbeta.\rho(\Psi_{Z,\eta,\s}(\alpha))), \,\Xi\rangle\cdot \beta \\
&=&  \sum_\beta(-1)^{\deg(\tbeta)}\langle\tau^{-1}(\rho(\Psi_{Z,\eta,\s}(\alpha))), \,\tbeta.\Xi\rangle\cdot \beta\\
&=&\sum_\beta(-1)^{\deg(\tbeta)}\langle\tau^{-1}(\rho(\Psi_{Z,\eta,\s}(\alpha))), \,\Psi_{\Sigma\times D^2,\eta,\s_k}(\tbeta)\rangle\cdot \beta \\
&=& \sum_{\beta} \O_{M,\Sigma\times S^1,\eta,\s}(\alpha\otimes \tbeta)\cdot \beta,
\end{eqnarray*}
up to an overall sign and translation by a power of $t$, where we sum over those $\beta$ whose degree is equal to the height of $\rho(\Psi_{Z,\eta,\s}(\alpha))$. The case $k = 0$ is identical except for the introduction of the elements $u_\beta^{-1}$.
\end{proof}

\subsection{Fiber sum formulas} \label{proofsec}

A minor technicality in deducing the formulae in the introduction is the presence of the orientation-reversing gluing map in the fiber sum construction. We will have occasion to refer to the map in Floer homology induced by this diffeomorphism, so we make a few basic observations. 

First, if $f: Y\to Y$ is an orientation-preserving diffeomorphism of a 3-manifold, we can construct the mapping cylinder $C_f = (Y\times [0,1]) \cup_f (Y\times \{1\})$ in the usual way, which we can view as a smooth cobordism $Y\to Y$ of oriented manifolds. The action of $f$ on Floer homology is by definition the homomorphism $F_{C_f}^\circ$ in Floer homology induced by $C_f$. We will normally write this action as $f_*$. 

It is easy to see from basic properties of the cobordism maps that  if $h\in H_1(Y)$ then for $x\in HF^\circ(Y)$ we have $f_*(h.x) = f_*(h).f_*(x)$. In twisted coefficients, there is an isomorphism $f_* = (f^{-1})^*: \zee[H^1(Y;\zee)]\to \zee[H^1(Y;\zee)]$, and for $\alpha\in \zee[H^1(Y;\zee)]$ we have $f_*(\alpha x) = f_*(\alpha)f_*(x)$. A similar statement holds in the perturbed case, if $Y$ is equipped with a class $\eta_1\in H^2(Y;\arr)$ and we take $\eta_2 = (f^{-1})^*\eta_1$.

In the case of a fiber sum, we are given two closed 4-manifolds $M_1$, $M_2$ with embedded surfaces $\Sigma_1$, $\Sigma_2$ of genus $g$ and square 0. Write $Z_i = M_i\setminus (\Sigma_i\times D^2)$, so that $\partial Z_i = \Sigma\times S^1$. Then we choose an orientation-preserving diffeomorphism between $\Sigma_1$ and $\Sigma_2$, extending it to $\Sigma_i\times S^1$ by conjugation in the $S^1$ factor. The result is an orientation-reversing diffeomorphism $f: \partial Z_2\to \partial Z_1$, and the fiber sum is defined to be $X = Z_1\cup_f Z_2$. To make the gluing $f$ more explicit, replace $Z_2$ by $W_2 = Z_2 \cup_{id} C_f$. Then $X = Z_1\cup_{id} W_2$, and the relative invariants of $W_2$ and $Z_2$ are related by $\Psi_{W_2} = f_* \Psi_{Z_2}$, according to the composition law. Thus both $\Psi_{Z_1}$ and $\Psi_{Z_2}$ naturally take values in $HF^-_\Dot(\Sigma\times S^1)$ (with appropriate coefficients), while $\Psi_{W_2}$ takes values in $HF^-_\Dot(-\Sigma\times S^1)$.

Note that in certain situations, the above observations are sufficient to determine the action of $f_*$. For example, if the genus of $\Sigma$ is 1, then the reduced part of $HF^-(\Sigma\times S^1;L(t))$ is isomorphic to $L(t)$, where $t$ is Poincar\'e dual to the torus $\Sigma$. Since the action of $f$ in cohomology reverses the sign of the latter class, linearity of the induced map in Floer homology forces $f_*: HF^-_{red}(\Sigma\times S^1; L(t)) \to HF^-_{red}(-\Sigma\times S^1;L(t))$ to be the conjugation map $L(t)\to \overline{L(t)}$, up to multiplication by $\pm t^n$. Hence, the same conclusion follows in perturbed Floer homology, using a class $\eta\in H^2(\Sigma\times S^1;\arr)$ fixed by $f^*$, e.g., the Poincar\'e dual to $[S^1]$. A similar conclusion holds when considering the action of the gluing map in higher genus, if we restrict attention to the highest (or lowest) nontrivial heights in the perturbed Floer homology.

The fiber sum formula in the genus 1 case is as follows.

\begin{theorem}\label{g1fibsumthm} Let $X = M_1\#_{T_1 = T_2} M_2$ be the fiber sum of two 4-manifolds $M_1$, $M_2$ along tori $T_1$, $T_2$ of square 0. Assume that there exist classes $\eta_i\in H^2(M_i;\arr)$, $i = 1,2$, such that the restrictions of $\eta_i$ to $T_i\times S^1\subset M_i$ correspond under the gluing diffeomorphism $f: T_2\times S^1\to T_1\times S^1$, and assume that $\int_{T_i}\eta_i > 0$. Let $\eta\in H^2(X;\arr)$ be a class whose restrictions to $Z_i = M_i\setminus(T_i\times D^2)$ agree with those of $\eta_i$, and choose \spinc structures $\s_i\in Spin^c(M_i)$, $\s\in Spin^c(X)$ whose restrictions correspond similarly. Then for any $\alpha\in \A(X)$, the image of $\alpha_1\otimes \alpha_2$ under the map $\A(Z_1)\otimes \A(Z_2)\to \A(X)$, we have
\[
\rho(\O_{X,T\times S^1,\eta,\s}(\alpha)) = (t^{1/2} - t^{-1/2})^2\,\O_{M_1,T_1\times S^1,\eta_1,\s_1}(\alpha_1)\cdot\O_{M_2,T_2\times S^1,\eta_2,\s_2}(\alpha_2)
\]
up to multiplication by $\pm t^n$. 
\end{theorem}

Note that the closed invariants $\O_{M_i,T_i\times S^1,\eta_i,\s_i}$ each take values in $\L(t)$, where $t$ is the appropriate generator of $K(M_i, T_i\times S^1)$, and the multiplication takes place in that Laurent series ring.

\begin{proof} By definition,
\[
\O_{X,T\times S^1,\eta,\s}(\alpha) = \langle \tau^{-1}(\Psi_{Z_1,\eta_1,\s_1}(\alpha_1)), \Psi_{W_2, \eta_2,\s_2}(\alpha_2)\rangle,
\]
where $W_2 = Z_2\cup C_f$ as in the remarks above. Applying $\rho$ to each side, we see
\[
\rho(\O_{X,T\times S^1,\eta,\s}(\alpha)) = \langle \tau^{-1}\rho(\Psi_{Z_1,\eta_1,\s_1}(\alpha_1)), f_* \rho(\Psi_{Z_2,\eta_2,\s_2}(\alpha_2))\rangle.
\]
Since $ f_* \rho(\Psi_{Z_2,\eta_2,\s_2}(\alpha_2)) = \overline{\rho(\Psi_{Z_2,\eta_2,\s_2}(\alpha_2))}$, the theorem follows quickly from this, Proposition \ref{g1relinvtprop}, and anti-linearity of $\langle\cdot,\cdot\rangle$. Observe that the pairing in Floer homology of $T\times S^1$ can only be multiplication, up to $\pm t^n$, since the modules are cyclic.
\end{proof}

The higher-genus case is similarly easy, after some preparatory remarks. Recall that given a (homogeneous) basis $\{\beta_i\}$ for $X(g,d)$, we obtain a ``dual'' collection $\{\tbeta_i\}$ of elements of $\A(\Sigma) = \Lambda^*H_1(\Sigma)\otimes \zee[U]$. This dual basis is defined by the condition that $\tbeta_i\cap \beta_j = \delta_{ij}$, where $\cap$ is the standard action of $\A(\Sigma)$ on $X(g,d)$, and it satisfies $\langle \tau^{-1}(\tbeta_j.\beta_i), \Xi\rangle = \delta_{ij} u_i$, where $\Xi$ is the usual topmost generator of $HF^-(-\Sigma\times S^1,\s_k;\L(t))$, $\tbeta_j. \beta_i$ denotes the action of $\A(\Sigma)$ on Floer homology, and $u_i$ is a unit in $\L(t)$ that equals 1 unless $k = 0$. Furthermore, the basis $\{\tbeta_i\}$ is unique if we specify that it is contained in the subgroup $\tX(g,d) = \bigoplus_{i=0}^d \Lambda^iH_1(\Sigma)\otimes \zee[U]/U^{d-i+1} \subset \A(\Sigma)$.

If $\{\tbeta_i\}$ is the ``Kronecker dual'' basis, we can find a ``Poincar\'e dual'' basis $\{\beta_i^\circ\}$ for $X(g,d)$, namely $\beta_i^\circ = \tbeta_i\cap (1\otimes U^{-d})$, where we think of $1\otimes U^{-d}$ as a topmost generator for the Floer homology $X(g,d)\otimes \L(t) = HF^-_\Dot(\Sigma\times S^1, \s_k; \L(t))$. We could also say that $\beta_i^\circ$ is the leading order part (in $t$) of $\tbeta_i . \Xi$, except that in our conventions, $\Xi$ is a generator for $HF^-_\Dot(-\Sigma\times S^1, \s_k;\L(t))$.

Associated to the basis $\{\beta_i^\circ\}$, of course, there is a dual $\{\tbeta_i^\circ\}$, generating a subset of $\A(\Sigma)$. This set satisfies $\tbeta^\circ_i\cap \beta_j^\circ = \tbeta_i^\circ\cap \tbeta_j\cap (1\otimes U^{-d}) = \delta_{ij}$.

With these conventions in mind, we have the following.

\begin{theorem}\label{gengfibsumthm} Let $X = M_1\#_{\Sigma_1= \Sigma_2} M_2$ be the fiber sum of two 4-manifolds $M_1$, $M_2$ along surfaces $\Sigma_1$, $\Sigma_2$ of genus $g$ and square 0. Let $\eta_1$, $\eta_2$, $\eta$ be 2-dimensional cohomology classes satisfying conditions analogous to those in the previous theorem, and choose \spinc structures $\s_1$, $\s_2$, and $\s$ restricting compatibly as before. If the Chern classes of each \spinc structure restrict to $\Sigma\times S^1$ as a class other than $2k\,PD[S^1]$ with $|k|\leq g-1$ then the {\OS} invariants of all manifolds involved vanish. Otherwise, writing $f$ for the gluing map $\Sigma_2\times S^1\to \Sigma_1\times S^1$, we have
\[
\rho(\O_{X,\Sigma\times S^1,\eta,\s}(\alpha)) = \sum_{\beta} \O_{M_1,\Sigma_1\times S^1,\eta_1,\s_1}(\alpha_1\otimes \tbeta)\cdot \O_{M_2,\Sigma_2\times S^1,\eta_2,\s_2}(\alpha_2\otimes f_*^{-1}(\tbeta^\circ))\cdot u_\beta
\]
up to multiplication by $\pm t^n$. Here $\{\beta\}$ is a basis for $HF^-_\Dot(\Sigma\times S^1,\s_k;\L(t))$ associated to a basis for $X(g,d)$, $d = g-1-|k|$, and $\{\tbeta\}$ and $\{\tbeta^\circ\}$ are the dual elements of $\A(\Sigma)$ described above. The elements $u_\beta\in \L(t)$ are units that are equal to 1 unless $k = 0$.
\end{theorem}

\begin{proof} As in the previous theorem,
\[
\rho(\O_{X,\Sigma\times S^1,\eta,\s}(\alpha)) = \langle \tau^{-1}\rho(\Psi_{Z_1,\eta_1, \s_1}(\alpha_1)), f_*\rho(\Psi_{Z_2,\eta_2,\s_2}(\alpha_2))\rangle.
\]
Applying Proposition \ref{gengrelinvtprop}, this is
\[
\langle \tau^{-1} \sum_i \O_{M_1,\Sigma_1\times S^1,\eta_1,\s_1}(\alpha_1\otimes \tbeta_i) \cdot \beta_i,\, f_*\sum_j\O_{M_2,\Sigma_2\times S^1,\eta_2,\s_2}(\alpha_2\otimes\tgamma_j)\cdot \gamma_j\rangle 
\]
for bases $\{\beta_i\}$ and $\{\gamma_j\}$ whose relationship will be determined momentarily. As before, $f_*$ is conjugate-linear in $\L(t)$ and the pairing is also conjugate-linear in the second variable. Hence the above is equal to
\[
\sum_{i,j} \O_{M_1,\Sigma_1\times S^1,\eta_1,\s_1}(\alpha_1\otimes \tbeta_i)\,\O_{M_2,\Sigma_2\times S^1,\eta_2,\s_2}(\alpha_2\otimes \tgamma_j) \langle\tau^{-1}(\beta_i),f_*(\gamma_j)\rangle.
\]
Choose the basis $\{\gamma_j\}$ by setting $\gamma_j = f_*^{-1}(\beta_j^\circ)$; then it is easy to see that $\langle \tau^{-1}(\beta_i), f_*(\gamma_j)\rangle = \delta_{ij}\,u_{\beta_i}$, and $\tgamma_j = f_*^{-1}(\tbeta_j^\circ)$. The result follows immediately.
\end{proof}

Suppose now that each of $M_1$, $M_2$, and $X$ have $b^+\geq 2$, so that Theorem \ref{pertproductthm} applies to identify the perturbed invariants $\O$ with the usual {\OS} invariants $\Phi$. Assume also that $\s\in Spin^c(X)$ restricts to a nonzero multiple of $PD[S^1]$, i.e., $k\neq 0$ in the theorem above. The coefficient change $\rho$ sums the coefficients of $\O_{X,\Sigma\times S^1}$ corresponding to \spinc structures differing by rim tori, so since $k \neq 0$, Theorem \ref{gengfibsumthm} translates to the equation
\[
\sum_n \Phi_{X,\s+nt}^{Rim}(\alpha)\, t^n = \sum_{\beta,n_1,n_2} \Phi_{M_1,\s+n_1t_1}(\alpha\otimes \tbeta)\,\Phi_{M_2,\s_2 + n_2t_2}(\alpha_2\otimes f^{-1}_*(\tbeta^\circ))\,t^{n_1+n_2},
\]
where $t_i$ is Poincar\'e dual to the class of $\Sigma_i$ in $M_i$, and $t$ is simultaneously the dual of $\Sigma$ in $X$ and the formal variable in $\L(t)$. The above holds after possibly a multiplication by a power of $t$; thus equating coefficients yields the formula
\[
\Phi_{X,\s}^{Rim}(\alpha) = \sum_{\scriptsize\begin{array}{cc} \beta \\ n_1 + n_2 = n_0\end{array}} \Phi_{M_1,\s+n_1t_1}(\alpha\otimes \tbeta)\,\Phi_{M_2,\s_2 + n_2t_2}(\alpha_2\otimes f^{-1}_*(\tbeta^\circ))
\]
for some fixed integer $n_0$. 

Now, if $\{\tbeta\}$ is a basis of homogeneous elements, it is not hard to see that $\{\tbeta^\circ\}$ are likewise homogeneous of complementary degree. Specifically, if $\deg(\tbeta) = m$ in $\tX(g,d)\subset \A(\Sigma)$ then $\deg(\tbeta^\circ) = 2d-m$ (as usual, $d = g-1-|k|$). Thus in the above formula, we have
\[
\deg(\alpha_1\otimes \tbeta) + \deg(\alpha_2\otimes f^{-1}_*(\tbeta^\circ)) = \deg(\alpha) + 2g-2-2|k|.
\]
On the other hand, if a \spinc 4-manifold $(N,\rr)$ has $\Phi_{N,\rr}(\xi) \neq 0$ for $\xi\in \A(N)$ then we must have $\deg(\xi) = d(\rr)$. Substituting this in the above and using $\sigma(X) = \sigma(M_1) + \sigma(M_2)$ and $e(X) = e(M_1) + e(M_2) + 4g-4$ gives
\begin{equation}\label{c1squarecond}
c_1^2(\s) = c_1^2(\s_1 + n_1t_1) + c_1^2(\s_2 + n_2t_2) + 8|k|.
\end{equation}

When $k = 0$, of course, changing $\s_i$ by multiples of $t_i$ does not affect the self-intersection so that \eqref{c1squarecond} holds in that case as well.

This observation motivates the following ``patching'' construction producing elements of $H^2(X;\zee)$ (modulo rim tori) from certain pairs of elements in $H^2(M_i;\zee)$. We find it easiest to describe this construction in homology rather than cohomology; the cohomological version is obtained by Poincar\'e duality. Suppose, then, that $x_1\in H_2(M_1)$ and $x_2\in H_2(M_2)$ are integral homology classes, represented by embedded surfaces also denoted $x_1$, $x_2$, and assume that $x_i . \Sigma_i = m$ for $i = 1,2$. Let $\rho: H_2(M_i)\to H_2(Z_i, \partial Z_i)$ denote the composition of the natural map $H_2(M_i)\to H_2(M_i, \Sigma_i\times D^2)$ followed by the excision isomorphism of the latter group with $H_2(Z_i,\partial Z_i)$ where $Z_i = M_i\setminus int ( \Sigma_i\times D^2)$. Consider the long exact sequence for $X = Z_1\cup_\partial Z_2$:
\[
\cdots\to H_2(\Sigma\times S^1)\to H_2(X)\to H_2(Z_1,\partial Z_1)\oplus H_2(Z_2,\partial Z_2) \to H_1(\Sigma\times S^1)\to \cdots
\]
The condition on $x_i . \Sigma_i$ and the fact that the $\rho(x_i)$ are restrictions of classes on the closed manifolds $M_i$ imply that there exists a lift $x\in H_2(X)$ of $(\rho(x_1), \rho(x_2))$, uniquely determined up to the image of $H_2(\Sigma\times S^1)$. 

Choose the surfaces $x_i$ to intersect $\Sigma_i\times D^2$ in a collection of normal disks; at the expense of increasing the genus of the $x_i$ we may assume that there are exactly $|m|$ such disks. Then removing $\Sigma_i\times D^2$ from each of $M_1$, $M_2$ and gluing we can obtain a smooth surface representing the lifted class $x$. It is clear that $x$ has $x. \Sigma = m$, and furthermore by using pushoffs of the $x_i$ that are disjoint from the normal disks in $\Sigma_i\times D^2$ we see that the self-intersection of $x$ satisfies $x^2 = x_1^2 + x_2^2$. 

Now let $x_1 * x_2 = x + 2\varepsilon\Sigma$, where $\varepsilon$ is the sign of $m$. Then the self-intersection of $x_1*x_2$ is 
\[
(x_1* x_2)^2 = x_1^2 + x_2^2 + 4|m|,
\]
and moreover the class $x_1*x_2$ is determined by this condition up to addition of elements of $H_2(\Sigma\times S^1)/[\Sigma]$, in other words, up to rim tori.

The multiplication in Theorem \ref{gengthm1} is the Poincar\'e dual of this patching construction; the proof of that theorem is immediate from Theorem \ref{gengfibsumthm} and the remarks leading to \eqref{c1squarecond}. Theorem \ref{g1thm} follows similarly from Theorem \ref{g1fibsumthm}.

\section{Manifolds of simple type}\label{simptypesec}

Corollary \ref{simptypecor} is an easy consequence of the fiber sum formula. Indeed, if $M_1$ and $M_2$ have simple type, then the only contributions to the right hand side of \eqref{gengformula} are those in which $\alpha_1$, $\alpha_2$, $\beta$ and $\beta^\circ$ have degree zero. Hence $\alpha = \alpha_1\otimes \alpha_2$ also has degree 0, showing that $\rho(\O_{X,\s}(\alpha)) = 0$ unless $\deg(\alpha) = 0$, which is the first statement of the corollary.  Furthermore, since $\beta$ and $\beta^\circ$ have complementary degree in $\tX(g,d)$, their degrees can both be 0 only if $|k| = g-1$, which gives \eqref{intermedvanish}. 

In the case of a 4-manifold containing a torus of square 0, we have the following analog of a result of Morgan, Mrowka and Szab\'o \cite{MMS} in Seiberg-Witten theory. Recall that a 4-manifold $X$ containing a surface $\Sigma$ is said to have $\A(\Sigma)$-simple type if all {\OS} invariants of $X$ vanish on elements of $\A(X)$ lying in the ideal generated by $U$ and $H_1(\Sigma)$.

\begin{prop}\label{torusSTprop} Suppose $X$ is a closed 4-manifold with $b^+(X) \geq 2$ containing a torus $T\subset X$ of self-intersection 0 representing a class of infinite order in $H_2(X)$. Then $X$ has $\A(T)$-simple type.
\end{prop}

\begin{proof} By Theorem \ref{pertproductthm} we can write
\[
\O_{X,T^3,\eta,\s}(\alpha) = \langle \tau^{-1}\Psi_{Z,\eta,\s}(\alpha_1),\Psi_{T^2\times D^2,\eta,\s}(\alpha_2)\rangle,
\]
for suitable perturbation $\eta$, and we may assume that $\s|_{T^2\times D^2} = \s_0$, the torsion \spinc structure. If $\alpha = \alpha'\cdot \alpha_T$ is in the ideal generated by $\A(T)$, where $\alpha'\in A(X)$ and $\alpha_T\in \A(T)$, then we can take $\alpha_1 = \alpha'$ and $\alpha_2 = \alpha_T$. But $\Psi_{T^2\times D^2, \eta,\s_0}(\alpha_T) = \alpha_T.\Psi_{T^2\times D^2,\eta,\s_0}(1)$, and the relative invariant $\Psi_{T^2\times D^2,\eta,\s_0}(1)$ lies in the only nontrivial degree of $HF^-(T^3,\s_0;\L(t))$. Hence if $\deg{\alpha_T}>0$ we have $\alpha_T.\Psi_{T^2\times D^2,\eta,\s_0}(1) = 0$ and the result follows.
\end{proof}

Note that the proof applies also to 4-manifolds with $b^+(X) = 1$, if we consider only the perturbed invariant $\O_{X,T^3,\eta,\s}$ relative to the decomposition of $X$ along the boundary of a tubular neighborhood of the torus.

It seems likely that the statements of Corollary \ref{STcor} are true without the sum over \spinc structures differing by rim tori (at least when $b^+(X)\geq 2$). At the moment we can prove that stronger statement only under the apparently stronger assumption of {\it relative simple type}, defined here.

\begin{definition} A 4-manifold $Z$ with boundary $Y\cong \Sigma \times S^1$ is said to have {\em relative $\A(\Sigma)$ simple type} if all nonvanishing relative invariants $\Psi_{Z,\s,\eta}(\alpha)$ lie in the submodule of $HF^-_\Dot(\Sigma\times S^1,\s; \K(Z,\eta))$ of minimal height. We say $Z$ has {\em relative simple type} if furthermore any nonvanishing relative invariant $\Psi_{Z,\eta,\s}(\alpha)$ has $\deg(\alpha) = 0$.

 If $M$ is a closed 4-manifold containing an essential surface $\Sigma$ of square 0, we say $M$ has {\em strong $\A(\Sigma)$ simple type} if $Z = M \setminus (\Sigma\times D^2)$ has relative $\A(\Sigma)$ simple type.
\end{definition}

It is not hard to see that if $M$ has (ordinary) $\A(\Sigma)$ simple type, then the relative invariants of $Z = M \setminus (\Sigma\times D^2)$ have the property that $\rho(\Psi_{Z,\eta,\s}(\alpha)) = 0$ unless $\Psi_{Z,\eta,\s}(\alpha)$ lies in the minimal-height summand. The condition of strong $\A(\Sigma)$ simple type indicates that the corresponding statement holds before projecting out rim tori. 

\begin{theorem} Suppose $X = M_1\#_\Sigma M_2$ is a closed 4-manifold with $b^+(X)\geq 2$ that is obtained by fiber sum along an essential surface $\Sigma$ of genus $g \geq 1$ and square 0.
\begin{enumerate}
\item Assume that $M_1$ has strong $\A(\Sigma)$ simple type: then $X$ has $\A(\Sigma)$-simple type.
\item Assume furthermore that the complement $Z_2 = M_2\setminus (\Sigma\times D^2)$ has $b^+(Z_2)\geq 1$, and also that the subgroup of $H^2(X)$ generated by the Poincar\'e duals of rim tori is infinite. Then
\[
\mbox{$\Phi_{X,\s} = 0$ unless $|\langle c_1(\s),\Sigma\rangle| = 2g-2$.}
\]
\end{enumerate}
\end{theorem}

\begin{proof} 
For (1), suppose $X$ does not have $\A(\Sigma)$ simple type: that is, there exists $\beta\in \A(\Sigma)$ of nonzero degree and an element $\alpha\in \A(X)$ such that $\Phi_{X,\s}(\alpha\otimes\beta)\neq 0$ for some \spinc structure $\s$. Since $b^+(X)\geq 2$, we can calculate $\Phi_{X,\s}$ using the perturbed invariant, relative to the obvious decomposition of $X = Z_1 \cup_{\Sigma\times S^1} Z_2$ coming from the fiber sum. That is, for an appropriate choice of $\eta$, the nonvanishing invariant $\Phi_{X,\s}$ appears as a coefficient of the expression
\[
\O_{X,\Sigma\times S^1,\eta,\s}(\alpha\otimes\beta) = \langle \tau^{-1}\Psi_{Z_1,\eta,\s}(\alpha_1\otimes \beta), \Psi_{Z_2,\eta,\s}(\alpha_2)\rangle.
\]
The above is also equal to 
\[
\langle \tau^{-1}\Psi_{Z_1,\eta,\s}(\alpha_1), \Psi_{Z_2,\eta,\s}(\alpha_2\otimes \beta)\rangle,
\]
hence both $\Psi_{Z_1,\eta,\s}(\alpha_1)$ and $\Psi_{Z_1,\eta,\s}(\alpha_1\otimes \beta) = \beta. \Psi_{Z_1,\eta,\s}(\alpha_1)$ are nonvanishing. These are two nonvanishing relative invariants of $Z_1$ having different heights, hence $Z_1$ does not have relative $\A(\Sigma)$ simple type. (As in the torus case, this proof applies when $b^+(X) = 1$ if we restrict attention to the perturbed invariants calculated from $(\Sigma\times S^1,\eta)$.)

To prove (2) we choose an admissible cut $N$ for $X$ contained in $Z_2$, which is possible since $b^+(Z_2)\geq 1$. As in the proof of Theorem \ref{productthm}, write $X = Z_1\cup_{\Sigma\times S^1} W\cup_N V_2$, where $W \cup V_2 = Z_2$. Suppose that $\alpha\in \A(X)$ is the image of $\alpha_1\otimes \alpha_2\otimes \alpha_3\in \A(Z_1)\otimes \A(W)\otimes \A(V_2)$. Then for any $r = \delta h \in K(X,\Sigma\times S^1)$, we have from \eqref{transeqn}
\[
\Phi_{X, \s+r}(\alpha) = \langle\tau^{-1}\circ\Pi_{V_1}\circ F^-_{W,\s}(\alpha_2\otimes e^h\cdot F^-_{Z_1, \s}(\alpha_1)), F^-_{V_2,\s}(\alpha_3)\rangle,
\]
Suppose now that $h$ is dual to a rim torus in $\Sigma\times S^1$. Since $Z_1$ has relative $\A(\Sigma)$ simple type, we may assume $F^-_{Z_1,\s}(\alpha_1)$ lies in the lowest degree of $HF^-_\Dot(\Sigma\times S^1, \s_k; \K(Z_1,\eta))$, where $\s_k = \s|_{\Sigma\times S^1}$. We can assume $PD(c_1(\s_k)) = 2k[pt\times S^1]$ for $|k|\leq g-1$; suppose $|k| < g-1$. Then we have seen that the minimal-height part of $HF^-_\Dot(\Sigma\times S^1, \s_k; \K(Z_1,\eta))$ is isomorphic to $\L(t)$, on which classes dual to rim tori act trivially. Hence the above expression is equal to $\Psi_{X,\s}(\alpha)$---that is, $\Psi_{X,\s + r}(\alpha) = \Psi_{X,\s}(\alpha)$ for any $r$ in the subgroup of $H^2(X)$ generated by the duals of rim tori. Since the latter is an infinite group, we infer from the fact that only finitely many \spinc structures have nonvanishing {\OS} invariant that $\Psi_{X,\s+r} = \Psi_{X,\s} = 0$.
\end{proof}


\newpage

\begin{thebibliography}{99}
\bibitem{eisenbud} D. Eisenbud, {\it Commutative Algebra with a View Toward Algebraic Geometry}, Graduate Texts in Math. {\bf 150}, {\it Springer-Verlag, New York}, 1995.
\bibitem{GS} R. Gompf and A. Stipsicz, {\it 4-manifolds and Kirby calculus},  Graduate Stud. in Math. {\bf 20}, Amer. Math. Soc., Providence, RI, 1999.

\bibitem{us} S. Jabuka and T. Mark, ``Heegaard Floer homology of a surface times a circle,'' preprint math.GT/0502328.
\bibitem{KMbook} P. Kronheimer and T. Mrowka, {\it Monopoles and three-manifolds,} book in preparation.
\bibitem{LL} T. J. Li and A. Liu, ``General wall crossing formula,'' {\it Math. Res. Lett.} {\bf 2} (1995), 797--810.
\bibitem{macdonald}  I. G. Macdonald, ``{Symmetric products of an
algebraic curve}'', Topology (1962), 319--343.
\bibitem{MMS} J. Morgan, T. Mrowka and Z. Szab\'o, ``Product formulas along $T^3$ for Seiberg-Witten invariants,'' {\it Math. Res. Lett.} {\bf 4} (1997), no. 6, 915--929.
\bibitem{MST} J. Morgan, Z. Szab\'o and C. H. Taubes, ``A product formula for the Seiberg-Witten invariants and the generalized Thom conjecture,'' {\it J. Differential Geom.} {\bf 44} (1996), no. 4, 706-788.

\bibitem{OS1} P. Ozsv\'ath and Z. Szab\'o, ``Holomorphic disks and topological invariants for closed three-manifolds,'' {\it Ann. of Math.} {\bf 159} (2004), no. 3, 1027--1158.
\bibitem{OS2} P. Ozsv\'ath and Z. Szab\'o, ``Holomorphic disks and three-manifold invariants: properties and applications,'' {\it Ann. of Math.} {\bf 159} (2004), no. 3, 1159--1245.
\bibitem{OS3} P. Ozsv\'ath and Z. Szab\'o, ``Holomorphic triangles and invariants for smooth four-manifolds,'' {\it Adv. Math.} {\bf 202} (2006), no. 2, 326--400.
\bibitem{OS4} P. Ozsv\'ath and Z. Szab\'o, ``Absolutely graded Floer homologies and intersection forms for four-manifolds with boundary,'' {\it Adv. Math.} {\bf 173} (2003), no. 2, 179--261.
\bibitem{OS6} P. Ozsv\'ath and Z. Szab\'o, ``On the {F}loer homology of plumbed three-manifolds,'' {\it Geom. Topol.} {\bf 7} (2003), 185--224.
\bibitem{OSknot} P. Ozsv\'ath and Z. Szab\'o, ``Holomorphic disks and knot invariants,'' {\it Adv. Math.} {\bf 186} (2004), no. 1, 58--116.
\bibitem{OSsymp} P. Ozsv\'ath and Z. Szab\'o,
``Holomorphic triangle invariants and the topology of symplectic
four-manifolds,'' {\it Duke Math. J.} {\bf 121} (2004), no. 1, 1--34.
\bibitem{OSsurg} P. Ozsv\'ath and Z. Szab\'o, ``Knot {F}loer homology and integer surgeries,'' preprint math.GT.0410300.
\bibitem{park} B. D. Park, ``A gluing formula for the Seiberg-Witten invariant along $T^3$,'' {\it Michigan Math. J.} {\bf 50} (2002), 593--611.
\bibitem{stern}R. Stern, ``Will we ever classify simply-connected smooth 4-manifolds?'' Lecture notes from 2004 Clay Institute Summer School on Floer homology, gauge theory, and low-dimensional topology. {\tt math.GT/0502164}
\bibitem{taubesST} C. H. Taubes, ``$SW=Gr$: counting curves and connections.'' In {\it Seiberg-Witten and Gromov invariants for symplectic 4-manifolds}, 275--401, First Int. Press Lect. Ser., 2, Int. Press, Somerville, MA, 2000.
\bibitem{taubes} C. H. Taubes, ``The Seiberg-Witten invariants and 4-manifolds with essential tori,'' {\it Geom. Topol.} {\bf 5} (2001), 441--519.
\bibitem{witten} E. Witten, ``Monopoles and four-manifolds,'' {\it Math. Res. Lett.} {\bf 1} (1994), no. 6, 769--796.
\end{thebibliography}
\end{document}